\documentclass[letterpaper, 11pt]{amsart}
\usepackage{amstext,amsmath,amssymb,amsfonts,orcidlink}%
\usepackage{amsthm}%
\usepackage{xcolor}%
\usepackage{manyfoot}%
\usepackage[notextcomp]{stix}
\usepackage[utf8]{inputenc}
\usepackage[italian, english]{babel}
\usepackage[T1]{fontenc} 
\usepackage{braket,mathtools}
\usepackage{hyperref,csquotes}
\usepackage[english,capitalize]{cleveref}
\usepackage[a4paper,twoside,top=0.8in, bottom=0.7in, left=0.7in, right=0.7in]{geometry}
\usepackage{tikz}
\usetikzlibrary{patterns}

\definecolor{newblue}{rgb}{0.2, 0.3, 0.85}
\hypersetup{colorlinks=true, linkcolor=newblue, citecolor=newblue, urlcolor  = newblue}

\usepackage[maxbibnames=99,backend=biber, sorting=nyt, doi=false, url=false, isbn=false, style=alphabetic]{biblatex}
\numberwithin{equation}{section}
\newcommand{\N}{\mathbb{N}}

\newcommand{\R}{\mathbb{R}}

\newcommand{\hh}{\mathcal{H}}
\renewcommand{\epsilon}{\varepsilon}
\renewcommand{\theta}{\vartheta}
\renewcommand{\rho}{\varrho}
\renewcommand{\phi}{\varphi}
\newcommand{\de}{\,{\rm d}}
\renewcommand{\d}{{\rm d}}

\newcommand{\st}{\ensuremath{\ :\ }} 
\newcommand{\mres}{\mathbin{\vrule height 1.6ex depth 0pt width
0.13ex\vrule height 0.13ex depth 0pt width 1.3ex}}

\theoremstyle{plain}
\newtheorem{theorem}{Theorem}[section]
\newtheorem{proposition}[theorem]{Proposition}
\newtheorem{lemma}[theorem]{Lemma}
\newtheorem{corollary}[theorem]{Corollary}

\theoremstyle{remark}

\newtheorem{remark}[theorem]{Remark}

\theoremstyle{definition}
\newtheorem{definition}[theorem]{Definition}

\begin{document}

\title[Quantitative isoperimetric inequalities for classical capillarity problems]{Quantitative isoperimetric inequalities\\ for classical capillarity problems}

\author{Giulio Pascale}
\address{Dipartimento di Matematica e Applicazioni "Renato Caccioppoli", Universit\'a degli Studi di Napoli "Federico II", via Cintia - Monte Sant'Angelo, 80126 Napoli, Italy} \email{giulio.pascale@unina.it \orcidlink{0000-0003-1680-3425}}
\author{Marco Pozzetta}
\address{Dipartimento di Matematica e Applicazioni "Renato Caccioppoli", Universit\'a degli Studi di Napoli "Federico II", via Cintia - Monte Sant'Angelo, 80126 Napoli, Italy} \email{marco.pozzetta@unina.it 
\orcidlink{0000-0002-2757-0826}}

\date{\today}

\begin{abstract}
We consider capillarity functionals which measure the perimeter of sets contained in a Euclidean half-space assigning a constant weight $\lambda \in (-1,1)$ to the portion of the boundary that touches the boundary of the half-space. Depending on $\lambda$, sets that minimize this capillarity perimeter among those with fixed volume are known to be suitable truncated balls lying on the boundary of the half-space.

We first give a new proof based on an ABP-type technique of the sharp isoperimetric inequality for this class of capillarity problems. Next we prove two quantitative versions of the inequality: a first sharp inequality estimates the Fraenkel asymmetry of a competitor with respect to the optimal bubbles in terms of the energy deficit; a second inequality estimates a notion of asymmetry for the part of the boundary of a competitor that touches the boundary of the half-space in terms of the energy deficit.

After a symmetrization procedure, the quantitative inequalities follow from a novel combination of a quantitative ABP method with a selection-type argument.
\end{abstract}

\maketitle

\vspace{-0.8cm}
\tableofcontents
\vspace{-0.8cm}

\section{Introduction}

The classical isoperimetric problem in the Euclidean space $\R^n$, for $n \ge 2$, aims at minimizing the $(n-1)$-dimensional area of boundaries of sets having fixed finite volume. More precisely, given $v > 0$, one aims to characterize minimizers to the problem
\begin{equation}\label{classical:isopprob}
\inf\left\{P(E) \st E \subset \R^n, |E| = v\right\},
\end{equation}
where $P(E)$ denotes the perimeter of $E$ (see \cref{sec:SetsFinitePerimeter}) and $|E|$ the Lebesgue measure of $E$. It is well-known that balls (uniquely) minimize \eqref{classical:isopprob}, cf. \cite{DeGiorgiIsoperimetrico} or \cite[Chapter 14]{MaggiBook}, and this is encoded in the classical isoperimetric inequality \begin{equation}\label{classical:isopineq}
P(E)\ge n \omega_n^{\frac{1}{n}}|E|^{\frac{n - 1}{n}}, \end{equation}
where $\omega_n$ denotes the measure of the unit ball in $\R^n$.
To prove a \emph{quantitative} version of \eqref{classical:isopineq} means to estimate the distance of a competitor from the set of minimizers in terms of the energy deficit of the competitor with respect to the infimum of the problem. The first quantitative isoperimetric inequality for \eqref{classical:isopprob} with sharp exponents was proved in \cite{FuscoMaggiPratelli}, and it reads \begin{equation}\label{classical:quantitative}
\alpha(E)^2 \le C(n) D(E),
\end{equation}
where $\alpha(E)$ and $D(E)$ are respectively the Fraenkel asymmetry and the isoperimetric deficit of $E$, i.e., \begin{equation*}
\alpha(E) := \inf \left\{\frac{E \Delta B(|E|, x)}{|E|} \st x \in \R^n\right\} 
\qquad
 D(E) := \frac{P(E) - P(B(|E|))}{P(B(|E|))}, \end{equation*}
where $B(v, x)$  denotes the ball in $\R^n$ with volume $v$ centered at $x$, for $v > 0$ and $x \in \R^n$, and $B(v):=B(v,0)$.
The inequality~\eqref{classical:quantitative} improves the previous non-sharp inequality proved in \cite{HallQuantitative}, after \cite{HallHaymanWeitsman, Fuglede}.
We also mention \cite{FigalliMaggiPratelli, CicaleseLeonardi, FigalliIndrei, FuscoJulin, IndreiNurbekyan} for further quantitative isoperimetric inequalities for possibly anisotropic perimeters, \cite{BogeleinDuzaarScheven, CavallettiMaggiMondino, ChodoshEngelsteinSpolaor} for quantitative isoperimetric inequalities on manifolds, and \cite{CianchiFuscoMaggiPratelli, BrascoDePhilippisRuffini, BarchiesiBrancoliniJulin, CGPROS, FuscoLaManna23} about weighted quantitative isoperimetric inequalities. 

\medskip

In this paper we prove quantitative isoperimetric inequalities for the following classical capillarity problems.
If $E$ is a measurable set in the half-space $\{x_n > 0\}\subset \R^n$ and $\lambda \in (-1,1)$, we define the weighted perimeter functional
\begin{equation*}
P_\lambda(E) := P(E, \{x_n>0\}) - \lambda \hh^{n - 1}(\partial^*E \cap \{x_n=0\} ),
\end{equation*}
where $\hh^{k}$, for $k \ge 0$, denotes the $k$-dimensional Hausdorff measure in $\R^n$, and $\partial^*E$ denotes the reduced boundary of $E$ (see \cref{sec:SetsFinitePerimeter}). Interpreting the perimeter as a measure of the surface tension of a liquid drop, the constant $\lambda$ basically represent the relative adhesion coefficient between a liquid drop and the solid walls of the container given by $\{x_n>0\}$.

If $v > 0$, we consider the isoperimetric capillarity problem \begin{equation}\label{weighted:isopprob} \inf\left\{P_\lambda(E) \st E \subset \{x_n>0\}, |E| = v\right\}. \end{equation}
Minimizers for~\eqref{weighted:isopprob}, below called isoperimetric sets for~\eqref{weighted:isopprob}, are given by suitably truncated balls lying on the boundary of the half-space. More precisely, if $B^\lambda = \{x \in B_1(0)\subset\R^n : \Braket{x, e_n} > \lambda\}$, and for $v>0$ we set
\[ B^\lambda(v) := \frac{v^{\frac{1}{n}}}{|B^\lambda|^{\frac1n}}(B^\lambda - \lambda e_n) \]
then minimizers for~\eqref{weighted:isopprob} are sets of the form
\begin{equation}\label{eq:OptimalBubbles}
	B^\lambda(v,x) := B^\lambda(v)  +x,
\end{equation}
for $x \in \{ x_n=0\}$, see \cite[Theorem 19.21]{MaggiBook} and \cref{fig:BolleOttime}.

\begin{figure}[h]
		\begin{center}
			\begin{tikzpicture}[scale=1.6]
			\draw[->](-1.5,0)--(1.5,0);
			\draw[->](0,-1.1)--(0,1.7);
			\draw [ultra thick,domain=30:150] plot ({cos(\x)}, {-0.5+sin(\x)});
			\path[font=\small]
			(0,-0.5)node{$\bullet$};
			\path[font=\small]
			(0,-0.5)node[right]{$-\lambda_1 e_n$};
			\path[font=\small]
			(0,1.6)node[right]{$x_n$};
			\filldraw[fill=white, pattern=north east lines]  (-0.85,0) -- (0.85,0) arc (30:150:1cm);
			\end{tikzpicture}
			\qquad
			\begin{tikzpicture}[scale=1.6]
			\filldraw[fill=gray!20]  (-0.86,-0.5) -- (0.86,-0.5) arc (-30:210:1cm);
			\draw[->](-1.5,0)--(1.5,0);
			\draw[->](0,-1.1)--(0,1.7);
			\draw [dashed,domain=210:330] plot ({cos(\x)}, {sin(\x)});
			\draw [ultra thick,domain=30:150] plot ({cos(\x)}, {sin(\x)});
			\draw[dotted](-1.4,0.5)--(1.4,0.5);
			\draw[dotted](-1.4,-0.5)--(1.4,-0.5);
			\path[font=\small]
			(1.25,0.45)node[above]{$x_n=\lambda_1$};
			\path[font=\small]
			(1.25,-0.45)node[below]{$x_n=\lambda_2$};
			\path[font=\small]
			(0,1.6)node[right]{$x_n$};
			\filldraw[fill=white, pattern=north east lines]  (-0.85,0.5) -- (0.85,0.5) arc (30:150:1cm);
			\end{tikzpicture}
			\qquad
			\begin{tikzpicture}[scale=1.6]
			\filldraw[fill=gray!20]  (-0.86,0) -- (0.86,0) arc (-30:210:1cm);
			\draw[->](-1.5,0)--(1.5,0);
			\draw[->](0,-1.1)--(0,1.7);
			\path[font=\small]
			(0,0.5)node{$\bullet$};
			\path[font=\small]
			(0,0.5)node[right]{$-\lambda_2 e_n$};
			\path[font=\small]
			(0,1.6)node[right]{$x_n$};
			\end{tikzpicture}
		\end{center}
\caption{Left: $B^{\lambda_1}(|B^{\lambda_1}|)$ for some $\lambda_1>0$. Middle: Diagonally striped set = $B^{\lambda_1}$, Gray set = $B^{\lambda_2}$, for some $\lambda_1>0, \lambda_2<0$. Right: $B^{\lambda_2}(|B^{\lambda_2}|)$ for some $\lambda_2<0$.}\label{fig:BolleOttime}
\end{figure}
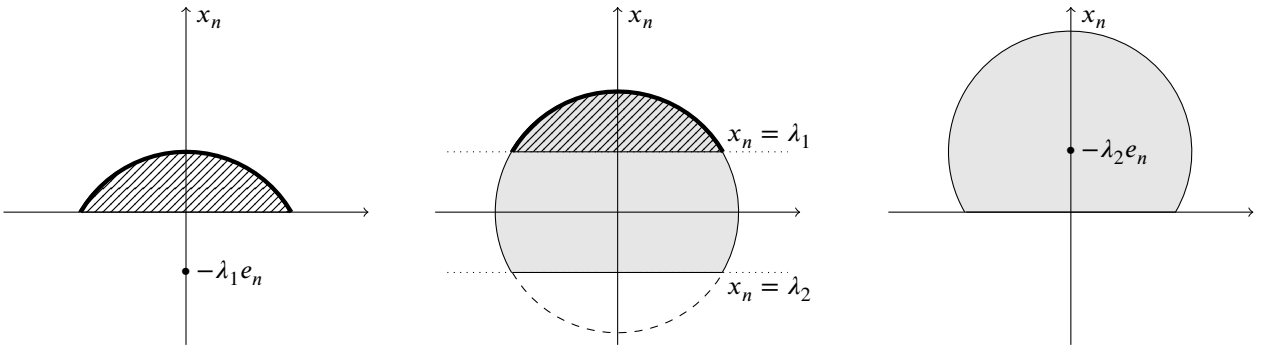

The first variational results regarding capillarity problems go back to works by Giusti, Gonzalez, Massari and Tamanini who established existence, symmetry and regularity results for the isotropic sessile drop problem, where an additional potential energy representing gravity is added to the minimization of $P_\lambda$ (see \cite{Gonzalez1976, Gonzalez1977, GonzalezTamanini, GonzalezMassariTamanini, Giusti1980, Giusti1981}; see also \cite{Finn1980} where uniqueness results for the symmetric sessile drop were established).
We refer to \cite{Finn1986} and \cite[Chapters 19, 20]{MaggiBook} for a more complete treatment regarding classical results. 
More recently, in \cite{Baer2015} the shape of liquid drops and crystals, resting on a horizontal surface and under the influence of gravity, are described in the anisotropic setting. The shape and the fine regularity of volume contrained minimizers of weighted perimeters like $P_\lambda$, where the weight on the interface touching the boundary of the container may be nonconstant and where an additional potential term is present, are addressed in \cite{MaggiMihaila, DePhilippisMaggi2015, ChodoshEdelenLi}. In \cite{DePhilippisMaggi2015} the perimeter functional measuring the area of the interface that does not touch the container is also possibly anisotropic.
Recently, the isoperimetric problem for the relative perimeter of sets contained in the complement of a convex set had been addressed in \cite{ChoeGhomiRitoreInequality}, where a sharp isoperimetric inequality is established, and in \cite{FuscoMorini}, where the rigidity of the inequality is addressed in the generality of measurable sets. Extensions of \cite{ChoeGhomiRitoreInequality} to higher codimension have been considered in \cite{LiuWangWeng, Krummel}.

\medskip

The minimality of sets $B^\lambda(v,x)$ for \eqref{weighted:isopprob} comes with an isoperimetric inequality for $P_\lambda$, see \cref{thm:IsopIneq} below. In order to prove a quantitative isoperimetric inequality for~\eqref{weighted:isopprob}, we define the corresponding Fraenkel asymmetry and isoperimetric deficit by setting
\begin{equation*} \alpha_\lambda(E) := \inf\left\{\frac{|E \Delta B^\lambda(v,x)|}{v} : x \in \{x_n=0\}\right\},
\qquad
 D_\lambda(E) := \frac{P_\lambda(E) - P_\lambda(B^\lambda(v))}{P_\lambda(B^\lambda(v))},
\end{equation*}
for any $E\subset\{x_n>0\}$ with volume $|E|=v$. The infimum defining the asymmetry is, in fact, a minimum.
 
The first main result of the paper is the following 

\begin{theorem}\label{thm:FinalQuantitativeInequality}
	Let $\lambda \in (-1, 1)$ and $n\in \N$ with $n\ge2$.
	There exists a constant $c_{\rm iso}=c_{\rm iso}(n, \lambda)>0$ such that for any measurable set $E\subset \R^n \cap \{x_n>0\}$ with finite measure there holds
	\begin{equation}\label{eq:FinalQuantitativeInequality}
	\alpha_\lambda(E)^2 \le c_{\rm iso}  D_\lambda(E).
	\end{equation}
\end{theorem}

As for the classical \eqref{classical:isopineq}, perturbing the boundary of an optimal bubble only inside the container $\{x_n>0\}$, it is possible to check that exponents in \eqref{eq:FinalQuantitativeInequality} are sharp.

\medskip

In the context of these capillarity problems it is also spontaneous to consider a notion of asymmetry for the part of the boundary of a set that touches the half-plane $\{x_n=0\}$.
For a measurable set $E\subset \{x_n>0\}$, we define
\begin{equation*}
    \beta_\lambda(E) := \inf \left\{ \frac{\hh^{n-1}\left(\partial^*E \cap\{x_n=0\} \,\, \Delta \,\,\partial^* B^\lambda(|E|,x) \cap \{x_n=0\}\right)}{\hh^{n-1}\left(\partial^* B^\lambda(|E|,x) \cap \{x_n=0\}\right)} 
    \st 
    x \in \{x_n=0\}
    \right\}.
\end{equation*}
The previous quantity measures the asymmetry of the set $\partial^*E \cap\{x_n=0\}$ with respect to $(n-1)$-dimensional balls in $\{x_n=0\}$ having volume equal to the one of the trace of the optimal bubble corresponding to the volume of $E$. We establish the following second quantitative isoperimetric inequality, that provides a quantitative estimate on $\beta_\lambda$.

\begin{theorem}\label{thm:QuantitativaBagnata}
    Let $\lambda \in (-1, 1)$ and $n\in \N$ with $n\ge2$. There exists a constant $c'_{\rm iso}=c'_{\rm iso}(n, \lambda)>0$ such that for any measurable set $E\subset \R^n \cap \{x_n>0\}$ with finite measure there holds
	\begin{equation}\label{eq:QuantitativaBagnata}
	\beta_\lambda(E) \le c'_{\rm iso} \max\left\{ D_\lambda(E) , D_\lambda(E)^{\frac{1}{2n}} \right\}.
	\end{equation}
\end{theorem}

\bigskip
\noindent\textbf{Strategy of the proof and comments.} 
Observing that, roughly speaking, the minimization problem \eqref{weighted:isopprob} is symmetric with respect to the first $n-1$ axes, it is possible to adapt arguments in the spirit of \cite{FuscoMaggiPratelli} to see that, in order to prove \cref{thm:FinalQuantitativeInequality}, it is sufficient to prove \eqref{eq:FinalQuantitativeInequality} in the class of Schwarz-symmetric sets, see \cref{prop:ScharzSymmSufficienti}. Here a set $E$ is said to be Schwarz-symmetric (\cref{def:SchwarzSymmetrization}) if the intersection $E \cap \{x_n=t\}$ is an $(n-1)$-dimensional ball in $\{x_n=t\}$ centered at $(0,t)$, for any $t$.
However, we point out that it seems not possible to push the strategy of \cite{FuscoMaggiPratelli} to the very end to prove \cref{thm:FinalQuantitativeInequality}. Indeed the arguments in \cite{FuscoMaggiPratelli} require to Schwarz-symmetrize a competitor with respect to a preferred axis depending on the competitor, while in our case it is only possible to symmetrize with respect to the $n$-th axis. One finds an analogous obstruction also in a possible adaptation of the proof via symmetrization revised in \cite{MaggiMethodsQuantitative08} (see also \cite{FuscoDispensa}); in \cite{MaggiMethodsQuantitative08} the quantitative isoperimetric inequality for Schwarz-symmetric sets is eventually obtained performing a quantitative version of Gromov's proof \cite{MilmanSchechtmanAppendixGromov} of the isoperimetric inequality, but again after having symmetrized a competitor with respect to a convenient axis.

\medskip

The proof of \eqref{eq:FinalQuantitativeInequality} in the class of Schwarz-symmetric sets is achieved here with a new combination of the so-called selection principle \cite{AcerbiFuscoMorini, CicaleseLeonardi} with an Alexandrov--Bakelman--Pucci-type technique in the spirit of \cite{CGPROS}.

In the recent \cite{CGPROS}, the authors prove sharp quantitative isoperimetric inequalities for a class of isoperimetric problems in cones where volume and perimeter are weighted in terms of a function satisfying suitable homogeneity and concavity properties.
The proof in \cite{CGPROS} stems from the fact that the isoperimetric inequality for the corresponding problem was proved in \cite{CabreRosOtonSerra} by a so-called ABP argument.
The methods that go under the name of ABP techniques were originally employed to derive regularity estimates for second order elliptic equations \cite[Chapter 9]{GilbargTrudinger} and they were applied to give a new direct proof of the classical isoperimetric inequality in \cite{CabreIsopCatalano, CabreIsopPubblicata} (see \cite{CabreSurvey} for a detailed account on the method). More precisely, for the classical isoperimetric problem, if $E \subset \R^n$ is a smooth connected open set, one would consider a solution $u$ to
\begin{equation}\label{eq:IntroProbABP}
    \begin{cases}
        \Delta u = \frac{P(E)}{|E|} & \text{ on } E,\\
        \partial_\nu u =1  & \text{ on } \partial E,\\
    \end{cases}
\end{equation}
where $\partial_\nu u$ denotes outward normal derivative. It is immediate to check that $\nabla u(E') \supset B_1(0)$ where $E':=\{ x \in E \st \nabla^2u(x) \ge0\}$, hence the area formula together with the arithmetic-geometric mean inequality readily imply the Euclidean isoperimetric inequality, indeed
\[
\omega_n = |B_1(0)| \le \int_{E'} \det\nabla^2 u \le \int_E \left(\frac{\Delta u}{n} \right)^n = \frac{P(E)^n}{n^n|E|^{n-1}}. 
\]
Now the rough idea is that a control on the energy deficit should control the "asymmetry" of the solution $u$ with respect to the solution corresponding to the optimal shape $B_1(0)$, that is the radially symmetric parabola $|x|^2/2$. In fact, this is achieved in \cite{CGPROS} by controlling the asymmetry of a \emph{coupling} function which is defined as a suitable convex envelope of $u$.
Adapting arguments from \cite{FigalliMaggiPratelli}, in \cite{CGPROS} the authors then show that it is possible to employ trace-type theorems to estimate the asymmetry of a competitor set in terms of the asymmetry of the coupling function, which is in turn estimated by the energy deficit.

In \cref{thm:IsopIneq} below we will give an ABP proof of the isoperimetric inequality for the problem \eqref{weighted:isopprob} by analyzing an elliptic problem analogous to \eqref{eq:IntroProbABP}, see \eqref{cabsur:3.2}. We are then in position to consider a coupling function as done in \cite{CGPROS} and we can quantitatively estimate its asymmetry, which will be achieved in \cref{prop:EsistenzaCoupling}. Moreover, Schwarz-symmetric sets that are sufficiently small $C^{1}$-perturbations (in the sense of \cref{def:VicinanzaC1Schwarz}) of an optimal bubble \eqref{eq:OptimalBubbles} readily verify the needed trace-type inequalities that relate asymmetry of the competitor with the asymmetry of the coupling. Hence this establishes the quantitative inequality \eqref{eq:FinalQuantitativeInequality} for Schwarz-symmetric $C^1$-perturbations of optimal bubbles, see \cref{quantitative:nearball}. Observe that, in our setting, isoperimetric sets are just Lipschitz-regular and a set $E$ is $C^1$-close to an optimal bubble if just the relative boundary $\partial E \cap \{x_n>0\}$ is close to the relative boundary of an optimal bubble as $C^1$-hypersurfaces with boundary.

Once \eqref{eq:FinalQuantitativeInequality} is proved for $C^1$-perturbations of optimal bubbles (\cref{quantitative:nearball}), we want apply a selection-type argument in the spirit of \cite{AcerbiFuscoMorini, CicaleseLeonardi} in the class of Schwarz-symmetric sets in order to extend the validity of the quantitative inequality to the whole class of Schwarz-symmetric sets. In this way we also avoid the implementation of further technical results that in \cite{FigalliMaggiPratelli, CGPROS} allow to reduce to just consider sets that enjoy the required trace-type inequalities.\\
Roughly speaking, in a selection-type argument one argues by contradiction assuming existence of sets contradicting the quantitative isoperimetric inequality and one uses such sets to define an auxiliary minimization problem, cf. \eqref{eq:ProblemSelectionPrinciple}. Minimizers to the previous problem still contradict the quantitative isoperimetric inequality, but at the same time they are shown to be small $C^1$-perturbations of some isoperimetric set, contradicting the inequality already proved for sets given by small perturbations of optimal ones.\\
In our case, we will prove that minimizers $E$ to the auxiliary minimization problem are $C^1$-perturbations of optimal bubbles up to the boundary of the half-space $\{x_n>0\}$ as a consequence of the classical interior regularity of $(\Lambda, r_0)$-minimizers of the perimeter (\cref{def:LambdaMin}), see \cite{Tamanini} and \cite[Chapter 26]{MaggiBook}, together with a simple variational argument that allows us to propagate the regularity up to the boundary of the half-space $\{x_n>0\}$. This essentially follows from the fact that a Schwarz-symmetric local $(\Lambda, r_0)$-minimizer $E$ in $\{x_n>0\}$ is locally of class $C^1$ and has bounded mean curvature (in a generalized sense, see \cref{prop:meancurvature}); hence a uniform bound on the whole second fundamental form on a portion of boundary $\partial E \cap \{0<x_n<\epsilon\}$ follows just by showing that the set $\partial E \cap \{0<x_n<\epsilon\}$ is far from the axis of revolution of $E$ (see \cref{curvature:revolution}), and the latter holds in the proof of \cref{thm:FinalQuantitativeInequality} by an energy estimate holding for minimizers to the auxiliary minimization problems.

\medskip

We stress that in our case it is not clear how to apply a Fuglede-type argument \cite{CicaleseLeonardi, Fuglede} to prove the quantitative inequality for sets given by small $C^1$-perturbations of optimal ones. Indeed, the classical Fuglede's method relies on the precise knowledge of the eigenvalues of the Laplace--Beltrami operator, which is not available for the operator on spherical caps corresponding to optimal bubbles \eqref{eq:OptimalBubbles} for generic $\lambda \in(-1,1)$. Moreover, observe that in our case it is not possible to globally parametrize $C^{1}$-close boundaries one on the other as normal graphs in general, introducing a further nontrivial technical difficulty in the implementation of a Fuglede-type argument.

\medskip

Once \cref{thm:FinalQuantitativeInequality} is proved, for the proof of \cref{thm:QuantitativaBagnata} we argue as follows. First we establish a quantitative inequality that estimates the Hausdorff distance between the relative boundary in $\{x_n>0\}$ of a competitor $E$ and the relative boundary of some bubble in terms of the Fraenkel asymmetry of $E$, under the assumption that $E$ is a so-called $(K,r_0)$-quasiminimal set, see \cref{def:QuasiminimalSets} and \cref{lem:StimaBetaQuasiminimal}. This is achieved since quasiminimal sets enjoy uniform density estimates at boundary points, see \cref{thm:RegolaritaQuasiminimal}. Exploiting \cref{thm:FinalQuantitativeInequality}, the previous quantitative inequality yields an inequality of the form \eqref{eq:QuantitativaBagnata} in the class of quasiminimal sets. Eventually, \cref{thm:QuantitativaBagnata} follows by applying again a selection-type argument where now $\beta_\lambda$ plays the role of the Fraenkel asymmetry.

\bigskip
\noindent\textbf{Organization.} 
In \cref{sec:Preliminaries} we collect definitions and facts on sets of finite perimeter and we prove some preliminary results on the capillarity functional. In \cref{sec:IsopIneqAsymmetryDeficit} we establish the sharp isoperimetric inequality for $P_\lambda$ via ABP method and we prove preliminary facts on the Fraenkel asymmetry and on the deficit. \cref{sec:Reductions} is devoted to a series of technical results that allow to reduce the analysis to Schwarz-symmetric sets. In \cref{sec:Quantitative} we complete the proof of \cref{thm:FinalQuantitativeInequality} by proving \eqref{eq:FinalQuantitativeInequality} on Schwarz-symmetric sets. Finally in \cref{sec:SecondQuantitative} we prove \cref{thm:QuantitativaBagnata}. In \cref{sec:Appendix} we recall some facts about $(\Lambda,r_0)$-minimizers and we recall a formula for the mean curvature of axially symmetric hypersurfaces.

\bigskip
\noindent\textbf{Addendum.} After this work was completed, we became aware of the independent \cite{KreutzSchmidt}, where the authors exploit the quantitative isoperimetric inequality for anisotropic perimeters provided in \cite{FigalliMaggiPratelli} to give a result analogous to \cref{thm:FinalQuantitativeInequality} in the anisotropic setting.

\section{Preliminaries}
\label{sec:Preliminaries}

\begin{center}
\emph{From now on and for the rest of the paper it is assumed that $\lambda \in (-1,1)$ and $n\in \N$ with $n\ge 2$ are fixed.}    
\end{center}

\medskip
\noindent\textbf{List of symbols.} Throughout the paper we shall adopt the following notation.

\begin{itemize}
    \item $|\cdot|$ denotes Lebesgue measure in $\R^n$.
    
    \item $B_r:= B_r(0) \subset \R^n$ for $r>0$, $B:=B_1$.
    
    \item $B^\lambda = \{x \in B : \Braket{x, e_n} > \lambda\}$.
    
    \item $B^\lambda(v) := \frac{v^{\frac{1}{n}}}{|B^\lambda|^{\frac1n}}(B^\lambda - \lambda e_n)$, for any $v>0$.
    
    \item $B^\lambda(v,x) := B^\lambda(v)  +x$, for any $x \in \{x_n=0\}$. In particular $B^\lambda(v)=B^\lambda(v,0)$.

    \item $c(n,\lambda), C(n,\lambda)$ denote strictly positive constants, depending on $n,\lambda$ only, that may change from line to line.

    \item $d_\hh$ denotes Hausdorff distance in $\R^n$.
    
    \item $H:=\{x_n\le0\}$.

    \item $\hh^d$ denotes $d$-dimensional Hausdorff measure in $\R^n$, for $d\ge0$.
    
    \item $P_\lambda(B^\lambda):= P_\lambda(B^\lambda(|B^\lambda|, x)) =  P(B, \{x_n>\lambda\}) -\lambda \hh^{n-1}(B\cap \{x_n=\lambda\})$.

    \item $Q_r:= [-r,r]^n \subset \R^n$, for any $r>0$.

    \item $r_\lambda := \min\left\{\sqrt{1 - \lambda^2}, 1 - \lambda\right\}$, $R_\lambda:= \max\left\{\sqrt{1 - \lambda^2}, 1 - \lambda\right\}$.
\end{itemize}

\subsection{Sets of finite perimeter}\label{sec:SetsFinitePerimeter}
We recall basic definitions and properties regarding sets of finite perimeter, referring to \cite{AmbrosioFuscoPallara, MaggiBook} for a complete treatment on the subject.
The perimeter of a measurable set $E\subset\R^n$ in an open set $A\subset \R^n$ is defined by
\begin{equation}\label{eq:DefPerimetro}
P(E,A):=\sup\left\{\int_E {\rm div}\, T(x) \de x \st T \in C_c^1(A;\R^n), \,\, \|T\|_\infty \le1\right\}. \end{equation}
Denoting $P(E):=P(E,\R^n)$, we say that $E$ is a set of finite perimeter if $P(E)<+\infty$. In such a case, the characteristic function $\chi_E$ has a distributional gradient $D\chi_E$ that is a vector-valued Radon measure on $\R^n$ such that \begin{equation*}
\int_E {\rm div}\, T(x) \de x = - \int_{\R^n} T  \de D\chi_E , \quad \forall T \in C_c^1(\R^n; \R^n). \end{equation*} 
It can be proved that the set function $P(E,\cdot)$ defined in \eqref{eq:DefPerimetro} is the restriction of a nonnegative Borel measure to open sets. The measure $P(E,\cdot)$ coincides with the total variation $|D\chi_E|$ of the distributional gradient, and it is concentrated on the reduced boundary
\begin{equation*}
\partial^*E : = \left\{ x \in {\rm spt} |D\chi_E| \st \exists\, \nu^E(x):=-\lim_{r\to0} \frac{D\chi_E(B_r(x))}{|D\chi_E(B_r(x))|} \text{ with } |\nu^E(x)|=1 \right\}.
\end{equation*}
Introducing the sets of density $t\in[0,1]$ points for $E$ defined by
\begin{equation*}
    E^{(t)} := \left\{ x \in \R^n \st \lim_{r\to0} \frac{|E \cap B_r(x)|}{|B_r(x)|}= t \right\},
\end{equation*}
we have that the reduced boundary coincides both with $\R^n\setminus(E^{(1)} \cup E^{(0)})$ and with the set $E^{(1/2)}$ up to $\hh^{n-1}$-negligible sets. The vector $\nu^E$ is called the generalized outer normal of $E$.
Moreover $P(E,\cdot) = \hh^{n-1}\mres \partial^*E$, and the distributional gradient can be written as $D\chi_E = - \nu^E \hh^{n-1}\mres \partial^*E$.

We conclude by recalling the definition of Schwarz symmetrization and of Schwarz-symmetric set.

\begin{definition}\label{def:SchwarzSymmetrization}
    Let $E\subset \R^n$ be a Borel set. Then its Schwarz symmetrization (with respect to the $n$-th axis) is the set
    \[
    E^* := \{ (x',t) \in \R^{n-1}\times \R \st \omega_{n-1}|x'|^{n-1}<\hh^{n-1}(E \cap \{x_n=t\}), \,\, t \in \R \}.
    \]
    A Borel set $E\subset \R^n$ is said to be Schwarz-symmetric if it coincides with its Schwarz symmetrization, up to negligible sets.
\end{definition}

\subsection{Preliminary results on the capillarity functional}

In this section we prove some preparatory results on the functional $P_\lambda$. From now on, $H$ denotes the closed half-space $H:=\{x_n\le 0\}$.

\begin{remark}\label{rem:PlambdaConDivergenza}
Let $E \subset \R^n \setminus H$ be a measurable set. We observe that
\[
P_\lambda(E) = \int_{\partial^*E \setminus H}1 - \lambda \Braket{e_n, \nu^E} \de \hh^{n-1},
\]
where $\nu^E$ is the generalized outer normal to $E$. In particular, since $|\lambda|<1$, we have that $P_\lambda(E)\ge0$. The previous identity follows from the divergence theorem, indeed
\begin{equation*} 0 = \int_E \text{div}\, e_n \de x = - \hh^{n - 1}(\partial^*E \cap \partial H) + \int_{\partial^*E \setminus H} \Braket{e_n, \nu^E}  \de \hh^{n-1}.
\end{equation*}
\end{remark}

We now aim at proving an approximation result for sets of finite perimeter contained in $\R^n \setminus H$ by sequences of sets having smooth boundary relative in $\R^n\setminus H$. We will employ the following lemma.

\begin{lemma}\label{lemma:predensity}
Let $E \subset \R^n \setminus H$ be a set of finite perimeter with finite measure. For any $t \ge 0$ define the diffeomorphism $\varphi_t: \{x_n\ge 0\} \to \{x_n\ge 0\}$ given by $\phi_t(x', x_n) := ( x' , x_n (1 + t e^{-|x'|^2}))$, where we wrote $x=(x',x_n) \in \R^{n-1}\times \R$ for any $x \in \{x_n\ge 0\} \subset \R^n$.
Then \begin{enumerate} 
\item for at most countably many $t \in [0,+\infty)$ there holds
\begin{equation*} \hh^{n - 1}\left(\left\{x \in \partial^*(\phi_t(E)) \setminus H \st \nu^{\phi_t(E)}(x) = \pm e_n \right\}\right) > 0; \end{equation*}

\item for at most countably many $\nu \in \mathbb{S}^{n - 1}$ there holds
\begin{equation*} \hh^{n - 1}\left(\left\{x \in \partial^*E \st \nu^{E}(x) = \pm \nu \right\}\right) > 0. \end{equation*} \end{enumerate}
\end{lemma}

\begin{proof} We begin by proving (1).
Define $f_t:\R^{n-1}\to \R$ by $ f_t(x') := 1 + t e^{- |x'|^2}$,
and let
\begin{equation*}
\begin{split}
    A_t := \Big\{(x', x_n) \in \partial^*E \cap \R^n \setminus H \; &{\rm with}\; x' \in \R^{n - 1}\setminus \{0\} \st \nu^E(x', x_n) \text{ is proportional to } (x_n \nabla f_t(x'), f_t(x')) \Big\}. 
\end{split}
\end{equation*}
We claim that, if $s$, $r >0$, with $s \neq r$, then $A_r \cap A_s = \emptyset$. 
Indeed, if $A_r \cap A_s \neq \emptyset$, there exist $(\bar{x}', \bar{x}_n) \in \partial^*E \cap \R^n \setminus H $ and $\bar{\alpha} \in \R\setminus\{0\}$ such that \begin{equation*} (\bar{x}_n \nabla f_r(\bar{x}'), f_r(\bar{x}')) = \bar{\alpha} (\bar{x}_n \nabla f_s(\bar{x}'), f_s(\bar{x}')).
\end{equation*}
Since $x_n>0$ and $x'\neq0$, then $\bar{x}_n \nabla f_r(\bar{x}')=\bar{\alpha} \bar{x}_n \nabla f_s(\bar{x}') $ implies $r=\bar\alpha \,s$. Thus $f_r(\bar{x}') = \bar\alpha f_s(\bar{x}')$ implies $1=\bar\alpha$, which in turn implies $r=s$, contradiction.\\
For $k \in \N_{\ge 1}$ let $M_k := \left\{t>0 \st \hh^{n - 1}(A_t) > 1/k\right\}$.
Since $A_r \cap A_s = \emptyset$ for $r\neq s$ and since $P(E)<+\infty$, it follows that $M_k$ is finite for any $k\ge1$, and thus $\left\{t >0 \st \hh^{n - 1}(A_t) > 0\right\} = \bigcup_k M_k$ is countable.

Since the differential of $\varphi_t$ is represented by the matrix
\[
\d \varphi_t (x',x_n)=\begin{pmatrix}
    {\rm Id}_{\R^{n-1}} & 0 \\
    x_n\nabla f_t(x') & f_t(x')
\end{pmatrix},
\]
requiring that $\nu^{\phi_t(E)}(\phi_t(x)) = \pm e_n$ for some $x=(x',x_n)$ means that
\begin{equation*} \nu^E(x', x_n) \text{ is proportional to } (x_n \nabla f_t(x'), f_t(x')), \end{equation*}
see \cite[Proposition 17.1]{MaggiBook} and \cite[Proposition 2.88]{AmbrosioFuscoPallara}.
Therefore for any $t \not \in \cup_k M_k$ there holds
\begin{equation*} \hh^{n - 1}\left(\left\{x \in \partial^*(\phi_t(E)) \setminus H \st \nu^{\phi_t(E)}(x) = \pm e_n \right\}\right)  =0. \end{equation*}

The proof of (2) is analogous, noticing that $N_k := \left\{\nu \in \mathbb{S}^{n - 1} \st \hh^{n - 1}\left(\left\{x \in \partial^*E \st \nu^E = \nu\right\}\right) > 1/k\right\}$ is finite for any $k \in \N_{\ge1}$, hence $\cup_k N_k$ is at most countable.
\end{proof}

\begin{lemma}[Approximation with regular sets]\label{lemma:Density}
Let $E\subset \R^n\setminus H$ be a set of finite perimeter with finite measure. Then there exists a sequence of sets $E_i\subset \R^n \setminus H$ such that
\begin{enumerate}
    \item $E_i$ is a bounded set such that $\partial E_i\setminus\partial H$ is a smooth hypersurface (possibly with smooth boundary) such that either $\partial E_i \cap \partial H=\emptyset$ or $\partial E_i\setminus\partial H$ intersects $\partial H$ orthogonally;

    \item $E_i \to E$ in $L^1$, $P( E_i , \R^n\setminus H) \to P( E , \R^n\setminus H)$, $\hh^{n-1}(\partial E_i \cap \partial H) \to \hh^{n-1}(\partial^* E \cap \partial H)$, as $i\to +\infty$.\\
    In particular, $P(E_i) \to P(E)$ and $P_\lambda(E_i) \to P_\lambda(E)$ as $i\to +\infty$, for any $\lambda \in (-1,1)$;

    \item $\hh^{n-1}(\{x \in \partial^* E_i \st \nu^{E_i}(x)=\pm e_j\}) =0$ for any $j=1,\ldots, n-1$;

    \item $\hh^{n-1}(\{x \in \partial^* E_i \setminus H \st \nu^{E_i}(x)=\pm e_n\}) =0$.    
\end{enumerate}
\end{lemma}

\begin{proof}
Since $P(E), |E| <+\infty$, by a diagonal argument we can assume without loss of generality that $E$ is bounded.

\emph{Step 1.} We first construct a sequence $F_i \subset \R^n\setminus H$ such that 1. and 2. hold with $F_i$ in place of $E_i$.
Let us denote by $F$ the union of $E$ with the reflection of $E$ with respect to the hyperplane $\{x_n = 0\}$. There exists a sequence of smooth sets $\tilde F_i\subset \R^n$, symmetric with respect to $\{x_n = 0\}$, such that they converge to $F$ in $L^1(\R^n)$ and $P(\tilde F_i) \to P(F)$ (see, e.g., \cite[Theorem 3.42]{AmbrosioFuscoPallara}). The fact that $\tilde F_i$ is symmetric with respect to $\{x_n = 0\}$ follows as $\tilde F_i$ can be obtained as superlevel set of a convolution of $\chi_F$ with a symmetric mollifier. 
In particular, if $\partial \tilde F_i \cap \partial H \neq \emptyset$, then $\partial \tilde F_i$ intersects $H$ orthogonally, and thus $\partial \tilde F_i \cap \partial H$ is a smooth $(n-2)$-dimensional manifold.
Let $F_i := \tilde F_i \setminus H$.
Then $F_i \to E$ in $L^1$ and \begin{equation*} P(F_i, \R^n \setminus H) = \frac{1}{2} P(\tilde F_i) \to \frac{1}{2} P(F) = P(E, \R^n \setminus H). \end{equation*}
Let us define the function \begin{equation*} \begin{split} & f :  \R^n \to [0, + \infty) \\ & f(v) = |v| - \lambda \Braket{e_n, v}. \end{split} \end{equation*} Note that $f$ is continuous; moreover $f(tv) = t f(v)$ for any $t\ge0 $ and $f$ is convex.
Let us set \begin{equation*} \begin{split} & \mu_i := \nu^{F_i} \hh^{n - 1}\mres (\partial^*F_i \cap (\R^n \setminus H)) \\ & \mu := \nu^E \hh^{n - 1}\mres(\partial^*E \cap (\R^n \setminus H)). \end{split} \end{equation*} 
Since $\nu^{F_i} \hh^{n - 1}\mres \partial^*F_i \to \nu^E \hh^{n - 1}\mres \partial^*E $ weakly$^*$ in $\R^n$, then $\mu_i \to \mu$ weakly$^*$ in $\R^n \setminus H$.
Since we already know that $|\mu_i|(\R^n \setminus H) \to |\mu|(\R^n \setminus H)$, by Reshetnyak continuity theorem (see, e.g., \cite[Theorem 2.39]{AmbrosioFuscoPallara}) we deduce $P_\lambda(F_i) = \int f(\mu_i/|\mu_i|) \de |\mu_i| \to \int f(\mu/|\mu|) \de |\mu| = P_\lambda(E)$.

\emph{Step 2.} Let us consider the flow $\phi_t : \R^n\setminus H \to \R^n\setminus H$ such that \begin{align*} \phi_t(x', x_n) &:= ( x' , x_n (1 + t e^{-|x'|^2}) ), \end{align*}
for $t\ge0$.
By \cref{lemma:predensity} for a.e. $t$ we have that $\phi_t(F_i)$ satisfies (4).
Moreover, for any $t\ge 0$ there exists a sequence of rotations $\mathcal{R}_{j,t}:(x',x_n)\mapsto(R_{j,t}(x'), x_n)$ along the $n$-th axis converging to the identity such that
\begin{equation*}
\mathcal{H}^{n - 1}(\{x \in \partial^*(\mathcal{R}_{j,t}( \phi_t(F_i)) : \nu^{\mathcal{R}_{j,t}(\phi_t(F_i))}(x) = \pm e_l\}) = 0 \end{equation*} for all $l = 1$, $\dots$, $n - 1$. By a diagonal argument, since $\varphi_t(F_i)$ maintains orthogonal intersection with $\partial H$, one extract the desired sequence $E_i$.
\end{proof}

\begin{corollary}\label{remark:positivity}
Let $E\subset \R^n\setminus H$ be a set of finite perimeter with finite measure. Then
\begin{equation*} P_\lambda(E) \ge \frac{1-\lambda}{2}P(E).
\end{equation*}
\end{corollary}

\begin{proof}
We can assume that $E$ is as in the conclusion of \cref{lemma:Density}. The claim readily follows from the fact that the orthogonal projection on $\partial H$ is $1$-Lipschitz and surjective from $\partial E \cap (\R^n \setminus H)$ onto $\partial E \cap \partial H$, recalling that $1$-Lipschitz maps do not increase the Hausdorff measure (see, e.g., \cite[Proposition 2.49]{AmbrosioFuscoPallara}).
\end{proof}

\section{Isoperimetric inequality for \texorpdfstring{$P_\lambda$}{}, Fraenkel asymmetry and deficit}\label{sec:IsopIneqAsymmetryDeficit}

\subsection{Isoperimetric inequality via ABP method}

In this section we give a proof of the isoperimetric inequality for the capillarity functional $P_\lambda$ exploiting an ABP method. We start by computing the capillarity perimeter of bubbles $B^\lambda(v)$.

\begin{lemma}\label{lem:ValuePlambda}
There holds $P_\lambda(B^\lambda(v)) = n \lvert B^\lambda\rvert^{\frac1n} v^{\frac{n-1}{n}}$, for any $v\ge0$ and $\lambda\in(-1,1)$.
\end{lemma}

\begin{proof}
By scale invariance, it is sufficient to prove that $P(B^\lambda, \{x_n>\lambda\}) - \lambda \hh^{n-1}(\partial B^\lambda \cap \{x_n=\lambda\})$ is equal to $n|B^\lambda|$. Indeed, let $u(x) = \tfrac12 |x|^2$. Then
\begin{equation*}
	\begin{split}
	n|B^\lambda| = \int_{B^\lambda} \Delta u = P(B^\lambda, \{x_n>\lambda\}) + \int_{\partial B^\lambda \cap \{x_n=\lambda\}}
	\Braket{- e_n, x} =  P(B^\lambda, \{x_n>\lambda\}) - \lambda \hh^{n-1}(\partial B^\lambda \cap \{x_n=\lambda\}).
	\end{split}
\end{equation*}
\end{proof}

\begin{remark}\label{remark:regularity}
Let $E\subset \R^n\setminus H$ be a connected open set such that $\partial E \setminus H$ is a smooth hypersurface with boundary that intersects $\partial H$ orthogonally. Then the Neumann problem
\begin{equation}\label{cabsur:3.2bis} \begin{cases}  \Delta u = \frac{P_\lambda(E)}{\lvert E\rvert} \qquad &\text{in} \quad E, \\ 
\frac{\partial u}{\partial \nu} = 1 & \text{on} \quad \partial E\setminus\partial H ,\\ 
\frac{\partial u}{\partial \nu} = - \lambda & \text{on} \quad \partial E\cap\partial H, \end{cases} \end{equation}
has a solution $u \in C^1(\overline{E}) \cap C^\infty(E)$.

Indeed, existence of a weak solution of~\eqref{cabsur:3.2bis} follows by classical arguments exploiting the Riesz representation theorem.
By \cite[Proposition 3.6]{NittkaRegularity} there exists $\gamma > 0$ such that every weak solution is in $C^{0, \gamma}(E)$.
Hence we can apply \cite[Theorem 1]{LiebermanOblique} to the equivalent problem
\begin{equation*} \begin{cases}  \Delta u -  u = \frac{P_\lambda(E)}{\lvert E\rvert} -  u =:f \qquad &\text{in} \quad E, \\ 
\frac{\partial u}{\partial \nu} = 1 & \text{on} \quad \partial E\setminus\partial H ,\\ 
\frac{\partial u}{\partial \nu} = - \lambda & \text{on} \quad \partial E\cap\partial H \end{cases} \end{equation*}
getting that a weak solution is in fact $C^1(\overline{E})$.
\end{remark}

\begin{theorem}[Isoperimetric inequality for $P_\lambda$]\label{thm:IsopIneq}
Let $E \subset \R^n \setminus H$ be a set of finite perimeter with $|E|\in(0,+\infty)$. Then
\begin{equation}\label{eq:IsoperimetricaLambda}
\frac{P_\lambda(E)}{|E|^{\frac{n - 1}{n}}} \ge \frac{P_\lambda(B^\lambda)}{|B^\lambda|^{\frac{n - 1}{n}}} = n |B^\lambda|^{\frac1n}.
\end{equation} 
Moreover, equality occurs in~\eqref{eq:IsoperimetricaLambda} if and only if $E = B^\lambda(|E|)$ up to a translation and up to negligible sets.
\end{theorem}

\begin{proof}
We just give a proof of the inequality \eqref{eq:IsoperimetricaLambda} here, referring to \cite[Theorem 19.21]{MaggiBook} for an alternative proof comprising the characterization of minimizers.

By the standard isoperimetric inequality, we can assume that $\hh^{n-1}(\partial^* E \cap \partial H)>0$. By \cref{lemma:Density}, we can further assume that $E$ is a bounded set such that $\partial E\setminus\partial H$ is smooth and intersects $\partial H$ orthogonally.

Let us further assume for the moment that $E$ is connected.
Let $u$ be the solution of the Neumann problem \begin{equation}\label{cabsur:3.2} \begin{cases}  \Delta u = \frac{P_\lambda(E)}{\lvert E\rvert} \qquad &\text{in} \quad E, \\
\frac{\partial u}{\partial \nu} = 1 & \text{on} \quad \partial E\setminus\partial H, \\ 
\frac{\partial u}{\partial \nu} = - \lambda & \text{on} \quad \partial E\cap\partial H, \end{cases} \end{equation}
where $\partial u /\partial \nu$ denotes the outer normal derivative of $u$ on $\partial E$. 
Observe that such a solution exists and $u \in C^1(\overline{E}) \cap C^\infty(E)$ (see \cref{remark:regularity}).
We consider the ``lower contact set'' of $u$ defined by 
\begin{equation}\label{cabsur:3.3} \Gamma_u := \left\{x \in E: u(y) \ge u(x) + \Braket{\nabla u(x), y - x} \; \text{for all} \; y \in \overline{E}\right\}. \end{equation} We claim that \begin{equation}\label{cabsur:3.4} 
B^\lambda \subset \nabla u(\Gamma_u). \end{equation}
To show~\eqref{cabsur:3.4}, take any $p \in B^\lambda$. Let $x \in \overline{E}$ be a point such that
\begin{equation*} \min_{y \in \overline{E}} \left\{u(y) - \Braket{p, y}\right\} = u(x) - \Braket{p, x}. \end{equation*}
If $x \in \partial E \setminus \partial H$ then the exterior normal derivative of $u(y) - \Braket{p, y}$ at $x$ would be nonpositive and hence $(\partial u /\partial \nu)(x) \le \lvert p\rvert < 1$, a contradiction with~\eqref{cabsur:3.2}. Similarly, if $x \in \partial E \cap \partial H$ then $(\partial u /\partial \nu)(x) \le \Braket{p, -e_n} < -\lambda$, a contradiction with~\eqref{cabsur:3.2}. It follows that $x \in E$ and, therefore, that $x$ is an interior minimum of the function $u(y) - \Braket{p, y}$ over $\overline{E}$. In particular $p = \nabla u(x)$ and $x \in \Gamma_u$, hence Claim~\eqref{cabsur:3.4} is now proved.

From~\eqref{cabsur:3.4}, since $u \in C^\infty(E)$ and $\Gamma_u\subset E$, we can apply the area formula on $\nabla u$ to deduce \begin{equation}\label{cabsur:3.5} \lvert B^\lambda\rvert \le \lvert \nabla u (\Gamma_u)\rvert = \int_{\nabla u (\Gamma_u)} \de p \le \int_{\Gamma_u} |\text{det} \, \nabla^2u(x)| \de x.
\end{equation}
Since points $x\in \Gamma_u$ are interior minima for $y \mapsto u(y)-\Braket{\nabla u(x), y}$, then $\nabla^2u(x)$ is positively semi-definite. Hence by the arithmetic-geometric mean inequality
 \begin{equation*}\label{cabsur:3.6} |\text{det}\, \nabla^2u|=\text{det}\, \nabla^2u \le \left(\frac{\Delta u}{n}\right)^n \qquad \text{in } \Gamma_u.
\end{equation*}
Hence
\[
|B^\lambda| \le \int_{\Gamma_u}  \text{det} \, \nabla^2u \de x \le \int_{\Gamma_u} \left(\frac{\Delta u}{n}\right)^n \de x \le \int_E \left(\frac{\Delta u}{n}\right)^n \de x,
\]
since $\Delta u \equiv P_\lambda(E)/|E|$.
Plugging in the value of $\Delta u$, the claimed inequality follows.\\
It remains to consider the case when $E$ is not connected, hence when $E$ is a disjoint union of finitely many bounded sets $E_i$, for $i=1,\ldots, k$, such that $\partial E_i \setminus \partial H$ is smooth and intersects $\partial H$ orthogonally. 
Summing over $i=1,\ldots, k$ the isoperimetric inequality that we just proved for $P_\lambda$ on each component $E_i$, and exploiting the subadditivity of $t\mapsto t^{\frac{n-1}{n}}$, the final inequality follows.
\end{proof}

\subsection{Asymmetry and deficit}

We now define the Fraenkel asymmetry with respect to optimal bubbles $B^\lambda(v,x)$ and the deficit corresponding to the functional $P_\lambda$, proving some preliminary properties on these quantities.

\begin{definition}[Fraenkel asymmetry]\label{def:Asymmetry}
Let $E \subset \R^n \setminus H$ be a Borel set with measure $|E|=v\in(0,+\infty)$. We define
\begin{equation*} \alpha_\lambda(E) := \inf\left\{\frac{|E \Delta B^\lambda(v,x)|}{v} : x \in \{x_n=0\}\right\}. \end{equation*}
It is readily checked that the Fraenkel asymmetry of $E$ is a minimum. 
\end{definition}

\begin{definition}[Isoperimetric deficit] 
Let $E \subset \R^n \setminus H$ be a Borel set with measure $|E|=v\in(0,+\infty)$. We define
\begin{equation*} D_\lambda(E) := \frac{P_\lambda(E) - P_\lambda(B^\lambda(v))}{P_\lambda(B^\lambda(v))}. \end{equation*}
\end{definition}

\begin{lemma}\label{lemma:DeficitPiccolo}
    There exists $\bar c=\bar c(n,\lambda)>0$ such that if a Borel set $E\subset \R^n\setminus H$ satisfies $D_\lambda(E) < \bar c$ then $P(E,\partial H)>0$.
\end{lemma}

\begin{proof}
    If $E$ is a Borel set such that $D_\lambda(E) < \bar c$ and $P (E, \partial H) = 0$, then the standard isoperimetric inequality together with \cref{lem:ValuePlambda} imply
    \begin{equation*}
    \begin{split} n \omega_n^{\frac{1}{n}} |E|^{\frac{n - 1}{n}} & \le P(E)  = P_\lambda(E)  < (1 + \bar c) P_\lambda(B^\lambda(|E|)) = n |B^\lambda|^{\frac{1}{n}} (1 + \bar c) |E|^{\frac{n - 1}{n}}. \end{split} \end{equation*}
    Since $|E|>0$ for $\bar c$ small enough, we get a contradiction if $\bar c$ is sufficiently small.
\end{proof}

\begin{lemma}\label{lemma:LowerSemicontinuity}
If $\{E_i\}_{i \in \N}$ and $E$ are sets of finite perimeter in $\R^n \setminus H$ with finite measure such that $E_i\to E$ in $L^1_{\rm loc}$.
Then \begin{equation*} \liminf_{i\to+\infty} P_\lambda(E_i) \ge P_\lambda(E),\qquad \liminf_{i\to+\infty} D_\lambda(E_i) \ge D_\lambda(E). \end{equation*}
\end{lemma}

\begin{proof}
Exploiting \cref{rem:PlambdaConDivergenza} and Reshetnyak lower semicontinuity theorem (see, e.g., \cite[Theorem 2.38]{AmbrosioFuscoPallara}), one easily checks that $P_\lambda$ is lower semicontinuous with respect to $L^1_{\rm loc}$ convergence (see the proof of \cref{lemma:Density}).
\end{proof}

\begin{lemma}\label{lem:ContinuityAlpha}
If $\{E_i\}_{i \in \N}$ and $E$ are sets of finite perimeter in $\R^n \setminus H$ with finite measure such that $|E|>0$ and such that $E_i \to E$ in $L^1(\R^n\setminus H)$.
Then \begin{equation*} \lim_{i\to+\infty}\alpha_\lambda(E_i) = \alpha_\lambda(E) 
\end{equation*}
\end{lemma}

\begin{proof}
The claim is a standard exercise that follows from the fact that the $L^1$-convergence holds on the whole $\R^n\setminus H$.
\end{proof}

\begin{corollary}\label{lemma:precompattezza}
If $\{E_i\}_{i \in \N}$ are sets of finite perimeter in $\R^n\setminus H$ such that $|E_i \setminus K|=0$ for any $i$ for some compact set $K\subset \R^n$, and if
\begin{equation*} \begin{split} & \sup_{i \in \N} |E_i| + P_\lambda(E_i) < \infty,
\end{split} \end{equation*}
then there exists $E $ of finite perimeter in $\R^n\setminus H$ and $i_k \to \infty$ as $k \to \infty$ such that 
\begin{equation*} E_{i_k} \xrightarrow[]{L^1} E \qquad \liminf_kP_\lambda(E_{i_k}) \ge P_\lambda(E). \end{equation*}
\end{corollary}

\begin{proof}
By \cref{remark:positivity} we have that $ P_\lambda(E_i) \ge \frac{1 - \lambda}{2} P(E)$. Then $\sup_{i \in \N} P(E_i) < \infty$. Hence by classical precompactness of sets of finite perimeter (see, e.g., \cite[Theorem 12.26]{MaggiBook}) and recalling \cref{lemma:LowerSemicontinuity}, the claim follows.
\end{proof}

\section{Reduction to bounded symmetric sets}\label{sec:Reductions}

In the following arguments, in order to prove \cref{thm:FinalQuantitativeInequality}, we will repeatedly reduce ourselves to consider sets $E$ having isoperimetric deficit smaller than some chosen constant.
This reduction is always possible.\\ Indeed, let $\delta>0$ be some positive constant; if $E$ is a set of finite perimeter such that $D_\lambda(E) \ge \delta$, since $\alpha_\lambda(E) \le 2$, we immediately get
\begin{equation*} \begin{split} & \alpha_\lambda^2(E) \le \frac{4}{\delta}\delta \le  \frac{4}{\delta} D_\lambda(E). \end{split} \end{equation*}
Therefore, if \cref{thm:FinalQuantitativeInequality} is proved on sets with deficit $\le \delta$, then it is proved for any set.\\
Hence,
\[
    \text{{\em within this section we will assume that $D_\lambda(E)< \bar c$ for any competitor $E$ involved, where $\bar c$ is given by \cref{lemma:DeficitPiccolo}.}}
\]
\[
    \text{{\em In particular $P(E,\partial H)>0$. }}
\]

\subsection{Reduction to bounded sets}

In this section we prove that, in order to prove \cref{thm:FinalQuantitativeInequality}, it is sufficient to prove the quantitative isoperimetric inequality \eqref{eq:FinalQuantitativeInequality} among suitably uniformly bounded sets.

From now on, we shall denote $Q_l:= [-l,l]^n \subset \R^n$. We start by proving an estimate on the area of horizontal slices of a set in terms of its deficit.

\begin{lemma}\label{lem:LowerBoundAreaFette}
Let $E\subset \R^n\setminus H$ be a bounded set of finite perimeter such that $\partial E \cap \R^n\setminus H$ is a smooth hypersurface (possibly with smooth boundary) with $|E|=|B^\lambda|$ and such that $\hh^{n-1}(\{x \in \partial^* E \setminus H : \nu^E(x) = \pm e_n\})=0$. Then
\begin{equation}\label{eq:LowerBoundAreaFette}
    \hh^{n-1}( E \cap \{x_n=t\}) \ge \frac12 P_\lambda(B^\lambda) \left(  \left( \frac{\omega_n}{|B^\lambda|} \right)^{\frac1n} \left( 1- \frac{|E \cap \{x_n<t\}|}{|B^\lambda|} \right)^{\frac{n-1}{n}} -1 - D_\lambda(E) \right),
\end{equation}
for every $t>0$. In particular
\begin{equation}\label{eq:LowerBoundAreaBagnata}
    \hh^{n-1}(\partial^*E \cap \partial H) \ge  \frac12 P_\lambda(B^\lambda) \left(  \left( \frac{\omega_n}{|B^\lambda|} \right)^{\frac1n}  -1 - D_\lambda(E) \right).
\end{equation}
Moreover, if $F\subset \R^n\setminus H$ is a set of finite perimeter with $|F|=|B^\lambda|$, then \eqref{eq:LowerBoundAreaFette} holds with $F$ in place of $E$ for almost every $t>0$, and \eqref{eq:LowerBoundAreaBagnata} holds with $F$ in place of $E$.
\end{lemma}

The estimates given by \cref{lem:LowerBoundAreaFette} are clearly nontrivial only when the deficit is sufficiently small. On the other hand, if the deficit $D_\lambda(E)$ is sufficiently small, since $\omega_n/|B^\lambda|>1$, \eqref{eq:LowerBoundAreaFette} and \eqref{eq:LowerBoundAreaBagnata} essentially yield a quantitative version of \cref{lemma:DeficitPiccolo}, nontrivial also for slices $\hh^{n-1}(E\cap\{x_n=t\})$ with $t>0$ as long as $|E \cap \{x_n<t\}|$ is small.

\begin{proof}[Proof of \cref{lem:LowerBoundAreaFette}]
Let $v_E(t):=\hh^{n-1}( E \cap \{x_n=t\})$ for any $t>0$, and let $g(t) := |E \cap \{x_n<t\}|/|B^\lambda|$. By the standard isoperimetric inequality we have that
\begin{equation}\label{eq:zzStimaTaglio1}
\begin{split}
    P(E, \{x_n>t\}) + v_E(t) &= P( E \cap \{x_n>t\}) \ge n\omega_n^{\frac1n} |E \cap \{x_n>t\}|^{\frac{n-1}{n}}  = P_\lambda(B^\lambda)  \left( \frac{\omega_n}{|B^\lambda|} \right)^{\frac1n} (1-g(t))^{\frac{n-1}{n}},
\end{split}
\end{equation}
for any $t>0$. Moreover, for any $t>0$, we observe that for any $x' \in \partial^*E \cap \partial H$, the halfline $[0,t] \ni x_n \mapsto (x',x_n)$ either intersects $\partial^* E \cap \{0<x_n\le t\}$ or it intersects $E \cap \{x_n=t\}$. Therefore
\begin{equation}\label{eq:zzStimaTaglio2}
    \begin{split}
        P(E, \{0<x_n\le t\}) + v_E(t)
        & \ge \hh^{n-1}(\partial^*E \cap \partial H),
    \end{split}
\end{equation}
for any $t>0$.
Hence we conclude that
\begin{equation*}
    \begin{split}
        P_\lambda(B^\lambda)(1+D_\lambda(E) )
        &= P_\lambda(E) = P(E, \{x_n>t\}) + P(E,\{0<x_n\le t\}) - \lambda\hh^{n-1}(\partial^*E \cap \partial H) \\
        &\overset{\eqref{eq:zzStimaTaglio2}}{\ge}
         P(E, \{x_n>t\}) - v_E(t) \\
        &\overset{\eqref{eq:zzStimaTaglio1}}{\ge}
        P_\lambda(B^\lambda)  \left( \frac{\omega_n}{|B^\lambda|} \right)^{\frac1n} (1-g(t))^{\frac{n-1}{n}} - 2 v_E(t),
    \end{split}
\end{equation*}
for any $t>0$, which yields \eqref{eq:LowerBoundAreaFette}. By \cite[Theorem 18.11]{MaggiBook} (see also \cite[Theorem 6.1]{FuscoMaggiPratelli}), the function $v_E$ belongs to $W^{1,1}(0,+\infty)$, thus \eqref{eq:LowerBoundAreaBagnata} follows by letting $t\to0^+$ in \eqref{eq:LowerBoundAreaFette}.

Now if $F\subset \R^n\setminus H$ is as in the assumptions, let $E_i$ be given by \cref{lemma:Density} applied to $F$, and let $\tilde E_i:= (|B^\lambda|/|E_i|)^{\frac1n} E_i$. Hence the inequality \eqref{eq:LowerBoundAreaBagnata} and the right hand side of \eqref{eq:LowerBoundAreaFette} applied with $E= \tilde E_i$ pass to the limit as $i\to\infty$. Moreover
\[
|\tilde E_i \Delta E| = \int_0^{+\infty} \hh^{n-1} ( \tilde E_i \Delta E \cap \{x_n=t\} ) \de t 
\ge \int_0^{+\infty} \left| \hh^{n-1}(\tilde E_i \cap \{x_n=t\}) - \hh^{n-1}( E \cap \{x_n=t\}) \right| \de t.
\]
Since $|\tilde E_i \Delta E| \to 0$, the left hand side of \eqref{eq:LowerBoundAreaFette} passes to the limit as well as $i\to\infty$, for a.e. $t>0$.
\end{proof}

We are ready to prove the claimed reduction to bounded sets. The proof follows the line of \cite[Lemma 5.1]{FuscoMaggiPratelli}, essentially truncating a competitor with coordinate slabs having estimated width. To give a bound for the truncation in the $n$-th direction will need to modify the argument in \cite[Lemma 5.1]{FuscoMaggiPratelli} and we will exploit \cref{lem:LowerBoundAreaFette}.

\begin{lemma}[Reduction to bounded sets]\label{lem:ReductionToBounded}
There exist $l = l(n, \lambda)>0$ and $C_1 = C_1(n, \lambda)>0$ such that, if $E \subset \R^n \setminus H$ is a set of finite perimeter with $|E|\in(0,+\infty)$, then there exists a set of finite perimeter $E' \subset Q_l \cap (\R^n\setminus H)$ such that $|E'|=|B^\lambda|$ and
\begin{equation}\label{fmp:5.3} \alpha_\lambda(E) \le \alpha_\lambda(E') + C_1 D_\lambda(E), \qquad D_\lambda(E') \le C_1 D_\lambda(E). \end{equation}
\end{lemma}

\begin{proof}
By scale-invariance of the asymmetry and of the deficit, it is sufficient to prove the claim assuming also $|E|=|B^\lambda|$.
First of all we observe that we may prove the claim assuming that $\partial E \cap \R^n\setminus H$ is smooth and \begin{equation}\label{fmp:5.4}
\mathcal{H}^{n - 1}(\{x \in \partial^* E \cap \R^n\setminus H : \nu^E(x) = \pm e_i\}) = 0 \end{equation}
for all $i = 1$, $\dots$, $n$. Indeed, if $E$ is a generic set of finite perimeter, then by \cref{lemma:Density} there exists a sequence of smooth sets $\{E_i\}_{i \in \N}$ converging to $E$ such that~\eqref{fmp:5.4} holds. If we know that the claim holds for $E_i$, we get the existence of $E_i'\subset Q_l \cap (\R^n \setminus H)$ such that \eqref{fmp:5.3} holds with $E, E'$ replaced by $E_i, E'_i$. Hence we can apply \cref{lemma:precompattezza} on the sequence $E_i'$, and by \cref{lemma:LowerSemicontinuity} and \cref{lem:ContinuityAlpha} the inequalities~\eqref{fmp:5.3} pass to the limit.

Without loss of generality, we can further assume that
\begin{equation}\label{eq:zzSmallDeficit}
    D_\lambda(E) < (2^{1/n} - 1)/4.
\end{equation}

Let us consider the axis $x_1$ first. Thanks to~\eqref{fmp:5.4}, by \cite[Theorem 18.11]{MaggiBook} (see also \cite[Theorem 6.1]{FuscoMaggiPratelli}) we deduce that \begin{equation*} v_E(t) := \mathcal{H}^{n - 1}(\{x' \in \R^{n - 1} : (t, x') \in E\}) \qquad \text{for } t \in \R \end{equation*} belongs to $W^{1, 1}(\R)$, hence we may assume that $v_E$ is continuous. Setting
\begin{equation*} E_t^- := \{x \in E : x_1 < t\} \end{equation*}
and
\begin{equation*}
P_\lambda(E, \{x_1 < t\}) := P(E, \{x \in \R^n \setminus H : x_1 < t\}) - \lambda \hh^{n - 1}(\{x \in \partial^*E\cap\partial H : x_1 < t\}) \end{equation*}
for all $t \in \R$, by smoothness of $E$ we have that
\begin{equation}\label{fmp:5.5} P_\lambda(E_t^-) = P_\lambda(E, \{x_1 < t\}) + v_E(t), \qquad P_\lambda(E\setminus E_t^-) = P_\lambda(E, \{x_1 > t\}) + v_E(t),
\end{equation}
where $P_\lambda(E, \{x_1 > t\})$ is defined analogously.
Let us now define the function $g : \R \to [0, + \infty)$ given by
\begin{equation*} g(t) := \frac{|E_t^-|}{|B^\lambda|}.
\end{equation*}
Hence $g$ is a nondecreasing $C^1$ function with $g'(t) = v_E(t)/|B^\lambda|$. Let $-\infty \le a < b \le + \infty$ be such that $\{t: 0 < g(t) < 1\} = (a, b)$. If $t \in (a, b)$, then by~\eqref{eq:IsoperimetricaLambda} we have
\begin{equation*}
P_\lambda(E_t^-) \ge g(t)^{\frac{n - 1}{n}}P_\lambda(B^\lambda).
\end{equation*}
Similarly,
\begin{equation*}
P_\lambda(E \setminus E_t^-) \ge (1 - g(t))^{\frac{n - 1}{n}}P_\lambda(B^\lambda).
\end{equation*}
Therefore, from~\eqref{fmp:5.4} and~\eqref{fmp:5.5} we get that \begin{equation*} P_\lambda(E) + 2 v_E(t) \ge P_\lambda(B^\lambda)\left(g(t)^{\frac{n - 1}{n}} + (1 - g(t))^{\frac{n - 1}{n}}\right) \end{equation*} for all $t \in (a, b)$. Since by definition we have $P_\lambda(E) = P_\lambda(B^\lambda) (1 + D_\lambda(E))$, we obtain
\begin{equation}\label{fmp:5.6}
v_E(t) \ge \frac{1}{2} P_\lambda(B^\lambda)\left(g(t)^{\frac{n - 1}{n}} + (1 - g(t))^{\frac{n - 1}{n}} - 1 - D_\lambda(E)\right). \end{equation}
Let us now define the concave function \begin{equation*} \begin{split} \psi : [0, 1] \to [0, + \infty) \qquad\qquad \psi(t) := t^{\frac{n - 1}{n}} + (1 - t)^{\frac{n - 1}{n}} - 1. \end{split} \end{equation*} Note that $\psi(0) = \psi(1) = 0$ and $\psi$ achieves its maximum at $\psi(1/2) = 2^{1/n} - 1$, hence by concavity
\begin{equation}\label{fmp:5.7} \psi(t) = \psi\left( 2t \frac12 + 0 \right) \ge 2t \psi(1/2) + 0 = 2(2^{1/n} - 1)t \qquad \forall t \in \left[0, \frac{1}{2}\right]. \end{equation}
Recall that by \eqref{eq:zzSmallDeficit} there holds $2D_\lambda(E) < \psi (1/2)$. Let $a < t_1 < t_2 < b$ be such that $g(t_1) = 1 - g(t_2)$ and $\psi(g(t_1)) = \psi(g(t_2)) = 2 D_\lambda(E)$. Then \begin{equation}\label{fmp:5.8}
\psi(g(t)) \ge 2 D_\lambda(E) \qquad \forall t \in (t_1, t_2) \end{equation}
and, by~\eqref{fmp:5.7}, \begin{equation}\label{fmp:5.9} 
g(t_1) = 1 - g(t_2) \le \frac{D_\lambda(E)}{2^{1/n} - 1}. \end{equation}
For any $t_1 \le t \le t_2$ we have \begin{equation}\label{fmp:5.10}
\begin{split} v_E(t) & 
\overset{\eqref{fmp:5.6}}{\ge} \frac{1}{2} P_\lambda(B^\lambda)(\psi(g(t)) - D_\lambda(E))  = \frac{1}{4} P_\lambda(B^\lambda)\psi(g(t)) + \frac{1}{4} P_\lambda(B^\lambda) (\psi(g(t)) - 2 D_\lambda(E)) \\ & \overset{\eqref{fmp:5.8}}{\ge} \frac{n|B^\lambda|}{4}\psi(g(t)). \end{split} \end{equation}
Since $v_E(t) = |B^\lambda| g'(t)$, we have
\begin{equation}\label{fmp:5.11} t_2 - t_1 \overset{\eqref{fmp:5.10}}{\le} \frac{4}{n} \int_{t_1}^{t_2}\frac{g'(t)}{\psi(g(t))}\de t = \frac{4}{n}\int_{g(t_1)}^{g(t_2)}\frac{1}{\psi(s)}\de s \le \frac{4}{n}\int_0^1\frac{1}{\psi(s)}\de s =: \alpha, \end{equation}
for some $\alpha = \alpha(n)>0$.

Let
\begin{equation*} \begin{split} & \tau_1 = \max\left\{t \in (a, t_1] : v_E(t) \le \frac{n |B^\lambda| D_\lambda(E)}{2}\right\}, \\ & \tau_2 = \min \left\{t \in [t_2, b) : v_E(t) \le \frac{n |B^\lambda| D_\lambda(E)}{2}\right\}. \end{split} \end{equation*}
Note that $\tau_1$ and $\tau_2$ are well defined since $v_E$ is continuous and $ v_E(t) \to 0$ as $t\to a$ or $t\to b$; moreover, by~\eqref{fmp:5.8} and~\eqref{fmp:5.10}, $v_E(\tau_1) = v_E(\tau_2) = \frac{n |B^\lambda| D_\lambda(E)}{2}$. Moreover, from~\eqref{fmp:5.9} and by definition of $\tau_1$, we have
\begin{equation*}
t_1 - \tau_1 \le \frac{2}{n |B^\lambda| D_\lambda(E)}\int_{\tau_1}^{t_1}v_E(t)\de t = \frac{2}{n D_\lambda(E)}\int_{\tau_1}^{t_1}g'(t)\de t \le \frac{2g(t_1)}{nD_\lambda(E)} \le \frac{2}{n(2^{1/n} - 1)} ,
\end{equation*}
and an analogous estimate holds for $\tau_2 - t_2$.

We consider the truncation $\tilde{E} := E \cap\{x : \tau_1 < x_1 < \tau_2\}$. From the above estimate and~\eqref{fmp:5.11}, we have that $\tau_2 - \tau_1 < \beta$ for some $\beta=\beta(n)>0$. Moreover by~\eqref{fmp:5.5} and~\eqref{fmp:5.9}, by the definition of $\tau_1$, $\tau_2$, and since $P(E,\{x_1<\tau_1, x_1>\tau_2\})\ge \lambda\hh^{n-1}(\partial^* E \cap \partial H \cap \{x_1<\tau_1, x_1>\tau_2\})$ (see the proof of \cref{remark:positivity}) we can estimate
\begin{equation}\label{fmp:5.12}
|\tilde{E}| \ge \left|B^\lambda\right|\left(1 - 2\frac{D_\lambda(E)}{2^{1/n} - 1}\right), \qquad \qquad 
P_\lambda(\tilde{E}) \le P_\lambda(E) + n|B^\lambda| D_\lambda(E).
\end{equation}
We finally define
\begin{equation*} 
\sigma := \left(\frac{|B^\lambda|}{|\tilde{E}|}\right)^{1/n}, \qquad \qquad E' := \sigma\tilde{E}. \end{equation*}
Clearly, $|E'| = |B^\lambda|$ and by \eqref{fmp:5.12} we get that $E'$ is contained in a strip $\{\tau_1' < x_1 < \tau_2'\}$, with $\tau_2' - \tau_1' \le \sigma (\tau_2-\tau_1) \le l'$, where $l'=l'(n,\lambda)>0$. Let us now show that $E'$ satisfies~\eqref{fmp:5.3} for a suitable constant $C_1=C_1(n,\lambda)>0$ that may change from line to line. To this aim, since we are assuming $D_\lambda(E)$ small by \eqref{eq:zzSmallDeficit}, from~\eqref{fmp:5.12} we get that $1 \le \sigma \le 1 + C_0D_\lambda(E)$, with $C_0 = C_0(n)$. Thus, from \eqref{fmp:5.12} and \eqref{eq:zzSmallDeficit}, we get
\begin{equation*}
\begin{split}
P_\lambda(E') & = \sigma^{n - 1}P_\lambda(\tilde{E}) \le\sigma^{n - 1}(P_\lambda(E) + n|B^\lambda|D_\lambda(E)) \\ &  = \sigma^{n - 1}P_\lambda(B^\lambda)(1 + 2 D_\lambda(E))  \le P_\lambda(B^\lambda)(1 + C_1 D_\lambda(E)). \end{split}
\end{equation*} 
Hence, the second inequality in~\eqref{fmp:5.3} follows. 
To prove the first inequality, let us denote by $B^\lambda(|B^\lambda|, p)$, with $p \in \partial H$, a spherical cap such that $\alpha_\lambda(E') = \frac{|E' \Delta B^\lambda(|B^\lambda|, p)|}{|B^\lambda|}$. 
From the first inequality in~\eqref{fmp:5.12}, recalling that $|E|=|B^\lambda|$, we then get
\begin{equation*} 
\begin{split} 
\alpha_\lambda(E) & \le \frac{|E \Delta B^\lambda(|B^\lambda|, p/\sigma)|}{|B^\lambda|} \\ & \le \frac{|E \Delta \tilde{E}|}{|B^\lambda|} + \frac{|\tilde{E} \Delta B^\lambda(|B^\lambda|/\sigma^n, p/\sigma)|}{|B^\lambda|} + \frac{|B^\lambda(|B^\lambda|/\sigma^n, p/\sigma)\Delta B^\lambda(|B^\lambda|, p/\sigma)|}{|B^\lambda|} \\ & = \frac{|E \setminus \tilde{E}|}{|B^\lambda|} + \frac{\alpha_\lambda(E')}{\sigma^n} + \frac{|B^\lambda(|B^\lambda|) \setminus B^\lambda(|B^\lambda|/\sigma^n)|}{|B^\lambda|} \\ & \overset{\eqref{fmp:5.12}}{\le} C_1 D_\lambda(E) + \alpha_\lambda(E') + C_1 (\sigma - 1) \\ & \le \alpha_\lambda(E') + C_1 D_\lambda(E). \end{split} \end{equation*}
Thus the set $E'$ satisfies~\eqref{fmp:5.3} and points in $E'$ have first coordinate contained in an interval of length bounded by $l'$.

Starting from $E'$, we can repeat the same construction finitely many times with respect to the axes $x_2, \ldots, x_{n-1}$, thus getting a new set, still denoted by $E'$, satisfying \eqref{fmp:5.3}.

\medskip

It remains to adapt the construction with respect to the coordinate axis $x_n$. In this case we eventually aim at truncating the set $E'$ in some controlled slab of the form $\{0<x_n<\bar \tau_2\}$. Define
\[
\bar v(t) := \mathcal{H}^{n - 1}(\{x' \in \R^{n - 1} : (x', t) \in E'\}) ,
\qquad
\bar E_t^- := \{x \in E' : x_n < t\},
\qquad
 \bar g(t) := \frac{|\bar E_t^-|}{|B^\lambda|},
\]
for $t>0$. It is readily checked that, arguing as above, one estimates
\begin{equation}\label{fmp:5.6'}
\bar v(t) \ge \frac{1}{2} P_\lambda(B^\lambda)\left(\psi(\bar g(t)) - D_\lambda(E)\right),
\end{equation}
which is analogous to \eqref{fmp:5.6}, for any $t$ such that $\bar g(t) \in(0,1)$. Similarly as before, we define $0<\bar t_1<\bar t_2$ such that $\bar g(\bar t_1) = 1- \bar g(\bar t_2)$ and $\psi(\bar g(\bar t_1)) = \psi(\bar g(\bar t_2)) = 2 D_\lambda(E')$. Therefore, using \eqref{fmp:5.6'} and the concavity of $\psi$, arguing as before one estimates
\begin{equation}\label{fmp:5.10'}
    \bar v(t) \ge \frac{n|B^\lambda|}{4}\psi(\bar g(t)) \qquad\qquad\forall\,t\in[\bar t_1,\bar t_2],
\end{equation}
which is analogous to \eqref{fmp:5.10}.

Let
\begin{equation*}
    A:= \frac12 P_\lambda(B^\lambda) \left(  \left( \frac{\omega_n}{|B^\lambda|} \right)^{\frac1n} -1 \right)>0.
\end{equation*}
We claim that there exists $\bar{\epsilon}=\bar{\epsilon}(n,\lambda)>0$ such that if $D_\lambda(E)<\bar\epsilon$ then
\begin{equation}\label{eq:ClaimFette}
    \bar v(t) \ge \frac{A}{2} \qquad \text{ for a.e. $t \in (0,\bar t_1)$.}
\end{equation}
Indeed, since $D_\lambda(E') \le C_1 D_\lambda(E)$ and $\psi(\bar g(\bar t_1))= 2 D_\lambda(E')$, then for any $\omega>0$ there is $\bar{\epsilon}=\bar{\epsilon}(n,\lambda)>0$ such that $\bar g(\bar t_1)<\omega$ whenever $D_\lambda(E)<\bar\epsilon$. Applying \cref{lem:LowerBoundAreaFette} with $F=E'$, for almost every $t\in(0,\bar t_1)$ we find
\begin{equation}
\begin{split}
    \bar v(t) &\ge \frac12 P_\lambda(B^\lambda) \left(  \left( \frac{\omega_n}{|B^\lambda|} \right)^{\frac1n} \left( 1- \bar g(t) \right)^{\frac{n-1}{n}} -1 - D_\lambda(E') \right) \\
    &\ge \frac12 P_\lambda(B^\lambda) \left(  \left( \frac{\omega_n}{|B^\lambda|} \right)^{\frac1n} \left( 1- \bar g(\bar t_1) \right)^{\frac{n-1}{n}} -1 - D_\lambda(E') \right) \\
    &\ge \frac12 P_\lambda(B^\lambda) \left(  \left( \frac{\omega_n}{|B^\lambda|} \right)^{\frac1n} \left( 1- \omega \right)^{\frac{n-1}{n}} -1 - C_1\bar{\epsilon} \right) \ge \frac{A}{2}, 
\end{split}
\end{equation}
provided $\bar{\epsilon}$ is small enough.

Therefore, assuming without loss of generality that $D_\lambda(E)<\bar\epsilon$, since $\bar g'(t) = \bar v(t)/|B^\lambda|$ we estimate
\begin{equation}\label{eq:barT2Bounded}
    \begin{split}
        \bar t_2 \overset{\eqref{fmp:5.10'}}{\le} \bar t_1 + \frac{4}{n} \int_{\bar t_1}^{\bar t_2} \frac{\bar g'(t)}{\psi(\bar g(t))} \de t
        \overset{\eqref{eq:ClaimFette}}{\le} \frac2A\int_0^{\bar t_1} \bar v(t) \de t + \frac4n \int_0^1 \frac{1}{\psi} \le \frac{2 |B^\lambda|}{A} + \frac4n \int_0^1 \frac{1}{\psi} =: \alpha'(n,\lambda).
    \end{split}
\end{equation}
The rest of the construction follows analogously as above by defining
\begin{equation*} 
\bar\tau_2 := \min \left\{t \ge \bar t_2 \st \bar g(t)<1, \,\, \bar v(t) \le \frac{n |B^\lambda| D_\lambda(E')}{2}\right\},
\end{equation*}
estimating $\bar \tau_2 - \bar t_2 \le \beta'(n,\lambda)$, hence finally taking the set
\begin{equation}\label{eq:zzzz}
    \left( \frac{|B^\lambda|}{| E' \cap \{x_n < \bar \tau_2\}|} \right)^{\frac1n} \, ( E' \cap \{x_n < \bar \tau_2\}).
\end{equation}
Up to translation along $\partial H$, the set defined in \eqref{eq:zzzz} yields the final one satisfying the claim of the lemma.
\end{proof}

\begin{corollary}[Non-quantitative stability]\label{lemma:fmp2.3}
For any $\bar\epsilon > 0$ there exists $\bar\delta = \bar\delta(n, \lambda, \bar\epsilon) > 0$ such that if $E \subset \R^n \setminus H$ is a Borel set such that $D_\lambda(E) \le \bar\delta$, then $\alpha_\lambda(E) \le \bar\epsilon$. 
\end{corollary}

\begin{proof} By scale-invariance of the asymmetry and of the deficit, it is sufficient to prove the claim assuming also $|E| = |B^\lambda|$. 
We argue by contradiction. 
Suppose there exist a number $\bar\epsilon > 0$ and a sequence of sets $\{E_i\}_i$, with $E_i \subset \R^n \setminus H$ and $|E_i| = |B^\lambda|$, such that $D_\lambda(E_i) < \frac{1}{i}$ and $\alpha_\lambda(E_i) > \bar\epsilon$ for all $i \in \N$. 
Let us consider the sequence of sets $\{E_i'\}_i$, with $E_i' \subset Q_l \cap(\R^n \setminus H)$ and $|E_i'| = |B^\lambda|$, given by \cref{lem:ReductionToBounded}. 
Moreover the Lemma assures that $\alpha_\lambda(E_i') > \bar\epsilon/2$ for large $i$, and $D_\lambda(E_i') \to 0$. 
Since each set $E_i'$ is contained in the same $Q_l$, by \cref{lemma:precompattezza} we can assume, up to a subsequence, that $E_i' \xrightarrow{L^1} E'$ for some set $E'$ of finite perimeter with $|E'| = |B^\lambda|$. 
By the lower semicontinuity of the perimeters we get $P_\lambda(E') \le P_\lambda(B^\lambda)$, hence $E'=B^\lambda(|B^\lambda|,x)$ for some $x \in \partial H \cap Q_l$ by uniqueness of minimizers. The convergence of $E_i'$ to $E'$ implies that $|E_i' \Delta E'| \to 0$, against the assumption $\alpha_\lambda(E_i') > \frac{\bar\epsilon}{2}$.
\end{proof}

\begin{corollary}\label{cor:DeficitBassoFetteGrosse}
There exist $A_\lambda, T_\lambda, \eta>0$ depending on $n,\lambda$ such that for any set of finite perimeter $E\subset \R^n\setminus H$ with $\left||E|-|B^\lambda|\right| \le \eta$ and $D_\lambda(E)\le\eta$ there holds
\begin{equation*}
    \hh^{n-1}(E \cap \{x_n=t\}) \ge A_\lambda,
\end{equation*}
for almost every $t \in (0,T_\lambda)$.
\end{corollary}

\begin{proof}
Let us prove the inequality assuming $|E|=|B^\lambda|$ first.
Fix $T_\lambda'>0$ such that
\[
 \left( \frac{\omega_n}{|B^\lambda|} \right)^{\frac1n} \left( 1- \frac{|B^\lambda(|B^\lambda|,0) \cap \{x_n<T_\lambda'\}|}{|B^\lambda|} \right)^{\frac{n-1}{n}} \ge 1+ a ,
\]
for some $a>0$.
By \cref{lemma:fmp2.3}, for any $\omega>0$ there is $\eta>0$ such that if $D_\lambda(E)<\eta$ then $|E \cap \{x_n<T_\lambda'\}| \le |B^\lambda(|B^\lambda|,0) \cap \{x_n<T_\lambda'\}|+\omega$. For such a set $E$, \cref{lem:LowerBoundAreaFette} implies that for almost every $t\in(0,T_\lambda')$ there holds
\begin{equation*}
\begin{split}
    \hh^{n-1}( E \cap \{x_n=t\}) &\ge \frac12 P_\lambda(B^\lambda) \left(  \left( \frac{\omega_n}{|B^\lambda|} \right)^{\frac1n} \left( 1- \frac{|E \cap \{x_n<t\}|}{|B^\lambda|} \right)^{\frac{n-1}{n}} -1 - D_\lambda(E) \right) \\
    &\ge 
    \frac12 P_\lambda(B^\lambda) \left(  \left( \frac{\omega_n}{|B^\lambda|} \right)^{\frac1n} \left( 1- \frac{|B^\lambda(|B^\lambda|,0) \cap \{x_n<T_\lambda'\}|+\omega}{|B^\lambda|} \right)^{\frac{n-1}{n}} -1 -\eta \right).
\end{split}
\end{equation*}
Hence for sufficiently small $\eta>0$ the right hand side in the previous estimate is bounded below by some constant $A_\lambda'(n,\lambda)>0$.\\
For a generic set $E$ such that $\left||E|-|B^\lambda|\right| \le \eta$ and $D_\lambda(E)\le\eta$, the set $E'=\left( |B^\lambda|^{\frac1n}/|E|^{\frac{1}{n}}\right) E$ has measure equal to $|B^\lambda|$ and deficit $D_\lambda(E')\le \eta$. Up to decreasing $\eta>0$, applying the first part of the proof to $E'$, the desired estimate holds on $E$ for $T_\lambda=T_\lambda'/2$ and $A_\lambda=A_\lambda'/2$.
\end{proof}

\subsection{Reduction to \texorpdfstring{$(n - 1)$}{(n - 1)}-symmetric sets}

In this section we prove that, in order to prove \cref{thm:FinalQuantitativeInequality}, it is sufficient to further reduce to show \eqref{eq:FinalQuantitativeInequality} among $(n-1)$-symmetric sets, i.e., sets which are symmetric with respect to reflection across $n-1$ orthogonal hyperplanes, each one orthogonal to $\{x_n=0\}$. The results are analogous to \cite[Section 6]{MaggiMethodsQuantitative08}.

\begin{lemma}\label{lemma:fmp2.2}
Let $E \subset \R^n \setminus H$ be a Borel set with finite measure, symmetric with respect to $k \in \{1, \dots, n - 1\}$ orthogonal half-hyperplanes $H_j = \left\{x \in \R^n \setminus H : \Braket{x, \nu_j} = 0\right\}$ for $1 \le j \le k$, where $|\nu_j|=1$ and $\Braket{\nu_j, e_n} = 0$ for any $1\le j \le k$. Then
\begin{equation}\label{dispense:4.16} \min_{x \in \partial H} |E \Delta B^\lambda(|E|, x)| \le \min_{y \in \partial H \cap\bigcap_{j = 1}^k H_j} |E \Delta B^\lambda(|E|, y)| \le 3 \min_{x \in \partial H} |E \Delta B^\lambda(|E|, x)|. \end{equation} \end{lemma}

\begin{proof}
We can suppose for simplicity that $\forall j \in \{1, \dots, k\}$ we have $\nu_j = e_j$. If $x^0 = (x^0_1, \dots, x^0_{n - 1}, 0)$ is such that $\alpha_\lambda(E)$ is achieved by $B^\lambda(|E|, x^0)$, then $\alpha_\lambda(E)$ is achieved also by $B^\lambda(|E|, \bar x^0)$, where $\bar x^0 = (- x^0_1, \dots, - x_k^0, x^0_{k + 1}, \dots,$ $x^0_{n - 1}, 0)$. Since $|B^\lambda(|E|, x^0) \Delta B^\lambda(|E|)| \le |B^\lambda(|E|, x^0) \Delta B^\lambda(|E|, \bar x^0)|$ we have \begin{equation*} \begin{split} |E \Delta B^\lambda(|E|)| & \le |E \Delta B^\lambda\left(|E|, x^0\right)| + |B^\lambda\left(|E|, x^0\right) \Delta B^\lambda(|E|)| \\ & \le |E \Delta B^\lambda\left(|E|, x^0\right)| + |B^\lambda\left(|E|, x^0\right) \Delta B^\lambda\left(|E|, \bar x^0\right)| \\ & \le |E \Delta B^\lambda\left(|E|, x^0\right)| + |B^\lambda\left(|E|, x^0\right) \Delta E| + |E \Delta B^\lambda\left(|E|, \bar x^0\right)| \\ & = 3 |E \Delta B^\lambda\left(|E|, x^0\right)|. \qedhere \end{split} \end{equation*} \end{proof}

Given a Borel set $E \subset \R^n \setminus H$ with finite measure and a unit vector $\nu$ with $\Braket{\nu, e_n} = 0$, we denote by $H^+_\nu = \{x \in \R^n : \Braket{x, \nu} > t\}$ an open half-space orthogonal to $\nu$ where $t \in \R$ is chosen in such a way that \begin{equation*} |E \cap H_\nu^+| = \frac{|E|}{2}. \end{equation*}
We also denote by $r_\nu : \R^n \setminus H \to \R^n \setminus H$ the reflection with respect to $H_\nu := \partial H_\nu^+$, and by $H_\nu^- := r_\nu(H_\nu^+)$ the open half-space complementary to $H_\nu^+$. Finally we write $E_\nu^{\pm} := E \cap H_\nu^\pm$. 

Observe that \begin{equation}\label{dispense:4.19} D_\lambda(E^\pm_\nu \cup r_\nu(E_\nu^\pm)) \le 2 D_\lambda(E). \end{equation}
Indeed \begin{equation*} \begin{split} P_\lambda(E^\pm_\nu\cup r_\nu(E_\nu^\pm)) - P_\lambda(B^\lambda(|E|)) & \le 2 P_\lambda(E) - P_\lambda(E^\mp_\nu \cup r_\nu(E_\nu^\mp)) - P_\lambda(B^\lambda(|E|)) \\ & = 2 (P_\lambda(E) - P_\lambda(B^\lambda(|E|))) + P_\lambda(B^\lambda(|E|)) - P_\lambda(E_\nu^\mp \cup r_\nu(E_\nu^\mp)) \\ & \le 2 (P_\lambda(E) - P_\lambda(B^\lambda(|E|))), \end{split} \end{equation*}
where in the last inequality we used the isoperimetric inequality of \cref{thm:IsopIneq}.

\begin{lemma}\label{lemma:fmp2.5}
There exist $\bar C_2, \bar\delta_2>0$ depending on $n, \lambda$ such that, if $E \subset \R^n \setminus H$ is a Borel set with finite measure such that $D_\lambda(E) \le \bar\delta_2$, and if $\nu_1$ and $\nu_2$ are two orthogonal vectors, with $\Braket{\nu_i, e_n} = 0$, such that $H_{\nu_1}$ and $H_{\nu_2}$ divide $E$ in four parts of equal measure, then there exist $i \in \{1, 2\}$ and $s \in \{+, -\}$ such that, setting $E' = E^s_{\nu_i} \cup r_{\nu_i}(E^s_{\nu_i})$, there holds \begin{equation}\label{fmp:2.2} \alpha_\lambda(E) \le \bar C_2 \alpha(E'). \end{equation}
\end{lemma}

\begin{proof} By scale-invariance of Fraenkel asymmetry, it is sufficient to prove the claim assuming also $|E| = |B^\lambda|$. If $i \in \{1, 2\}$ and $s \in \{+, -\}$, let $E^{'s}_{\nu_i}$ denote the sets obtained by reflecting $E^s_{\nu_i}$ along $H_{\nu_i}$ and let $B_i^{\lambda, s} = B^\lambda(|B^\lambda|, x_i^{s})$ be four spherical caps such that
\begin{equation*} |E_{\nu_i}^{'s} \Delta B_i^{\lambda, s}| = \min_{x \in H_{\nu_i} \cap \partial H} |E_{\nu_i}^{' s} \Delta B^\lambda(|B^\lambda|, x)|. \end{equation*} 
For $i = 1, 2$, by the triangular inequality we have \begin{equation}\label{dispense:4.22} \begin{split} \min_{x \in \partial H} |E \Delta B^\lambda(|B^\lambda|, x)| & \le |E \Delta B_i^{\lambda, +}| \\ & = |(E \Delta B_i^{\lambda, +}) \cap H_{\nu_i}^+| + |(E \Delta B_i^{\lambda, +}) \cap H_{\nu_i}^-| \\ & \le |(E \Delta B_i^{\lambda, +}) \cap H^+_{\nu_i}| + |(E \Delta B_i^{\lambda, -}) \cap H^-_{\nu_i}| + |(B_i^{\lambda, +} \Delta B_i^{\lambda, -}) \cap H_{\nu_i}^-| \\ & = \frac{1}{2} |E^{'+}_{\nu_i} \Delta B_i^{\lambda, +}| + \frac{1}{2} |E_{\nu_i}^{'-} \Delta B_i^{\lambda, -}| + \frac{1}{2} |B_i^{\lambda, +} \Delta B_i^{\lambda, -}|. \end{split} \end{equation} 

Once we show that if $D_\lambda(E)$ is sufficiently small then there exists $c_{n,\lambda}>0$ such that at least one the following \begin{equation}\label{dispense:4.23} \begin{split} & |B_1^{\lambda, +} \Delta B_1^{\lambda, -}| \le 2c_{n,\lambda} \Big(|E_{\nu_1}^{'+} \Delta B_1^{\lambda, +}| + |E_{\nu_1}^{'-} \Delta B_1^{\lambda, -}|\Big) \\ 
& |B_2^{\lambda, +} \Delta B_2^{\lambda, -}| \le 2c_{n,\lambda} \Big(|E_{\nu_2}^{'+} \Delta B_2^{\lambda, +}| + |E_{\nu_2}^{'-} \Delta B_2^{\lambda, -}|\Big) \end{split} \end{equation}
holds, then we soon conclude the proof. 
Indeed, assume for example that the first inequality in~\eqref{dispense:4.23} holds. 
Then, from~\eqref{dispense:4.16} and~\eqref{dispense:4.22} with $i = 1$, we get
\begin{equation*}
\begin{split}
\min_{x \in \partial H} |E \Delta B^\lambda(|B^\lambda|, x)| &\overset{\eqref{dispense:4.22}}{\le} (c_{n,\lambda} + 1/2) \Big(|E_{\nu_1}^{'+} \Delta B_1^{\lambda, +}| + |E_{\nu_1}^{'-} \Delta B_1^{\lambda, -}|\Big) \\&
\overset{\eqref{dispense:4.16}}{\le} 3(c_{n,\lambda} + 1/2) \Big(\min_{x \in \partial H} |E_{\nu_1}^{'+} \Delta B^\lambda(|B^\lambda|, x)| + \min_{x \in \partial H} |E_{\nu_1}^{'-} \Delta B^\lambda(|B^\lambda|, x)|\Big),     
\end{split}
\end{equation*}
thus proving~\eqref{fmp:2.2} with $\bar C_2 = 6(c_{n,\lambda} + 1/2)$ and $E'$ equal to $E_{\nu_1}^{'+}$ or $E_{\nu_1}^{'-}$.

Observe that, given $\tilde\epsilon > 0$, \cref{lemma:fmp2.3},~\eqref{dispense:4.16} and~\eqref{dispense:4.19} imply that there exists $\bar\delta_2(n, \lambda) > 0$ such that if $D_\lambda(E) < \bar\delta_2$ then \begin{equation}\label{maggi:6.10} \max\left\{\alpha_\lambda(E), \frac{|E_{\nu_i}^{'\pm} \Delta B_i^{\lambda, \pm}|}{|B^\lambda|} \st i = 1, 2\right\} < \tilde\epsilon.
\end{equation} 
Thanks to~\eqref{maggi:6.10}, we can show that the caps $B_i^{\lambda, \pm}$ get closer and closer to the optimal ones for $E$, as $\bar{\delta}_2$ decreases.

Indeed, let us assume by contradiction that there exists $\eta > 0$ such that for every $j \in \N$ there exist $E_j$, with $|E_j|=|B^\lambda|$, $D_\lambda(E_j) < \frac{1}{j}$, with $B^\lambda_j:= B^\lambda(|B^\lambda|,x_j)$ realizing the asymmetry of $E_j$, but for $i \in \{1, 2\}$ and $s \in \{+, -\}$, if $B^{\lambda, s}_{i, j}:=B^\lambda(|B^\lambda|, x_{i,j}^s)$ is such that
\[
|E'^{s}_{j, \nu_i^j} \Delta B^{\lambda, s}_{i, j}| = \min_{x \in H_{\nu_i^j} \cap \partial H} |E'^{s}_{j, \nu_i^j} \Delta B^{\lambda}(|B^\lambda|, x)|, 
\]
where $E'^{s}_{j, \nu_i^j}$ is given by reflections of truncations of $E_j$ along orthogonal subspaces $H_{\nu_1^j}, H_{\nu_2^j}$, then $|x_{i, j}^s- x_j| > \eta$ for some $i\in\{1,2\}, s \in \{+,-\}$ and any $j$.
Without loss of generality we can assume that that $i = 1$ and $s = +$.\\
Let us translate every set in the above contradiction assumption by $-x_j$. Without relabeling the objects involved, up to subsequences, we have that $E_j \to B^\lambda_0:= B^\lambda(|B^\lambda|, 0)$ in $L^1$. We can show that $E'^+_{j, \nu_1^j} \to B^\lambda_0$ as well.
Indeed, up to a rotation we can further assume that
\begin{equation*} H_{\nu_1^j} = \left\{(a_j, 0, \dots, 0) + e_1^\perp\right\} , \qquad \nu_1^j = e_1 ,\qquad\forall\,j,
\end{equation*}
with $a_j \in \R$.
By the definition of $H_{\nu_1^j}$, we have
\begin{equation*}
\frac{|B^\lambda|}{2} = \frac{|E_j|}{2} = |E_j \cap \{x_1 > a_j\}|.
\end{equation*}
Then $\{a_j\}$ is bounded and, up to subsequence, converges to $a_\infty = 0$ because, if $a_\infty \neq 0$, then \begin{equation*} E_j \cap \{x_1 > a_j\} \to B_0^\lambda \cap \{x_1 > a_\infty\} \end{equation*} with $\left|B_0^\lambda \cap \{x_1 > a_\infty\}\right| \neq \frac{|B^\lambda|}{2}$.
In particular $E_{j, \nu_1^j}'^+ \to B^\lambda_0$.
Finally by~\eqref{maggi:6.10} \begin{equation*} \tilde \epsilon |B^\lambda| \ge \lim_j \left|E'^+_{j, \nu_1^j} \Delta B_{1, j}^{\lambda, +}\right| \ge \min_{x\in\partial H, |x|\ge \eta} \left\{|B^\lambda(|B^\lambda|, x) \Delta B^\lambda_0\right\} =:C(\eta) >0. \end{equation*}
But for $j$ sufficiently large, since $D_\lambda(E_j)\to0$, by \eqref{maggi:6.10} we can choose $\tilde \epsilon$ such that $\tilde \epsilon |B^\lambda| < C(\eta)/2$, getting a contradiction.

{
We observe that for $\tilde\epsilon > 0$ sufficiently small, that is for $\bar\delta_2 > 0$ sufficiently small, there exists $c_{n,\lambda}>0$ such that for all possible choices of $s,t\in \{+, -\}$  there holds
\begin{equation}\label{eq:DiffSimmSplittataUniforme}
|(B_1^{\lambda, s} \Delta B_2^{\lambda, t}) \cap (H_{\nu_1}^s \cap H_{\nu_2}^t)| > \frac{|B_1^{\lambda,s} \Delta B_2^{\lambda, t}|}{c_{n,\lambda}}.
\end{equation}
We only sketch the argument for \eqref{eq:DiffSimmSplittataUniforme}. Letting $Q:= (H_{\nu_1}^s \cap H_{\nu_2}^t)$ and $B_1(h):=B^{\lambda}(|B^\lambda|, h\, x_1^s), B_2(h):=B^{\lambda}(|B^\lambda|, h\, x_2^t)$ for $h\in[0,1]$, one can compute
\begin{equation*}
\begin{split}    
    \frac{\d}{\d h} \big|( B_1(h)& \Delta B_2(h)) \cap Q\big| = \int_{\partial B_1(h) \cap B_2(h) \cap  Q} \Braket{\nu^{B_1(h)}, x_2^t -x_1^s} \de \hh^{n-1} 
    + \int_{\partial B_2(h) \cap B_1(h) \cap  Q} \Braket{\nu^{B_2(h)}, x_1^s -x_2^t} \de \hh^{n-1} \\
    & = \left( \int_{\partial B_1(h) \cap B_2(h) \cap  Q} \left\langle\nu^{B_1(h)},\frac{x_2^t -x_1^s}{|x_1^s -x_2^t|} \right\rangle
    + \int_{\partial B_2(h) \cap B_1(h) \cap  Q} \left\langle\nu^{B_2(h)}, \frac{x_1^s -x_2^t}{|x_1^s -x_2^t|}  \right\rangle \right) |x_1^s -x_2^t| \\
    &=\int_{\partial (B_1(h) \cap B_2(h)) \cap  Q} \left\langle\nu^{B_1(h) \cap B_2(h)}, v_{12}^{s,t} \right\rangle \de \hh^{n-1} \, |x_1^s -x_2^t|
    \\
    &\ge c |x_1^s -x_2^t|,
\end{split}
\end{equation*}
where $v_{12}^{s,t}$ is obviously defined, provided $\tilde\epsilon$ is small enough, for some $c=c(n,\lambda)>0$ that will change from line to line. 
The last estimate follows since $\left\langle\nu^{B_1(h) \cap B_2(h)}, v_{12}^{s,t} \right\rangle\ge0$ pointwise and, for $\tilde\epsilon$ small, centers $x_1^s,x_2^t$ are so close that $\left\langle\nu^{B_1(h) \cap B_2(h)}, v_{12}^{s,t} \right\rangle$ can be estimated from below by a positive constant on a set of $\hh^{n-1}$-measure uniformly bounded from below away from zero. On the other hand one can estimate
\[
 |(B_1^{\lambda, s} \Delta B_2^{\lambda, t}) | \le c |x_1^s -x_2^t| .
\]
Hence
\[
|(B_1^{\lambda, s} \Delta B_2^{\lambda, t}) \cap (H_{\nu_1}^s \cap H_{\nu_2}^t)| = \int_0^1  \frac{\d}{\d h} \big|( B_1(h)\Delta B_2(h)) \cap Q\big| \de h \ge  c |x_1^s -x_2^t| \ge c |(B_1^{\lambda, s} \Delta B_2^{\lambda, t}) | ,
\]
and \eqref{eq:DiffSimmSplittataUniforme} follows.
}

Letting
\begin{equation*} S_1 = (B_1^{\lambda, +} \cap H_{\nu_1}^+) \cup (B_1^{\lambda, -} \cap H_{\nu_1}^-), \quad S_2 = (B_2^{\lambda, +} \cap H_{\nu_2}^+) \cup (B_2^{\lambda, -} \cap H_{\nu_2}^-),
\end{equation*} 
we deduce
\begin{equation*} |S_1 \Delta S_2| \ge |(S_1 \Delta S_2) \cap (H_{\nu_1}^s \cap H_{\nu_2}^t)| = |(B_1^{\lambda, s} \Delta B_2^{\lambda, t}) \cap (H_{\nu_1}^s \cap H_{\nu_2}^t)| > \frac{|B_1^{\lambda, s} \Delta B_2^{\lambda, t}|}{c_{n,\lambda}}. \end{equation*} 
In particular we have
\begin{equation}\label{eq:zzaw}
\begin{split}
    &|B_1^{\lambda, +} \Delta B_1^{\lambda, -}| \le |B_1^{\lambda, +} \Delta B_2^{\lambda, +}| + |B_2^{\lambda, +} \Delta B_1^{\lambda, -}| < 2 c_{n,\lambda} |S_1 \Delta S_2|,\\
    &|B_2^{\lambda, +} \Delta B_2^{\lambda, -}| \le |B_2^{\lambda, +} \Delta B_1^{\lambda, +}| + |B_1^{\lambda, +} \Delta B_2^{\lambda, -}| < 2 c_{n,\lambda} |S_1 \Delta S_2|.
\end{split}
 \end{equation}

If by contradiction \eqref{dispense:4.23} were false, then
\begin{equation}\label{dispense:4.24} |E_{\nu_1}^{'+} \Delta B_1^{\lambda, +}| + |E_{\nu_1}^{'-} \Delta B_1^{\lambda, -}| < \frac{|B_1^{\lambda, +} \Delta B_1^{\lambda, -}|}{2c_{n,\lambda}} \quad \text{and} \quad  |E_{\nu_2}^{'+} \Delta B_2^{\lambda, +}| + |E_{\nu_2}^{'-} \Delta B_2^{\lambda, -}| < \frac{|B_2^{\lambda, +} \Delta B_2^{\lambda, -}|}{2c_{n,\lambda}}.
\end{equation}
Hence
\begin{equation*}
\begin{split}
    |S_1 \Delta S_2| &\le |S_1 \Delta E| + |E \Delta S_2| = \frac{1}{2} \sum_{i = 1}^2 \left(|E_{\nu_i}^{'+} \Delta B_i^{\lambda, +}| + |E_{\nu_i}^{'-} \Delta B_i^{\lambda, -}|\right) \overset{\eqref{dispense:4.24}}{<} \frac{1}{4 c_{n,\lambda}} \sum_{i = 1}^2|B_i^{\lambda, +} \Delta B_i^{\lambda, -}| \overset{\eqref{eq:zzaw}}{\le} |S_1 \Delta S_2|,
\end{split}
\end{equation*} 
getting a contradiction.
\end{proof}

\begin{lemma}[Reduction to $(n-1)$-symmetric sets]\label{theorem:fmp2.1}
There exist $C_2, \delta_2>0$ depending on $n, \lambda$ such that, if $E$ is a Borel set with $E \subset \R^n \setminus H$, $E \subset Q_l$, $|E| = |B^\lambda|$ and $D_\lambda(E) \le \delta_2$, there exists a Borel set $F \subset \R^n \setminus H$, $F \subset Q_{2l}$, $|F| = |B^\lambda|$, symmetric with respect to $n - 1$ orthogonal half-hyperplanes (each orthogonal to $\partial H$) and such that \begin{equation*} \alpha_\lambda(E) \le C_2 \alpha_\lambda(F), \qquad D_\lambda(F) \le 2^{n-1} D_\lambda(E). \end{equation*} \end{lemma}

\begin{proof}
Let us define $\delta_2 :=\bar \delta_2  2^{-(n - 2)}$, where $\bar\delta_2$ is the constant appearing in \cref{lemma:fmp2.5}. 
We can apply \cref{lemma:fmp2.5} $n - 2$ times to different pairs of orthogonal vectors in $\{e_1, \ldots, e_{n-2}\}$ normal to corresponding pairs of affine hyperplanes splitting the measure of $E$ in two halves. Therefore, also recalling~\eqref{dispense:4.19}, we find an $(n-2)$-symmetric set $E'$ such that $|E'| = |B^\lambda|$ and \begin{equation*} \alpha_\lambda(E) \le \bar C_2^{n - 2} \alpha_\lambda(E'), \quad D_\lambda(E') \le 2^{n - 2}D_\lambda(E). \end{equation*} 
To perform the last symmetrization, let us consider a half-hyperplane $H_{n - 1}$ orthogonal to $e_{n - 1}$ and dividing $E'$ into two parts of equal measure. For simplicity let us assume that $H_{n - 1} = \{x_{n - 1} = 0\} \setminus H$. We denote by $E^{'+}$ (resp. $E^{'-}$) the set obtained by the union of $E' \cap\{ x_{n-1}>0\}$ (resp. $E' \cap\{ x_{n-1}<0\}$) with its reflection along $H_{n-1}$.
By \eqref{dispense:4.19} we have \begin{equation*} D_\lambda(E^{'\pm}) \le 2 D_\lambda(E') \le 2^{n - 1}D_\lambda(E). \end{equation*} 
Regarding the asymmetry of $E^{' \pm}$ note that since $E'$ is symmetric with respect to the first $n - 2$ coordinate hyperplanes, $E^{'+}$ and $E^{'-}$ are $(n - 1)$-symmetric. 
By \cref{lemma:fmp2.2} we get \begin{equation*} \begin{split} |B^\lambda| \alpha_\lambda(E') & \le |E' \Delta B^\lambda(|B^\lambda|)| \\ & = |(E' \Delta B^\lambda(|B^\lambda|)) \cap \{x_{n - 1} > 0\}| + |(E' \Delta B^\lambda(|B^\lambda|)) \cap \{x_{n - 1} < 0\}| \\ & = \frac{1}{2}\Big(|E^{'+} \Delta B^\lambda(|B^\lambda|)| + |E^{'-} \Delta B^\lambda(|B^\lambda|)|\Big) \\ & \le \frac{3 |B^\lambda|}{2}\Big(\alpha_\lambda(E^{'+}) + \alpha_\lambda(E^{'-})\Big).
\end{split} \end{equation*} 
Therefore at least one of the sets $E^{'+}, E^{'-}$ has asymmetry greater than $\frac{1}{3} \alpha_\lambda(E')$ and, denoting by $F$ this set, we have \begin{equation*} \begin{split} & D_\lambda(F) \le 2 D_\lambda(E') \le 2^{n - 1} D_\lambda(E) \\ & \alpha_\lambda(E) \le \bar C_2^{n - 2} \alpha_\lambda(E') \le 3 \bar C_2^{n - 2}\alpha_\lambda(F). \end{split} \end{equation*} 
Finally, the inclusion $F \subset Q_{2l}$ follows since $F$ was obtained by performing reflections of $E\subset Q_l$ along affine hyperplanes of the form $\{x_j=a_j\}$ for $j=1,\ldots,n-1$ with $a_j \in (-l,l)$.
\end{proof}

\subsection{Reduction to Schwarz-symmetric sets}

In this section we observe that, in order to prove \cref{thm:FinalQuantitativeInequality}, it is sufficient to further reduce to show \eqref{eq:FinalQuantitativeInequality} just among Schwarz-symmetric sets. The proof is analogous to \cite[Proposition 4.9]{FuscoDispensa}.

\begin{lemma}[Reduction to Schwarz-symmetric sets] \label{lem:Schwarz}
There exists $C_3=C_3(n, \lambda)>0$ such that the following holds.
Let $E\subset \R^n\setminus H$ be a set of finite perimeter with $|E| = |B^\lambda|$. Suppose that $E$ is symmetric with respect to the coordinate hyerplanes $\{x_1=0\}, \ldots, \{x_{n-1}=0\}$ and that
\begin{equation*} D_\lambda(E) < 1, \qquad E \subset Q_{2l}, \end{equation*}
where $l = l(n, \lambda)$ is as in in \cref{lem:ReductionToBounded}. 
Then
\begin{equation}\label{dispense:4.27}
|E \Delta E^*| \le C_3 \sqrt{D_\lambda(E)} \qquad \text{and} \qquad D_\lambda(E^*) \le D_\lambda(E), \end{equation}
where $E^*$ denotes the Schwarz symmetrization of $E$ with respect to the $n$-th axis.
\end{lemma}

\begin{proof}
The second inequality in~\eqref{dispense:4.27} follows from the fact $|E^*| = |E|$ and $P_\lambda(E^*)$ $\le$ $P_\lambda(E)$. 
Exploiting Lemma \ref{lemma:Density}, we may assume \begin{equation}\label{dispense:4.28} \hh^{n - 1}(\{x \in \partial^* E \setminus H \st \nu^E(x) = \pm e_n\}) = 0 ,\end{equation}
and thus that $v_E \in W^{1,1}(\R)$. 
Indeed, if $E_i$ is given by \cref{lemma:Density} and the claim is proved for $E_i$, by the contractivity of Schwartz rearrangement (\cite[Exercise 19.14]{MaggiBook}), for every $\tilde\epsilon > 0$ and with $i$ sufficiently large, we get \begin{equation*} \begin{split} |E \Delta E^*| & \le |E \Delta E_i| + |E_i \Delta E_i^*| + |E_i^* \Delta E^*| 
\le 2|E \Delta E_i| + |E_i \Delta E_i^*| 
\le 2 \tilde\epsilon + C_3\sqrt{D_\lambda(E_i)}, \end{split} \end{equation*}
and the last term tends to $C_3\sqrt{D_\lambda(E)}$ as $i\to\infty$.

For $\hh^1$-a.e. $t \in (0, \infty)$ denote
\begin{equation*} \begin{split} & v_E(t) := \hh^{n - 1}(\{x' \in \R^{n - 1} \st (x', t) \in E\}), \\ & p_E(t) := \hh^{n - 2}(\partial^*\{x' \in \R^{n - 1} \st (x', t) \in E\}),
\end{split} \end{equation*}
and employ analogous notation for $E^*$.
Since $|\partial^*E \cap \partial H| = |\partial^* E^* \cap \partial H|$, we get 
\begin{equation*}
P_\lambda(E) - P_\lambda(B^\lambda) \ge P_\lambda(E) - P_\lambda(E^*)  = P(E, \R^n \setminus H) - P(E^*, \R^n \setminus H).
\end{equation*}
It is therefore possible to reproduce the computation in \cite[Proposition 4.9]{FuscoDispensa} verbatim up to \cite[Eq. (4.29)]{FuscoDispensa} to estimate the right hand side in the last equation from below. We include the computation for the convenience of the reader. By \cite[Theorem 19.11, (19.30)]{MaggiBook} one has
\begin{equation*} \begin{split}
P_\lambda(E) &- P_\lambda(B^\lambda)  \ge  P(E, \R^n \setminus H) - P(E^*, \R^n \setminus H) \ge \int_0^\infty \left(\sqrt{p_E^2 + v_E'^2} - \sqrt{p^2_{E^*} + v_E'^2}\right)\de t \\ & = \int_0^\infty \frac{p^2_E - p^2_{E^*}}{\sqrt{p_E^2 + v_E'^2} + \sqrt{p^2_{E^*} + v_E'^2}}\de t  \ge \left(\int_0^\infty \sqrt{p_E^2 - p^2_{E^*}}\de t\right)^2\frac{1}{\int_0^\infty \sqrt{p_E^2 + v_E'^2} + \sqrt{p^2_{E^*} + v_E'^2}\de t } \\ & \ge \left(\int_0^\infty \sqrt{p_E^2 - p^2_{E^*}}\de t\right)^2\frac{1}{P(E, \R^n \setminus H) + P(E^*, \R^n \setminus H)} \\ & \ge c(\lambda) \left(\int_0^\infty \sqrt{p_E^2 - p^2_{E^*}}\de t\right)^2\frac{1}{P_\lambda(E) + P_\lambda(E^*)},
\end{split}
\end{equation*}
where in the last inequality we used \cref{remark:positivity}.
Since $D_\lambda(E)  < 1$ and $p_E \ge p_{E^*}$, we have $P_\lambda(E^*) \le P_\lambda(E) \le 2 P_\lambda(B^\lambda)$ and \begin{equation}\label{dispense:4.29} \begin{split} \sqrt{D_\lambda(E)} & \ge c \int_0^\infty \sqrt{p_E^2 - p^2_{E^*}}\de t  = c \int_0^\infty \sqrt{p_E + p_{E^*}} \sqrt{p_{E^*}} \sqrt{\frac{p_E - p_{E^*}}{p_{E^*}}} \de t \\ & \ge \sqrt{2} c \int_0^\infty p_{E^*} \sqrt{\frac{p_E - p_{E^*}}{p_{E^*}}} \de t, \end{split}
\end{equation}
for some constant $c=c(n,\lambda)>0$ changing from line to line.
Note that $(E^*)_t$ is a $(n - 1)$-dimensional ball with the same $\hh^{n - 1}$ measure of $E_t$. 
Then the quantity
\begin{equation*} \frac{p_E(t) - p_{E^*}(t)}{p_{E^*}(t)} \end{equation*}
is the classical isoperimetric deficit in $\R^{n - 1}$ of $E_t$ with respect to the standard perimeter. 
By standard the quantitative isoperimetric inequality in $\R^{n-1}$ \cite{FuscoMaggiPratelli}, the fact that $E_t$ is $n-1$ symmetric with $(E^*)_t$ centered at the center of symmetry of $E_t$ and \cref{lemma:fmp2.2}, we have
\begin{equation*}
\frac{\hh^{n-1}(E_t \Delta E^*_t)}{\hh^{n-1}((E^*)_t)} \le c(n) \sqrt{\frac{p_E(t) - p_{E^*}(t)}{p_{E^*}(t)}}. \end{equation*}
By~\eqref{dispense:4.29} and the inclusion $E \subset Q_{2l}$ we conclude
\begin{equation*}
\begin{split}
\sqrt{D_\lambda(E)} &\ge 
c \int_0^\infty \frac{p_{E^*}(t)}{\hh^{n-1}((E^*)_t)} \hh^{n-1}(E_t \Delta E^*_t) \de t
\ge\frac{c}{l} \int_{0}^\infty \hh^{n - 1}(E_t \Delta E_t^*)\de t = \frac{c}{l} |E \Delta E^*|.    
\end{split}
\end{equation*}
\end{proof}

Putting together \cref{lem:ReductionToBounded}, \cref{theorem:fmp2.1} and \cref{lem:Schwarz}, one immediately gets the following corollary.

\begin{corollary}\label{prop:ScharzSymmSufficienti}
There exist $\delta_4,\tilde l>0$ depending on $n, \lambda$ such that for every $0<\delta\le \delta_4$, if the quantitative isoperimetric inequality \eqref{eq:FinalQuantitativeInequality} holds (with some constant depending on $n,\lambda$) for every Schwarz-symmetric set $F$, with $|F| = |B^\lambda|$, contained in a cube $Q_{\tilde l}$ and with $D_\lambda(F)<\delta$, then~\eqref{eq:FinalQuantitativeInequality} holds for any measurable set of finite measure (up to changing the multiplicative constant, depending on $n,\lambda$ only).
\end{corollary}

\section{First quantitative isoperimetric inequality}\label{sec:Quantitative}

\subsection{Coupling}

We start by recalling the general definition and properties of the restricted envelopes introduced in \cite{CGPROS}.

\begin{definition}[{\cite[Definition 3.1]{CGPROS}}]
Let $K \subset \R^n$ be a compact convex set, $E \subset \R^n$ be a bounded open set, and $u \in C^0(\overline{E}) \cap C^2(E)$. The $K$-envelope of $u$ is the function $\bar u^K : \R^n \to \R$ given by
\begin{equation*} \bar u^K(x) := \sup \{a + \Braket{\xi, x} \st \xi \in K, \quad a + \Braket{\xi, y} \le u(y) \quad \forall y \in \overline{E}\}. \end{equation*}
\end{definition}

The desired coupling corresponding to a competitor $E$ will be essentially the $K$-envelope of the solution to the elliptic problem \eqref{cabsur:3.2}, for $K$ equal to the closure of the optimal bubble $B^\lambda$.

\begin{definition}[{\cite[Definition 3.9]{CGPROS}}]
Let $K \subset \R^n$ be a compact convex set, $E \subset \R^n$ be a bounded open set and $u \in C^0(\overline E) \cap C^2(E)$. Given $x \in \R^n$ and $\xi \in K$, we define \begin{equation*} S_\xi := \arg\min_{x \in \overline E} \{u(x) - \Braket{\xi, x}\} \end{equation*} and \begin{equation*} H(x, \xi, K) := \left\{\sum_{i=1}^m \lambda_i \nabla^2 u(s_i) \st \begin{array}{lr} 1 \le m \le n + 1 \\ \lambda_i \ge 0, \; \sum\lambda_i = 1 \\ s_i \in S_\xi \cap E \\ x - \sum \lambda_is_i \in N(\xi, K) \end{array}\right\},
\end{equation*}
where $N(\xi, K)$ is the normal cone of $K$ at $\xi$, defined as \begin{equation*} N(\xi, K) := \{v \in \R^n \st \Braket{v ,\xi' - \xi} \le 0 \; \text{for all} \; \xi' \in K\}. \end{equation*} \end{definition}

We will need the following

\begin{proposition}[{\cite[Proposition 3.10]{CGPROS}}]\label{CGPROS:prop3.10} Let $K \subset \R^n$ be a compact convex set, $E \subset \R^n$ be a bounded open set, and $u \in C^0(\overline E) \cap C^2(E)$. Assume that for any $\xi \in K$ it holds $S_\xi \subset E$, then $\bar u^K : \R^n \to \R$ is a $C^{1, 1}$ convex function such that \begin{equation*} \nabla\bar u^K(\R^n) = \nabla\bar u^K(E) = K. \end{equation*} Moreover, for any $x \in \R^n$ and any $H_x \in H(x, \nabla\bar u^K(x), K) \neq \emptyset$, it holds $\nabla^2 \bar u^K(x) \le H_x$.
\end{proposition}

We will need to associate a coupling only to Lipschitz-regular connected competitors, having $C^1$ relative boundary in $\R^n\setminus H$. This is established in the next result, whose proof is analogous to the one in \cite{CGPROS}.

\begin{proposition}\label{prop:EsistenzaCoupling}
There exists $\hat C=\hat C(n, \lambda)>0$ such that the following holds.
Let $E\subset \R^n\setminus H$ be a connected bounded open set. 
Suppose that $E$ has Lipschitz boundary and that $\partial E \cap \{x_n\ge0\}$ is a hypersurface of class $C^1$ with boundary.\\ 
Then there exists a $1$-Lipschitz convex function $\Psi : \R^n \to \R$ of class $C^{1, 1}$ such that  $\Delta \Psi \le \frac{P_\lambda(E)}{|E|}$ and such that $\nabla \Psi (\R^n) = \nabla \Psi (E) = B^\lambda$ up to negligible sets. 
Moreover, if $|E|=|B^\lambda|$ and $D_\lambda(E) \le 1$, then
\begin{equation}\label{CGPROS:4.7} \int_E |\nabla^2 \Psi - {\rm id}| \de x \le \hat C \sqrt{D_\lambda(E)} \end{equation} \begin{equation}\label{CGPROS:4.8} \int_{\partial^*E\cap (\R^n\setminus H)} (1 - |\nabla \Psi|) \de \hh^{n - 1} \le \hat C D_\lambda(E). \end{equation}
\end{proposition}

\begin{proof}
Let us assume first that $\partial E \setminus\partial H$ is a smooth hypersurface with smooth boundary intersecting $\partial H$ orthogonally. 
Let $u : E \to \R$ be a solution of~\eqref{cabsur:3.2} and $(K_i)_{i \in \N}$ a sequence of compact convex sets such that $K_i \subset\subset \mathring{K}_{i + 1}$ and $\cup_{i \in \N} K_i = B^\lambda$. 
We showed in the proof of \cref{thm:IsopIneq} that for any $\xi \in B^\lambda$ the minimum of $u(x) - \Braket{\xi, x}$ cannot be achieved on the boundary of $E$.
Moreover, at any point $x \in E$ such that $\nabla^2 u(x) \ge 0$, it holds \begin{equation*} 0 \le \nabla^2 u(x) \le \Delta u(x) \, {\rm id} = \frac{P_\lambda(E)}{|E|} \,{\rm id}. \end{equation*}
Therefore, recalling \cref{remark:regularity}, by \cref{CGPROS:prop3.10} we get that $\bar u^{K_i}$ is a sequence of $1$-Lipschitz functions, being suprema of $1$-Lipschitz functions, that is uniformly bounded in $C^{1, 1}$ on compact sets; hence up to  subsequence it converges to a limit function $\Psi$ in $C^1_{\rm loc}(\R^n)$. 
Since $\nabla\bar u^{K_i}(\R^n)=\nabla\bar u^{K_i}(E)=K_i$, then $\nabla \Psi (\R^n) = \nabla \Psi (\overline E) = \overline{B^\lambda}$, $\Psi$ is a convex function of class $C^{1,1}$, and writing the inequality $\Delta \bar u^{K_i} \le \frac{P_\lambda(E)}{|E|}$ in the sense of distributions, one checks that it readily passes to the limit as $i \to + \infty$ for the function $\Psi$.\\
To prove that $|\nabla\Psi(E) \Delta B^\lambda|=0$, let $Z\subset E$ be a compact set and notice that
\[
|\nabla\bar u^{K_i}(E\setminus Z) | \le c(n, |E|,P_\lambda(E)) |E\setminus Z|,
\]
\[
\begin{split}
    |\nabla\bar u^{K_i} (Z)| &= |\nabla\bar u^{K_i}(E \setminus (E\setminus Z))| \ge |\nabla\bar u^{K_i}(E)| - |\nabla\bar u^{K_i}(E\setminus Z)|
    = |K_i| - c |E\setminus Z|.
\end{split}
\]
Passing to the limit we find
\[
|\nabla \Psi(Z)| \ge \left| \limsup_i  \nabla\bar u^{K_i} (Z) \right| \ge \limsup_i |\nabla\bar u^{K_i} (Z)| \ge  |B^\lambda| - c |E\setminus Z|,
\]
hence letting $Z\nearrow E$, we get that  $\nabla \Psi (\R^n) = \nabla \Psi (E) = B^\lambda$ up to negligible sets.

Suppose now that $E$ is a generic connected set as in the assumptions. If $n\ge3$ we can apply the above argument to a sequence of sets $E_i$ approximating $E$ given by \cref{lemma:Density}, suitably modified connecting possibly disconnected components with thin tubes vanishing in the limit. If $n=2$, then $\partial E\setminus\partial H$ is a union of $C^1$ curves, which thus can be approximated by smooth ones touching $\partial H$ orthogonally preserving the connectedness of the set. Applying the first part of the proof on the approximating sequence $E_i$ we get a corresponding sequence of functions $\Psi_i$ uniformly bounded in $C^{1,1}$ on compact sets, hence converging in $C^1_{\rm loc}$ up to subsequence to a convex function $\Psi$ of class $C^{1,1}$ with $\Delta \Psi \le P_\lambda(E) /|E|$. Also, since $\nabla\Psi_i(\R^n)= \overline{B^\lambda}$, then $\nabla\Psi(\R^n)\subset \overline{B^\lambda}$ by $C^1_{\rm loc}$-convergence.
Moreover, since $E$ has Lipschitz boundary, there exists a sequence of compact sets $Z_j\subset E$ such that $Z_j\subset E_i$ for any $i\ge i_j$ and such that $Z_j\nearrow E$. Hence one can repeat the above argument with $\Psi_i, E, Z_j$ in place of $\bar u^{K_i}, E, Z$, respectively, to deduce that
\[
 |\nabla\Psi_i(Z_j)| \ge |B^\lambda| - c(n,|E_i|,P_\lambda(E_i)) |E_i\setminus Z_j| \ge |B^\lambda| - c(n,|E|,P_\lambda(E)) |E_i\setminus Z_j| .
\]
Letting $i\to\infty$ first, and then $j\to\infty$, we get that $\nabla \Psi (\R^n) = \nabla \Psi (E) = B^\lambda$ up to negligible sets.

We now prove \eqref{CGPROS:4.7} and \eqref{CGPROS:4.8}. The symbol $\hat C$ shall denote a positive constant depending on $n,\lambda$ changing from line to line.
By the area formula, the arithmetic-geometric mean inequality and the properties of $\Psi$ we get \begin{equation}\label{CGPROS:4.10} \begin{split} |B^\lambda| & = |\nabla\Psi(E)|  \le \int_E \det (\nabla^2 \Psi) \de \hh^{n} \le \int_E \left(\frac{\Delta \Psi}{n}\right)^n \de x  \le \int_E \left(\frac{P_\lambda(E)}{n |E|}\right)^n \de x  = \int_E \left(\frac{P_\lambda(E)}{n |B^\lambda|}\right)^n \de x \\ & = \left(\frac{P_\lambda(E)}{P_\lambda(B^\lambda)}\right)^n |B^\lambda|  = (1 + D_\lambda(E))^n |B^\lambda|. \end{split} \end{equation}
Hence
\begin{equation}\label{CGPROS:4.11} \int_E \left(\left(\frac{P_\lambda(E)}{n |E|}\right)^n - \det(\nabla^2\Psi)\right) \de x \le |B^\lambda|(1 + D_\lambda(E))^n - |B^\lambda| \le \hat C D_\lambda(E). \end{equation} 
By \cite[Lemma A.1]{CGPROS} applied with $m=n$, $\lambda_1=\ldots=\lambda_n =1$, $(x_1, \dots, x_n)$ equal to the eigenvalues of $\nabla^2\Psi$ and $c = \frac{P_\lambda(E)}{n|E|} = \frac{P_\lambda(E)}{P_\lambda(B^\lambda)} \ge 1$, we obtain \begin{equation}\label{CGPROS:4.12}
\begin{split}
    |\nabla^2\Psi - {\rm id}|^2 &\le
    2|\nabla^2\Psi - c\,{\rm id}|^2 + 2 |(c  -1) {\rm id}|^2
    \le \hat C\left(\left(\frac{P_\lambda(E)}{n|E|}\right)^n - \det(\nabla^2\Psi)\right) 
    +2 n \left(\frac{P_\lambda(E) - n|E|}{n|E|} \right)^2 \\
    &\le \hat C  \left(\left(\frac{P_\lambda(E)}{n|E|}\right)^n - \det(\nabla^2\Psi) + D_\lambda(E)^2\right) .
\end{split}
\end{equation} 
Therefore by~\eqref{CGPROS:4.11} and~\eqref{CGPROS:4.12} we get
\begin{equation*} \int_E |\nabla^2 \Psi - {\rm id}|^2 \de  x \le \hat C D_\lambda(E),
\end{equation*} 
which implies \eqref{CGPROS:4.7}.

Arguing as in~\eqref{CGPROS:4.10}, by the divergence theorem and using that $\Braket{ \nabla \Psi, -e_n} \le - \lambda$ since $\nabla\Psi(\R^n) \subset \overline{B^\lambda}$, we get
\begin{equation*} \begin{split} |B^\lambda| & \le \int_E \left(\frac{\Delta \Psi}{n}\right)^n \de x  \le \frac{P_\lambda(E)^{n - 1}}{n^n |E|^{n - 1}} \int_E \Delta\Psi \de x   = \frac{P_\lambda(E)^{n - 1}}{n^n |E|^{n - 1}} \int_{\partial^*E}\Braket{\nabla\Psi, \nu^E}\de\hh^{n - 1} \\
&\le \frac{P_\lambda(E)^{n - 1}}{n^n |E|^{n - 1}} \left(\int_{\partial^*E\cap (\R^n\setminus H)}\Braket{\nabla\Psi, \nu^E}\de\hh^{n - 1} + \int_{\partial^* E \cap \partial H} \Braket{ \nabla \Psi, -e_n} \de \hh^{n-1}
\right)\\
&\le \frac{P_\lambda(E)^{n - 1}}{n^n |E|^{n - 1}} \left(\int_{\partial^*E\cap (\R^n\setminus H)}\Braket{\nabla\Psi, \nu^E}-1\de\hh^{n - 1} + P_\lambda(E) 
\right)\\
&
\le \frac{P_\lambda(E)^{n}}{n^n |E|^{n - 1}} - \frac{P_\lambda(E)^{n - 1}}{n^n |E|^{n - 1}} \int_{\partial^*E\cap (\R^n\setminus H)} (1 - |\nabla\Psi|)\de\hh^{n - 1}. \end{split} \end{equation*}
Rearranging terms, since $D_\lambda(E)\le 1$ and $|E|=|B^\lambda|$, we obtain
\begin{equation*} \int_{\partial^*E\cap (\R^n\setminus H)} (1 - |\nabla\Psi|)\de\hh^{n - 1} 
\le \frac{P_\lambda(E)^n - P_\lambda(B^\lambda)^n}{P_\lambda(E)^{n-1}} 
\le \hat C D_\lambda(E). \end{equation*}
\end{proof}

We now want to translate the quantitative estimates obtained on the coupling in \cref{prop:EsistenzaCoupling} into quantitative estimates on the asymmetry of a competitor. We will need some technical results first.

The next lemma is analogous to \cite[Lemma 6.2]{CGPROS}, but with a varying range of parameters that here must depend on $\lambda$.

\begin{lemma}\label{CGPROS:lem6.2}
Let $E \subset [0, \infty)$ be a $1$-dimensional set of locally finite perimeter with $|E| < \infty$ and set
\begin{equation*}
r_\lambda := \min\left\{\sqrt{1 - \lambda^2}, 1 - \lambda\right\},
\qquad
R_\lambda := \max\left\{\sqrt{1 - \lambda^2}, 1 - \lambda\right\}.
\end{equation*}
There exists $\hat c=\hat c(n,\lambda)>0$ such that for any $\frac{7}{8}r_\lambda \le l \le \frac{9}{8}R_\lambda$ there holds \begin{equation}\label{CGPROS:6.1} \int_{E \Delta[0, l]} t^{n - 1} \de t \le \hat c\left(\int_{\left[0, \frac{r_\lambda}{2}\right]\setminus E}t^{n - 1}\de t + \int_{\partial^*E}t^{n - 1}|l - t|\de\hh^0\right). \end{equation}
\end{lemma}

\begin{proof}
It holds $E \Delta [0, l] = ((l, \infty) \cap E) \cup ([0, l] \setminus E)$. 
We claim that \begin{equation}\label{CGPROS:claim} \max\left\{\int_{[l, \infty) \cap E} t^{n - 1}\de t, \int_{\left[\frac{r_\lambda}{2}, l\right]\setminus E} t^{n - 1} \de t\right\} \le \hat c(n, \lambda) \left(\int_{\left[0, \frac{r_\lambda}{2}\right]\setminus E}t^{n - 1}\de t + \int_{\partial^*E} t^{n - 1}|l - t|\de\hh^0(t)\right). \end{equation} 
We will estimate the two terms on the left-hand side separately.

Without loss of generality, suppose $[l, \infty) \cap E \neq \emptyset$. Since $|E| < \infty$, then $\partial^* E \cap [l, \infty)$ is nonempty and we can assume it has finite supremum $\bar t$ (otherwise the right hand side in \eqref{CGPROS:claim} equals $+\infty$). 
In particular the right-hand side in~\eqref{CGPROS:claim} is finite. 
It holds \begin{equation*} \int_{[l, \infty) \cap E} t^{n - 1} \de t \le \int_{l}^{\bar t} t^{n - 1} \de t \le \bar t^{n - 1} |\bar t - l| \le \int_{\partial^*E}t^{n - 1}|l - t| \de\hh^0(t), \end{equation*} and the first term in the left-hand side of~\eqref{CGPROS:claim} is bounded as wished.

Let us now consider $\int_{\left[\frac{r_\lambda}{2}, l\right]\setminus E} t^{n - 1} \de t$. 
Its value is a priori bounded by $\left(\frac{9}{8}R_\lambda\right)^{n}$. 
If $\partial^*E \cap\left[\frac{r_\lambda}{4}, \frac{3}{4}r_\lambda\right] \neq \emptyset$ and $\tau$ is one of its elements, then \begin{equation*} \begin{split}  \int_{\partial^*E} t^{n - 1}|l - t|\de\hh^0(t) & \ge \tau^n \left|l - \tau\right|  \ge \left(\frac{r_\lambda}{4}\right)^n \frac{1}{8}r_\lambda \ge \hat c(n, \lambda) \left(\frac{9}{8}R_\lambda\right)^{n - 1} \left|\frac{9}{8}R_\lambda - \frac{r_\lambda}{2}\right| \ge \hat c(n, \lambda) \int_{\left[\frac{r_\lambda}{2}, l\right] \setminus E} t^{n - 1} \de t. \end{split} \end{equation*} 
So from now on we can assume that $\partial^*E \cap\left[\frac{r_\lambda}{4}, \frac{3}{4}r_\lambda\right] = \emptyset$. 
If $\left[\frac{r_\lambda}{4}, \frac{3}{4}r_\lambda\right] \setminus E \neq \emptyset$, then $E \cap\left[\frac{r_\lambda}{4}, \frac{3}{4}r_\lambda\right] = \emptyset$ and 
\begin{equation*}
\begin{split}
\int_{\left[0, \frac{r_\lambda}{2}\right]\setminus E}t^{n - 1}\de t + \int_{\partial^*E} t^{n - 1}|l - t|\de\hh^0(t) & \ge \int_{\left[\frac{r_\lambda}{4}, \frac{r_\lambda}{2}\right]}t^{n - 1}\de t  = \hat c(n, \lambda)  \ge \hat c(n, \lambda) \left(\frac{9}{8}R_\lambda\right)^{n - 1} \left|\frac{9}{8}R_\lambda - \frac{r_\lambda}{2}\right| \\ & \ge \hat c(n, \lambda) \int_{\left[\frac{r_\lambda}{2}, l\right] \setminus E} t^{n - 1} \de t. \end{split} \end{equation*} 
So we can further assume $\left[\frac{r_\lambda}{4}, \frac{3}{4}r_\lambda\right] \subset E$. 
If $\partial^*E \cap\left[\frac{r_\lambda}{4}, l\right] = \emptyset$, then $\left[\frac{r_\lambda}{2}, l\right]$ $\subset$ $E$ and there is nothing to prove. 
Finally, if $\partial^*E \cap\left[\frac{r_\lambda}{4}, l\right] \neq \emptyset$, let us denote by $\underline t$ the infimum of $\partial^*E \cap \left[\frac{r_\lambda}{4}, l\right]$. 
Then \begin{equation*} \begin{split} \int_{\left[\frac{r_\lambda}{2}, l\right] \setminus E} t^{n - 1}\de t & \le \int_{\underline t}^{l}t^{n - 1} \de t \le l^{n - 1} |l - \underline t|  = \underline t^{n - 1} |l - \underline t| \left(\frac{l}{\underline t}\right)^{n - 1}  \le \hat c(n, \lambda) \underline t^{n - 1} |l - \underline t| \\ & \le \hat c(n, \lambda)\int_{\partial^*E}t^{n - 1}|l - t|\de\hh^0(t). \end{split} \end{equation*} 
This concludes the proof of the claim~\eqref{CGPROS:claim}. 
Hence
\begin{equation*}
\begin{split}
\int_{E \Delta [0, l]} t^{n - 1} \de t & \le 
\max\left\{\int_{[l, \infty) \cap E} t^{n - 1}\de t, \int_{\left[\frac{r_\lambda}{2}, l\right]\setminus E} t^{n - 1} \de t\right\} 
+  \int_{\left[0,\frac{r_\lambda}{2}\right]\setminus E} t^{n - 1} \de t
\\&\overset{\eqref{CGPROS:claim}}{\le} 
\hat c(n, \lambda) \left(\int_{\left[0, \frac{r_\lambda}{2}\right]\setminus E}t^{n - 1}\de t + \int_{\partial^*E} t^{n - 1}|l - t|\de\hh^0(t)\right) +  \int_{\left[0,\frac{r_\lambda}{2}\right]\setminus E} t^{n - 1} \de t, 
\end{split}\end{equation*}
and the proof follows.
\end{proof}

We shall also need the following standard technical result stating that a Vol'pert property holds for the intersections of a set of finite perimeter with rays from the origin, cf. \cite{Volpert}. The proof follows, for example, by adapting the proof of \cite[Theorem 3.21]{FuscoClassicalIsopProb} working in polar coordinates rather than in Cartesian coordinates.

\begin{lemma}\label{lemma:volpert}
Let $E \subset \R^n \setminus H$ be a set of finite perimeter with $|E|<+\infty$. If $\theta \in \mathbb{S}^{n - 1} \cap (\R^n \setminus H)$, we define \begin{equation*} E_\theta := \{t \ge 0 \st t\theta \in E\}. \end{equation*} Then, for $\hh^{n - 1}$-almost every $\theta \in \mathbb{S}^{n - 1} \cap (\R^n \setminus H)$, $E_\theta$ is a $1$-dimensional set of locally finite perimeter such that
\begin{equation*} \partial^*E_\theta \cap \{t > 0\} = \{t > 0 \st t\theta \in \partial^*E\}. \end{equation*}
Moreover, if $\eta \in L^1(\partial^*E)$ is nonnegative, we have \begin{equation}\label{volpert:area} \int_{\partial^*E \setminus H}\eta\de\hh^{n - 1} \ge \int_{\mathbb{S}^{n - 1} \setminus H} \left(\int_{\partial^*E_\theta}t^{n - 1}\eta(t\theta)\de\hh^0(t)\right)\de\hh^{n - 1}(\theta). \end{equation}
\end{lemma}

Combining \cref{CGPROS:lem6.2} with \cref{lemma:volpert} we get the following result that estimates the symmetric difference of a competitor with a bubble that is just close to a standard bubble $B^\lambda(v,x)$. The result is analogous to  \cite[Proposition 6.1]{CGPROS}.

\begin{lemma}\label{CGPROS:prop6.1}
There exist $\epsilon, \tilde c>0$ depending on $n,\lambda$ such that the following holds.
Let $E \subset \R^n\setminus H$ be a bounded set of finite perimeter.
If $\left|E \cap B_{\frac{r_\lambda}{2}}(0)\right| \ge \frac{1}{2} \left|B_{\frac{r_\lambda}{2}}(0) \setminus H\right|$, then
\begin{equation*} |E \Delta (B_1(x_0) \setminus H)| \le \tilde c\int_{\partial^*E \setminus H} ||x - x_0| - 1|\de\hh^{n - 1}(x), \end{equation*} for any $x_0 \in \R^n$ such that $|x_0 - (0, \dots, 0, - \lambda)| < \epsilon$.
\end{lemma}

\begin{proof}
If $\epsilon$ is sufficiently small, depending only on $n,\lambda$, then for any $\theta \in \mathbb{S}^{n - 1} \cap (\R^n \setminus H)$ the set $\{t > 0 \st |t\theta - x_0| < 1\}$ is an open segment $(0, t(\theta))$, with $t(\theta)$ close to the number
\begin{equation*} T_\theta \in \left[\min\left\{\sqrt{1 - \lambda^2}, 1 - \lambda\right\}, \max\left\{\sqrt{1 - \lambda^2}, 1 - \lambda\right\}\right] =: [ r_\lambda, R_\lambda]\end{equation*}
such that $|T_\theta \theta + \lambda e_n| = 1$. 
In particular, if $\epsilon$ is sufficiently small, then $\frac{9}{8} R_\lambda \ge t(\theta) \ge \frac{7}{8}r_\lambda$  for any $\theta \in \mathbb{S}^{n - 1} \cap (\R^n \setminus H)$.

As before, for any $\theta \in \mathbb{S}^{n - 1} \setminus H$ let \begin{equation*} E_\theta := \{t \ge 0 \st t\theta \in E\}. \end{equation*} 
By coarea formula we get \begin{equation}\label{CGPROS:coarea} |E \Delta B_1(x_0) \cap (\R^n \setminus H)| = \int_{\mathbb{S}^{n - 1} \setminus H} \left(\int_{E_\theta \Delta [0, t(\theta)]} t^{n - 1} \de t\right)\de \hh^{n - 1}(\theta) . \end{equation} 
By \cref{CGPROS:lem6.2} we obtain \begin{equation}\label{eq:zxsa}
|E \Delta B_1(x_0) \setminus H| \le c(n, \lambda) \int_{\mathbb{S}^{n - 1} \setminus H} \left(\int_{\left[0, \frac{r_\lambda}{2}\right]\setminus E_\theta}t^{n - 1}\de t + \int_{\partial^*E_\theta}t^{n - 1}|t(\theta) - t|\de\hh^0\right)\de \hh^{n - 1}(\theta).
\end{equation} 
For every $t > 0$, we claim that \begin{equation}\label{CGPROS:claim2} ||t\theta - x_0| - 1| \ge c(n, \lambda) |t - t(\theta)|. \end{equation} 
Note that there exists $\delta=\delta(n,\lambda,\epsilon)\in(0,r_\lambda/8)$ such that for $t \in [t(\theta)-\delta, t(\theta) + \delta]$ there holds \begin{equation*} \left|\frac{\d}{\d t}|t\theta - x_0|\right| = \left|\Braket{(t\theta - x_0)/|t\theta - x_0|, \theta}\right| \ge c(n, \lambda, \epsilon)>0. \end{equation*} 
Then, for $t \in [t(\theta)-\delta, t(\theta) + \delta]$,
\begin{equation*} |t(\theta) - t| \le c(n, \lambda) ||t\theta - x_0| - 1|, \end{equation*} and in this case the claim follows. 
Regarding the remaining cases, note that \begin{equation*} \frac{||t\theta - x_0| - 1|}{|t(\theta) - t|} \to 1 \qquad \text{as} \; |t| \to + \infty \end{equation*} and the claim follows for $t\ge R=R(n,\lambda, \epsilon)>0$ big enough. 
Finally, if $0<t < R$ and $t \not\in [t(\theta)-\delta, t(\theta) + \delta]$, then \begin{equation*} \left||t\theta - x_0| - 1\right| \ge c(n,\lambda,\delta)>0 \end{equation*} \begin{equation*} |t(\theta) - t| \le c(n,\lambda,R) \end{equation*}
hence the claim follows as well.

Therefore
\begin{equation}\label{eq:zxsa2}
\int_{\mathbb{S}^{n - 1} \setminus H} \int_{\partial^*E_\theta}t^{n - 1}|t - t(\theta)|\de\hh^0(t)\de\hh^{n - 1}(\theta) \overset{\eqref{CGPROS:claim2}}{\le} c(n, \lambda) \int_{\mathbb{S}^{n - 1} \setminus H}\int_{\partial^*E_\theta}t^{n - 1}||t\theta - x_0| - 1|\de\hh^0(t)\de\hh^{n - 1}(\theta).
\end{equation} 
By \cref{lemma:volpert} we deduce \begin{equation}\label{CGPROS:6.2} \begin{split} \int_{\partial^*E \setminus H}& ||x - x_0| - 1| \de\hh^{n - 1}(x) \ge \int_{\mathbb{S}^{n - 1} \setminus H}\left(\int_{\partial^*E_\theta}t^{n - 1}||t\theta - x_0| - 1|\de\hh^0(t)\right)\de\hh^{n - 1}(\theta). \end{split} \end{equation} 
Since $||x - x_0| - 1| \ge c(n, \lambda)>0$ in $B_{\frac{r_\lambda}{2}}(0)$, by coarea formula and relative isoperimetric inequality we get \begin{equation}\label{CGPROS:6.3} \begin{split} \int_{\mathbb{S}^{n - 1} \setminus H} \left(\int_{\left[0, \frac{r_\lambda}{2}\right]\setminus E_\theta} t^{n - 1} \de t\right)\de\hh^{n - 1}(\theta) & = \left|\left(B_{\frac{r_\lambda}{2}}(0) \setminus H\right) \setminus E\right|  = \left|\left(B_{\frac{r_\lambda}{2}}(0) \setminus H\right) \setminus E\right|^{\frac{n - 1}{n}} \left|\left(B_{\frac{r_\lambda}{2}}(0) \setminus H\right) \setminus E\right|^{\frac{1}{n}} \\ & \le c(n, \lambda) \int_{\partial^*E \cap \left(B_{\frac{r_\lambda}{2}}(0) \setminus H\right)} \de \hh^{n - 1} \\& \le c(n, \lambda) \int_{\partial^*E \setminus H} ||x - x_0| - 1|\de\hh^{n - 1}(x). \end{split} \end{equation} 
Putting together \eqref{eq:zxsa}, \eqref{eq:zxsa2}, \eqref{CGPROS:6.2} and~\eqref{CGPROS:6.3}, the proof follows.
\end{proof}

We can finally show that if a suitably regular Schwarz-symmetric set satisfies a trace inequality, then the quantitative estimates in \cref{prop:EsistenzaCoupling} imply a quantitative isoperimetric inequality.

\begin{proposition}\label{prop:QuantitativaDataCT}
There exists $\delta_5=\delta_5(n, \lambda) >0$ such that for any $c_T>0$ there exists $\gamma=\gamma(n,\lambda,c_T)>0$ such that the following holds.
Let $E \subset \R^n\setminus H$ be a bounded connected open set with $|E|=|B^\lambda|$. Suppose that $E$ has Lipschitz boundary and that $\partial E \cap \{x_n\ge0\}$ is a hypersurface of class $C^1$ with boundary. Assume that $E$ is Schwarz-symmetric with respect to the $n$-th axis and that there exists a constant $c_T$ such that for every function $f \in BV(\R^n) \cap L^\infty(\R^n)$ there is a constant $c \in \R$ such that the following holds
\begin{equation} \label{CGPROS:7.3}
\int_E \de|D f|(x) \ge c_T \int_{\partial^*E \cap (\R^n \setminus H)} {\rm tr}_E(|f - c|)\de\hh^{n - 1}(x). \end{equation}
If $D_\lambda(E) < \delta_5$, then
\begin{equation*} \alpha^2_\lambda(E) \le \gamma D_\lambda(E). \end{equation*} \end{proposition}

\begin{proof}
Let $\Psi$ be given by \cref{prop:EsistenzaCoupling}.
By~\eqref{CGPROS:7.3} and~\eqref{CGPROS:4.7} we get \begin{equation*}
\int_{\partial^*E \cap (\R^n \setminus H)} |(\nabla\Psi - x) + x_0| \de\hh^{n - 1}(x) \le c(n, \lambda, c_T) \sqrt{D_\lambda(E)}, \end{equation*} where $x_0 = (x_0^1, \dots, x_0^n)$ is the vector whose $i$-th component is the constant $c$ of~\eqref{CGPROS:7.3} corresponding to the $i$-th component of $\nabla\Psi - x$. Therefore
\begin{equation}\label{eq:zzw}
\begin{split} \int_{\partial^*E \cap (\R^n \setminus H)} \left||x - x_0| - 1\right| \de \hh^{n - 1}(x) &\le \int_{\partial^*E \cap (\R^n \setminus H)} |\nabla\Psi - (x - x_0)| + |1 - |\nabla\Psi||\de\hh^{n - 1}(x) \\ & \overset{\eqref{CGPROS:4.8}}{\le} c(n, \lambda, c_T)\sqrt{D_\lambda(E)}. \end{split}
\end{equation}
We observe that if $\epsilon=\epsilon(n,\lambda)$ is given by \cref{CGPROS:prop6.1}, then for $\delta_5$ small enough depending on $\epsilon$, we ensure that $|x_0-(0,\ldots,0,-\lambda)|<\epsilon$. Indeed, if for every $i \in \N$ there were $E_i$ satisfying the hypotheses of \cref{prop:QuantitativaDataCT} such that $D_\lambda(E) < \frac{1}{i}$ with corresponding $x_{0,i}$ verifying $|x_{0,i}-(0,\ldots,0,-\lambda)| \ge \epsilon$, passing to limit in~\eqref{eq:zzw} we would get a contradiction with the fact that $E_i$ converges to $B^\lambda(|B^\lambda|)$.

Hence we can apply \cref{CGPROS:prop6.1}. Since $E$ is Schwarz-symmetric, we get
\begin{equation}\label{eq:zzQuasiQuant}
\begin{split} |E \Delta B_1(0, \dots, 0, x_0^n) \cap (\R^n \setminus H)| & \le |E \Delta B_1(x_0) \cap (\R^n \setminus H)| \\ & \le c(n, \lambda, c_T) \sqrt{D_\lambda(E)}. \end{split}
\end{equation}

Arguing as above, up to taking a smaller $\delta_5$, we can assume that $|x_0^n +\lambda|\le |x_0- (0,\ldots,0,-\lambda)|$ is so small that
\begin{equation*}
     \left| \frac{\d}{\d t} |B_1(0,\ldots,0,t) \setminus H | \right| \ge \frac12 \omega_{n-1}(1-\lambda^2)^{\frac{n-1}{2}} ,
\end{equation*}
for any $t \in[-|x_0^n +\lambda|, |x_0^n +\lambda|]$. Hence
\begin{equation*}
\begin{split}
c(n, \lambda, c_T) \sqrt{D_\lambda(E)}
&\overset{\eqref{eq:zzQuasiQuant}}{\ge} \left||B_1(0, \dots, 0, x_0^n) \setminus H| - |E|\right| 
= \left||B_1(0, \dots, 0, x_0^n) \setminus H| - |B_1(0, \dots, 0, - \lambda) \setminus H|\right|\\
&\ge \frac12 \omega_{n-1}(1-\lambda^2)^{\frac{n-1}{2}} |x_0^n +\lambda|,
\end{split}
\end{equation*}
which implies
\begin{equation}\label{eq:zzVicinanzaXzeron}
    |x_0^n +\lambda| \le c(n, \lambda, c_T) \sqrt{D_\lambda(E)},
\end{equation}
for a suitable constant. Therefore
\begin{equation}\label{zz:QuasiQuant2}
    \begin{split}
        \left|\left(B_1(0, \dots, 0, x_0^n)\setminus H\right)\Delta\left(B_1(0, \dots, 0, -\lambda)\setminus H\right)\right| 
    &\le c(n,\lambda) |x_0^n+\lambda| \\&\overset{\eqref{eq:zzVicinanzaXzeron}}{\le} c(n, \lambda, c_T) \sqrt{D_\lambda(E)},
    \end{split}
\end{equation}
where in the first inequality we used that $t\mapsto \left|\left(B_1(0, \dots, 0, t)\setminus H\right)\Delta\left(B_1(0, \dots, 0, -\lambda)\setminus H\right)\right|$ is Lipschitz for some Lipschitz constant $c(n,\lambda)>0$.

Finally
\begin{equation*}
\begin{split}
|E & \Delta B_1(0, \dots, 0, x_0^n) \cap (\R^n \setminus H)| \\
& \ge |B^\lambda(|B^
\lambda|) \Delta E| - \left|\left(B_1(0, \dots, 0, x_0^n)\setminus H\right)\Delta\left(B_1(0, \dots, 0, -\lambda)\setminus H\right)\right| \\
& \overset{\eqref{zz:QuasiQuant2}}{\ge} \alpha_\lambda(E) - c(n, \lambda, c_T) \sqrt{D_\lambda(E)}.
\end{split}
\end{equation*}
\end{proof}

In the next lemma we observe that optimal bubbles do satisfy trace inequalities.

\begin{lemma}
There exists $\bar c=\bar c(n, \lambda)>0$ such that for every function $f \in BV(\R^n) \cap L^\infty(\R^n)$ there is a constant $c \in \R$ such that the following holds \begin{equation}\label{traccia:calotta} \int_{B^\lambda(|B^\lambda|)} \de|Df|(x) \ge \bar c \int_{\partial B^\lambda(|B^\lambda|) \setminus H} {\rm tr}_{B^\lambda(|B^\lambda|)}(|f - c|) \de\hh^{n - 1}(x). \end{equation}
\end{lemma}

\begin{proof}
The proof follows combining the classical Poincaré inequality \cite[Theorem 3.44]{AmbrosioFuscoPallara} with the boundary trace theorem \cite[Theorem 3.87]{AmbrosioFuscoPallara}.
\end{proof}

We now introduce a notion of $C^1$-distance from $B^\lambda(|B^\lambda|)$ for sets in the half-space $\R^n\setminus H$, and we deduce that Schwarz-symmetric sets sufficiently close in $C^1$ to $B^\lambda(|B^\lambda|)$ enjoy a quantitative isoperimetric inequality.

\begin{definition}\label{def:VicinanzaC1Schwarz}
Let $\phi_\lambda : \partial B_1 \setminus H \to \R$ be such that $\partial B^\lambda(|B^\lambda|) \setminus H = \{ \varphi_\lambda(x) \, x \st x \in \partial B_1 \setminus H\}$.\\
Let $E \subset \R^n\setminus H$ be a bounded open set. Suppose that $E$ has Lipschitz boundary and that $\partial E \cap \{x_n\ge0\}$ is a hypersurface of class $C^1$ with boundary. 
Assume that $E$ is Schwarz-symmetric. Suppose that there exists a $C^1$ functions
\begin{equation*} \begin{split}
\phi : \partial B_1 \setminus H \to \R
\end{split} \end{equation*}
whose graph parametrizes the boundary of $E$ in $\R^n\setminus H$, that is
\[
\partial E \setminus H = \{ \varphi(x) \, x \st x \in \partial B_1 \setminus H\}.
\]
We define the $C^1$ distance of $E$ to $B^\lambda(|B^\lambda|)$ by $\d_{C^1}(E,B^\lambda(|B^\lambda|)):=\|\phi - \phi_\lambda\|_{C^1(\partial B_1 \cap \R^n \setminus H)}$.\\
A sequence of sets $E_j$ as above is said to converge to $B^\lambda(|B^\lambda|)$ in $C^1$ if $\d_{C^1}(E_j,B^\lambda(|B^\lambda|))\to0$ as $j\to+\infty$.
\end{definition}

\begin{corollary}\label{quantitative:nearball}
There exist $\hat\epsilon, \hat\gamma>0$ depending only on $n,\lambda$ such that the following holds. 
Let $E \subset \R^n\setminus H$ be as in \cref{def:VicinanzaC1Schwarz}. If $\d_{C^1}(E,B^\lambda(|B^\lambda|)) \le \hat\epsilon$, then
\begin{equation*} \alpha^2_\lambda(E) \le \hat\gamma(n, \lambda) D_\lambda(E). \end{equation*}
\end{corollary}

\begin{proof}
Let $\varphi,\varphi_\lambda$ be as in \cref{def:VicinanzaC1Schwarz}.
If $\chi : [0, + \infty) \to [0, 1]$ is a smooth cut-off function such that $\chi(t)=0$ for $t<\tfrac14 \min\{ \sqrt{1 - \lambda^2}, 1-\lambda\}$ and such that $\chi(t) = 1$ for $t >\tfrac12 \min\{ \sqrt{1 - \lambda^2}, 1-\lambda\}$, we define the diffeomorphism \begin{equation*}
\psi : \R^n\setminus H \to \R^n\setminus H \qquad\qquad \psi (x) = \left(1 - \chi(|x|) + \chi(|x|) \frac{\phi\left(\frac{x}{|x|}\right)}{\phi_\lambda\left(\frac{x}{|x|}\right)}\right)x.
\end{equation*} 
Note that
\begin{equation*}
\|\psi - {\rm id}\|_{C^1} \le c \hat\epsilon,
\qquad\qquad
\psi(\partial B^\lambda(|B^\lambda|) \setminus H) = \partial E \setminus H,
\end{equation*} 
for some $c=c(n,\lambda,\chi)$, if $\d_{C^1}(E,B^\lambda(|B^\lambda|)) \le \hat\epsilon<1$.

Let $g \in {\rm Lip}_c(\R^n)$ and define $f := g \circ \psi$. If $c$ is the constant in~\eqref{traccia:calotta} corresponding to $f$, then by area formula and~\eqref{traccia:calotta} we get
\begin{equation*} \begin{split} \int_{\partial^*E \cap \setminus H}|g - c| \de\hh^{n - 1} & \le C(n, \lambda)\int_{\partial B^\lambda \setminus H} |f - c| \de\hh^{n - 1}  \le C(n, \lambda) \int_{B^\lambda} |\nabla f| \de x \le C(n, \lambda) \int_E |\nabla g| \de x. \end{split} \end{equation*}
Therefore, if $\hat\epsilon$ is small enough, we can apply \cref{prop:QuantitativaDataCT} with $c_T$ therein depending on $n,\lambda$ only, and we get
\begin{equation*} \alpha_\lambda(E) \le \hat\gamma(n, \lambda) D_\lambda(E).
\end{equation*}
\end{proof}

\subsection{Proof of the first quantitative isoperimetric inequality}

We are ready to prove the main quantitative isoperimetric inequality. Let us recall the following immediate result, completely analogous to \cite[Lemma 5.3]{FuscoDispensa}.
 
\begin{lemma}\label{dispense:lemma5.3}
The standard bubble $B^\lambda(|B^\lambda|)$ is the unique solution, up to translations along $\partial H$, of
\begin{equation*} \min\left\{P_\lambda(F) + \Lambda \left||F| - |B^\lambda|\right| \st F \subset \R^n \setminus H\right\}, \end{equation*} for any $\Lambda > n$. \end{lemma}

\begin{proof}[Proof of \cref{thm:FinalQuantitativeInequality}]
Let $\hat\gamma(n, \lambda)$ be the constant given by \cref{quantitative:nearball} and let $\delta_4, \tilde l$ be given by \cref{prop:ScharzSymmSufficienti}.
By \cref{prop:ScharzSymmSufficienti} it is sufficient to prove that there exists $\delta \in(0,\delta_4)$ such that, if $E$ is a Schwarz-symmetric set contained in $Q_{\tilde l}$ such that $|E| = |B^\lambda|$ and $D_\lambda(E) < \delta$, then $\alpha_\lambda(E) \le 2 \hat \gamma(n, \lambda) \sqrt{D_\lambda(E)}$.

We argue by contradiction. Let $\{E_j\}_j$ be a sequence of Schwarz-symmetric sets contained in $Q_{\tilde l}$ such that $|E_j| = |B^\lambda|$, with $P_\lambda(E_j) \to P_\lambda(B^\lambda)$ and \begin{equation}\label{selection:contradiction} \alpha_\lambda(E_j) > 2 \hat\gamma(n, \lambda)\sqrt{D_\lambda(E_j)}. \end{equation} 
For every $j$ we consider a minimizer $F_j$ of the problem
\begin{equation}\label{eq:ProblemSelectionPrinciple}
\min\{P_\lambda(F) + |\alpha_\lambda(F) - \alpha_\lambda(E_j)| + \Lambda ||F| - |E_j|| \st F \;\text{Schwarz-symmetric contained in}\,Q_{\tilde l}\}, \end{equation}
for $\Lambda > 0$ to be chosen large. Up to subsequence, $F_j$ converges in $L^1$ to a minimizer of $F\mapsto P_\lambda(F) + \alpha_\lambda(F) + \Lambda\left| |F| - |B^\lambda| \right|$, hence, taking $\Lambda>n$, we have that $F_j$ converges to $B^\lambda(|B^\lambda|)$ by \cref{dispense:lemma5.3}. Also, by comparison with with $E_j$, we have that $P_\lambda(F_j) \to P_\lambda(B^\lambda)$.

We prove that $F_j$ is a local $(\Lambda_1, r_0)$-minimizer in $\R^n\setminus H$, for some $\Lambda_1 , r_0> 0$ and $j$ large. 
Let us consider a ball $B_r(x)\subset \subset \R^n\setminus H$, with $r < \min\{r_0, d(x, \partial H)\}$, and a set $G$ such that $F_j \Delta G \subset \subset B_r(x)$. 
Denoting by $(\cdot)^*$ the Schwarz symmetrization with respect to the $n$-th axis and by $Z:= G \cap Q_{\tilde l}$, we have
\begin{equation*} \begin{split}
P(F_j, \R^n\setminus H) & \le P(Z^*, \R^n\setminus H) + |\alpha_\lambda(Z^*) - \alpha_\lambda(E_j)| - |\alpha_\lambda(F_j) - \alpha_\lambda(E_j)| + \Lambda[||Z| - |E_j|| - ||F_j| - |E_j|| ] \\ & \le P(Z, \R^n\setminus H) + |\alpha_\lambda(Z^*) - \alpha_\lambda(F_j)| + \Lambda |Z \Delta F_j| \\
&\le P(G, \R^n\setminus H) + |\alpha_\lambda(Z^*) - \alpha_\lambda(F_j)| + \Lambda |G \Delta F_j|.
\end{split} \end{equation*}
For $r_0$ small enough and $j$ sufficiently large we have that $|G|\ge |Z| \ge c(n,\lambda)>0$. Assume for instance that $\alpha_\lambda(Z^*) \ge \alpha_\lambda(F_j)$ (the opposite case being symmetric), then
\[
\begin{split}
    \alpha_\lambda(Z^*) - \alpha_\lambda(F_j)
    & \le |Z|^{-1} \left( |Z^* \Delta F_j| + |F_j \Delta B^\lambda(|F_j|) | + \left| |Z| - |F_j| \right| \right) - |F_j|^{-1} |F_j \Delta B^\lambda(|F_j|) | \\
    &\le c(n,\lambda) |Z\Delta F_j| + c(n,\lambda)  \left( |F_j| - |Z| \right) + |Z\Delta F_j| \\
    &\le c(n,\lambda) |G \Delta F_j|.
\end{split}
\]
Arguing analogously in case $\alpha_\lambda(Z^*) < \alpha_\lambda(F_j)$, we deduce
\[
P(F_j, \R^n\setminus H)\le P(G, \R^n\setminus H) + \Lambda_1 |G \Delta F_j|,
\]
for some $\Lambda_1=(\Lambda, n, \lambda)$.\\
By \cref{thm:RegolaritaLambdaMinimizers} we know that $\partial^* F_j \cap \{x_n > 0\}$ is a $C^{1, \frac{1}{2}}$ manifold and $\partial F_j \cap \{x_n>0\} \setminus \partial^* F_j$ has Hausdorff dimension $\le n-8$. Since $F_j$ is Schwarz-symmetric, if there exists a point $(r\theta, t)\in \partial F_j \cap \{x_n>0\} \setminus \partial^* F_j$ for some $r,t>0, \theta \in \mathbb{S}^{n-2}$, then $(r\theta', t) \in\partial F_j \cap \{x_n>0\} \setminus \partial^* F_j$ for any $\theta' \in \mathbb{S}^{n-2}$. Hence $\partial F_j \cap \{x_n>0\} \setminus \partial^* F_j \subset \{ te_n \st t>0\}$. However by \cref{thm:RegolaritaLambdaMinimizers} for every $\epsilon > 0$ the set $\partial F_j \cap \{x_n \ge \epsilon\}$ converges to $\partial B^\lambda(|B^\lambda|) \cap \{x_n \ge \epsilon\}$ in $C^{1, \alpha}$ for any $0 < \alpha < \frac{1}{2}$. Also, for $j$ large we can apply \cref{cor:DeficitBassoFetteGrosse} which implies that $\hh^{n-1}(F_j \cap \{x_n = t\}) \ge A_\lambda$ for a.e. $t \in (0,T_\lambda)$, for some $A_\lambda,T_\lambda>0$ depending on $n,\lambda$. Then points $te_n$ for $t \in (0,T_\lambda/2)$ are points of density $1$ for $F_j$, hence they belong to the interior of $F_j$. Therefore, for $j$ large enough, $\partial F_j \cap \{x_n>0\} \setminus \partial^* F_j$ must be empty and $\partial F_j \cap \{x_n > 0\}$ is an axially symmetric hypersurface of class $C^{1, \frac{1}{2}}$.

By the minimality of the $F_j$,~\eqref{selection:contradiction} and \cref{dispense:lemma5.3} we observe that
\begin{equation}\label{dispense:5.13} \begin{split} P_\lambda(F_j) &+ \Lambda \left||F_j| - |B^\lambda|\right| + \left|\alpha_\lambda(F_j) - \alpha_\lambda(E_j)\right|  \le P_\lambda(E_j) \\ & \le P_\lambda(B^\lambda(|B^\lambda|)) + \frac{P_\lambda(B^\lambda(|B^\lambda|))}{4\hat\gamma^2(n, \lambda)} \alpha_\lambda^2(E_j) 
\le P_\lambda(F_j) + \Lambda\left||F_j| - |B^\lambda|\right| + \frac{P_\lambda(B^\lambda(|B^\lambda|))}{4\hat\gamma^2(n, \lambda)} \alpha_\lambda^2(E_j). \end{split} \end{equation}
Therefore, we have that \begin{equation*} |\alpha_\lambda(F_j) - \alpha_\lambda(E_j)| \le \frac{P_\lambda(B^\lambda(|B^\lambda|))}{4\hat\gamma^2(n, \lambda)} \alpha_\lambda^2(E_j). \end{equation*}
Since $\alpha_\lambda(E_j) \to 0$ we get that \begin{equation*} \frac{\alpha_\lambda(F_j)}{\alpha_\lambda(E_j)} \to 1. \end{equation*}
Let $\{\hat\lambda_j\} \subset (0, \infty)$ such that, setting $\tilde F_j := \hat\lambda_j F_j$, then $|\tilde F_j| = |B^\lambda|$.
Clearly $\hat\lambda_j \to 1$ since $|F_j|\to |B^\lambda|$.
Since $P_\lambda(F_j) \to P_\lambda(B^\lambda(|B^\lambda|))$ and $\Lambda > n$, for $j$ sufficiently large  we have $P_\lambda(F_j) < \Lambda |F_j|$ and \begin{equation*} \begin{split} \left|P_\lambda(\tilde F_j) - P_\lambda(F_j)\right| & = P_\lambda(F_j) \left|\hat\lambda_j^{n - 1} - 1\right|  \le P_\lambda(F_j) \left|\hat\lambda_j^n - 1\right|  \le \Lambda \left|\hat\lambda_j^n - 1\right|\,|F_j|  = \Lambda \left||\tilde F_j| - |F_j|\right|. \end{split} \end{equation*}
Hence, by definition of $\hat\lambda_j$ and by \eqref{dispense:5.13} we get
\begin{equation}\label{dispense:5.14} \begin{split} P_\lambda(\tilde F_j) & \le P_\lambda(F_j) + \Lambda\left||\tilde F_j| - |F_j|\right| = P_\lambda(F_j) + \Lambda \left||F_j| - |B^\lambda|\right| 
\\&\overset{\eqref{dispense:5.13}}{\le} P_\lambda(B^\lambda(|B^\lambda|)) + \frac{P_\lambda(B^\lambda(|B^\lambda|))}{4\hat\gamma^2(n, \lambda)} \alpha_\lambda^2(E_j). \end{split} \end{equation}
Since $\alpha_\lambda(F_j)/\alpha_\lambda(E_j) \to 1$ as $j \to \infty$ we have $\alpha_\lambda(E_j)^2 < 2 \alpha_\lambda (\tilde F_j)^2$ for $j$ sufficiently large.
Hence from~\eqref{dispense:5.14} we finally obtain \begin{equation}\label{selection:contradiction2}
\alpha_\lambda(\tilde F_j) > \sqrt{2} \hat\gamma(n, \lambda)\sqrt{D_\lambda(\tilde F_j)}.
\end{equation}

{
For $t>0$ let 
\begin{equation*} 
\phi_{\tilde F_j}^-(t) := 
\begin{cases}
    \min_{x \in \partial\tilde F_j \cap \{x_n = t\}} \left\{\left|x - t e_n\right|\right\} \qquad & {\rm if} \quad \partial\tilde F_j \cap \{x_n = t\}\neq \emptyset,\\
    0 &{\rm if} \quad \partial\tilde F_j \cap \{x_n = t\}= \emptyset.
\end{cases}
\end{equation*}
be the function measuring the distance of $\partial\tilde F_j \cap \{x_n = t\}$ from the $n$-th axis, set to zero in case $\partial\tilde F_j \cap \{x_n = t\}= \emptyset$. For $j$ large we can apply \cref{cor:DeficitBassoFetteGrosse} again to deduce that there exists $T_\lambda, A_\lambda>0$ such that $\hh^{n-1}( \tilde F_j \cap \{x_n =t\} ) \ge A_\lambda$ for almost every $t \in (0,T_\lambda)$. Since $\tilde F_j$ is Schwarz-symmetric and its relative boundary in $\{x_n>0\}$ is $C^1$ regular, then we can write that $\phi_{\tilde F_j}^-(t) \ge A'_\lambda >0$ for $j$ large and for any $t \in (0,T_\lambda)$.
}

Recalling that $F_j$ is a local $(\Lambda_1,r_0)$-minimizer, by \cref{prop:meancurvature} its boundary has generalized mean curvature bounded by $\Lambda_1$ for any $j$.
Since $\tilde F_j= \hat{\lambda}_j F_j$ with $\hat{\lambda}_j\to1$, then $\partial \tilde F_j \cap (\R^n\setminus H)$ is a hypersurface of class $C^{1,\frac12}$ with generalized mean curvature $H_{\partial \tilde F_j}$ bounded by $2 \Lambda_1$ for any $j$. Observe that if we locally parametrize $\partial \tilde F_j \cap (\R^n\setminus H)$ with the graph of a function $\Phi_j$, then $\Phi_j$ weakly solves the mean curvature equation
\begin{equation*}
 {\rm div} \left(\frac{\nabla\Phi_j}{\sqrt{1 + |\nabla\Phi_j|^2}}\right) = \Braket{ H_{\partial \tilde F_j} , N_{\Phi_j} },
 \end{equation*}
where $ N_{\Phi_j}$ is the unit normal corresponding to $\Phi_j$ and $H_{\partial \tilde F_j}$ is evaluated along the graph of $\Phi_j$. Since $H_{\partial \tilde F_j}$ is bounded, we get that $\Phi_j$ is of class $W^{2, p}$ for every $p<\infty$ (see \cite{GilbargTrudinger}).

Fix $p_0 \in \partial \tilde F_j \cap \{x_1>0, 0<x_n\le T_\lambda\} \cap {\rm span}\{e_1,e_n\}$. Since $\partial \tilde F_j \cap (\R^n\setminus H)$ is $C^{1,\frac12}$, there exists a curve $\gamma_j=(\alpha_j, 0,\ldots, 0, \beta_j):(a,b)\to {\rm span}\{e_1,e_n\} \setminus H$ such that the map $\mathbb{S}^{n-2}\times(a,b) \ni (\theta, t) \mapsto (\alpha_j(t)\theta, \beta_j(t))$ parametrizes $\partial \tilde F_j$ in a neighborhood of $p_0$. We claim that $\alpha_j ,\beta_j \in W^{2,p}$, up to reparametrization.\\
Indeed, we can also parametrize $\partial \tilde F_j$ in a neighborhood $U$ of $p_0$ as the graph of a function $\Phi_j$ with domain contained in some affine hyperplane of the form $p_0 + V$, and without loss of generality we can assume that either $V=\{x_1=0\}$ or $V=\{x_n=0\}$. If $n=2$, then the claimed regularity immediately follows from the regularity of $\Phi_j$. Then assume $n\ge 3$, and suppose for example that $V=\{x_1=0\}$.
The image of the curve $\gamma_j$ in $U$ can be parametrized as the graph of a function $t\mapsto(f(t), 0,\ldots,0,t)$. Writing as $(x',x_n)\in p_0+V$ the variable for $\Phi_j$, the 
fact that the distance from the $n$-th axis is constant on the intersection of $\partial \tilde F_j$ with any horizontal hyperplane yields the identity
\[
\left( \Phi_j(x',x_n) + {\rm dist}_{x_n}(p_0) \right)^2 + |x'|^2 =  f(x_n)^2,
\]
where ${\rm dist}_{x_n}(p_0)$ denotes distance of $p_0$ from the $n$-th axis. Since $\Phi_j$ is of class $C^1$ and $W^{2,p}$ and ${\rm dist}_{x_n}(p_0)>0$
because $\varphi_{\tilde F_j}^-\ge A'_\lambda>0$, inverting the above identity we find that $f$ is of class $W^{2,p}$, hence so is $\gamma_j$, up to reparametrization. In case $V=\{x_n=0\}$, the observation follows analogously relating $\Phi_j$ with a parametrization for $\gamma_j$.

We further observe that, for $\alpha_j$, $\beta_j : (a, b) \to (0, \infty)$ as above, since $\alpha_j$, $\beta_j$ are of class $W^{2,p}_{\rm loc}$, up to reparametrization by arclength we can apply \cref{curvature:revolution} to get that
\begin{equation*}
H_{\partial \tilde F_j} \big|_{(\alpha_j(t)\theta, \beta_j(t))}= \left(\Braket{k_{\gamma_j}, \nu} - (n - 2) \frac{\beta'_j}{\alpha_j}\right) (- \beta'_j\theta, \alpha'_j),
\end{equation*} 
in the notation of \cref{curvature:revolution}.
Recalling that $\varphi_{\tilde F_j}^-\ge A'_\lambda>0$ on $(0,T_\lambda)$, we have that $|\alpha_j|\ge A'_\lambda$ and thus
\begin{equation}\label{bounded:curvature}
|k_{\gamma_j}|  \le 2 \Lambda_1 + \frac{n-2}{A'_\lambda}.
\end{equation}
Observe that the upper bound in \eqref{bounded:curvature} is independent of $j$ and of the initially chosen point $p_0$.

Fix now $q_0 \in \partial  \tilde F_j \cap \{x_1>0, x_n= T_\lambda\} \cap {\rm span}\{e_1,e_n\}$, let $\gamma_j^0:[0,l_0)\to {\rm span}\{e_1,e_n\}$ be part of a curve defined as before, parametrized by arclength, such that $\Braket{\gamma_j^0(t), e_n} \le T_\lambda$ for any $t$. If $\lim_{t\to l_0^-}\gamma_j^0(t) \not\in \partial H$, the curve can be extended to a longer one, parametrized by arclength, by joining $\gamma_j^0$ with a curve defined as before for the choice $p_0= \lim_{t\to l_0^-}\gamma_j^0(t)$. Hence we can consider $\sigma_j:[0,L_j)\to {\rm span}\{e_1,e_n\}$ the maximal extension of $\gamma_j^0$ parametrized by arclength that parametrizes $\tilde F_j \cap \{x_1>0, 0<x_n\le T_\lambda\} \cap {\rm span}\{e_1,e_n\}$. Since the perimeter $P(\tilde F_j, \R^n\setminus H)$ is uniformly bounded, then $\sup_j L_j<+\infty$. Obviously $\lim_{t\to L_j}\sigma_j(t) \in \partial H$, for otherwise the curve could be further extended. By construction, the uniform bound in \eqref{bounded:curvature} holds pointwise for the curvature of $\sigma_j$. Therefore $\sigma_j$ can be extended to a curve $\gamma_j:[0,L_j]\to {\rm span}\{e_1,e_n\}$ such that
\begin{equation}\label{eq:BoundC11UpToBoundary}
\|\gamma_j\|_{C^{1, 1}([0, L_j])} \le C,
\end{equation}
with $C$ independent of $j$, depending only on $n,\lambda, \tilde l$ and the upper bound on the curvature given by \eqref{bounded:curvature}.

Up to a subsequence, since $\tilde F_j\to B^\lambda(|B^\lambda|)$ and we already know that for every $\epsilon > 0$ the set $\partial \tilde F_j \cap \{x_n \ge \epsilon\}$ converges to $\partial B^\lambda(|B^\lambda|) \cap \{x_n \ge \epsilon\}$ in $C^{1, \alpha}$ for any $0 < \alpha < \frac{1}{2}$, the bound \eqref{eq:BoundC11UpToBoundary} implies that $\tilde F_j$ converges in $C^1$ sense to $B^\lambda(|B^\lambda|)$ in the sense of \cref{def:VicinanzaC1Schwarz}.
Hence, by \cref{quantitative:nearball}, for $j$ sufficiently large there holds
\begin{equation*}
\alpha_\lambda^2(\tilde F_j) \le \hat\gamma(n , \lambda) D_\lambda(\tilde F_j),
\end{equation*}
in contradiction with~\eqref{selection:contradiction2}.
\end{proof}

\section{Second quantitative isoperimetric inequality}\label{sec:SecondQuantitative}

We will need the following technical lemma, proving that if the energy $P_\lambda(E_i)$ a sequence of sets $E_i$ converges to the energy of the limit, then the sequence strictly converges in the sense of $BV$ functions. The proof essentially follows by analyzing the equality case in the Reshetnyak lower semicontinuity theorem, see \cite[Theorem 2.38]{AmbrosioFuscoPallara}.

\begin{lemma}\label{lem:ConvergenzaForteSuccessioni}
Let $\{E_i\}_{i \in \N}$ be a sequence of sets of finite perimeter in $\R^n\setminus H$ such that $E_i\to E$ in $L^1$, for some set of finite perimeter $E$ with $|E|<+\infty$. If $P_\lambda(E_i)\to P_\lambda(E)$, then
\[
\lim_i P(E_i, \R^n\setminus H) = P(E, \R^n\setminus H) ,
\qquad
\lim_i \hh^{n-1}(\partial^*E_i \cap \partial H) =  \hh^{n-1}(\partial^*E \cap \partial H) .
\]
\end{lemma}

\begin{proof}
Let $f(v):= |v| - \lambda\braket{e_n,v}$, for any $v \in \R^n$, and let $\nu_i:= |D\chi_{E_i}| \otimes \delta_{\nu^{E_i}}$ be a measure on $\R^n\setminus H \times \mathbb{S}^{n-1}$. Since $|\nu_i|(\R^n\setminus H \times \mathbb{S}^{n-1}) \le P(E_i) \le 2P_\lambda(E_i)/(1-\lambda)$ by \cref{remark:positivity}, up to subsequence, $\nu_i$ weakly* converges to a finite measure $\nu$. Up to subsequence, also $|D\chi_{E_i}|$ weakly* converges to a finite measure $\mu$ on $\R^n\setminus H$.
Denoting by $\pi:\R^n\setminus H \times \mathbb{S}^{n-1}\to \R^n\setminus H$ the natural projection, we have $\pi_\sharp \nu_i = |D\chi_{E_i}| \to \mu = \pi_\sharp \nu$. Moreover, $\mu \ge |D\chi_E|$ by lower semicontinuity. By the disintegration theorem \cite[Theorem 2.28]{AmbrosioFuscoPallara}, we can write $\nu= \mu \otimes \nu_x$, for a $\mu$-measurable map $\R^n\setminus H \ni x\mapsto \nu_x$, where $\nu_x$ is a probability measure on $\mathbb{S}^{n-1}$. Analogously to \cite[Eq. (2.30)]{AmbrosioFuscoPallara}, we observe that
\begin{equation}\label{bbb}
    \int_{\mathbb{S}^{n-1}} v \de \nu_x(v) = \nu^E(x) \frac{|D\chi_E|}{\mu}(x),
\end{equation}
at $\mu$-a.e. $x \in \R^n\setminus H$. Indeed, for any continuous function $g$ with ${\rm spt}(g) \subset\subset \R^n\setminus H$ we find
\begin{equation*}
    \begin{split}
        \int_{\R^n\setminus H} g(x) &\int_{\mathbb{S}^{n-1}} v \de \nu_x(v) \de \mu(x) 
        =
        \int_{\R^n\setminus H \times \mathbb{S}^{n-1}} g(x) \, v \de \nu(x,v)
        = \lim_i
         \int_{\R^n\setminus H \times \mathbb{S}^{n-1}} g(x) \, v \de \nu_i(x,v) \\
         &=- \lim_i
         \int_{\R^n\setminus H} g(x) \de D\chi_{E_i}(x)
         =
         - \int_{\R^n\setminus H} g(x) \de D\chi_{E}(x)
         = \int_{\R^n\setminus H} g(x) \nu^E(x) \frac{|D\chi_E|}{\mu}(x) \de \mu(x).
    \end{split}
\end{equation*}
Since $f$ is nonnegative, convex and continuous, by \cref{rem:PlambdaConDivergenza} we find
\begin{equation}\label{bbb2}
\begin{split}
    \lim_i P_\lambda(E_i) &=
    \lim_i \int_{\R^n\setminus H} f(\nu^{E_i}) \de |D\chi_{E_i}| 
    =
    \lim_i \int_{\R^n\setminus H\times \mathbb{S}^{n-1}} f(v) \de \nu_i(x,v)
    \ge \int_{\R^n\setminus H\times \mathbb{S}^{n-1}} f(v) \de \nu(x,v) 
    \\&
    = \int_{\R^n\setminus H} \int_{\mathbb{S}^{n-1}} f(v) \de \nu_x(v) \de \mu(x)
    \ge 
     \int_{\R^n\setminus H} f \left( \int_{\mathbb{S}^{n-1}} v \de \nu_x(v) \right)\de \mu(x) \\
     &
     \overset{\eqref{bbb}}{=}   
        \int_{\R^n\setminus H} f \left( \nu^E(x) \frac{|D\chi_E|}{\mu}(x) \right)\de \mu(x) =
         \int_{\R^n\setminus H} f( \nu^E(x) ) \de |D\chi_E|(x) = P_\lambda(E),
\end{split}  
\end{equation}
where in the second inequality we applied Jensen inequality, and where the last equality follows since $f$ is positively $1$-homogeneous. Since $\lim_i P_\lambda(E_i)=P_\lambda(E)$ by assumption and since $f$ is not affine, equality in Jensen inequality implies that the identity map $\mathbb{S}^{n-1}\ni v \mapsto v$ is constant $\nu_x$-a.e., for $\mu$-a.e. $x\in \R^n\setminus H$. This means that $\nu_x=\delta_{v_x}$ for some $v_x \in \mathbb{S}^{n-1}$ for $\mu$-a.e. $x\in \R^n\setminus H$. Hence \eqref{bbb} implies
\[
v_x = \nu^E(x) \frac{|D\chi_E|}{\mu}(x),
\]
at $\mu$-a.e. $x \in \R^n\setminus H$, and since $|v_x|=|\nu^E(x)|=1$, then $|D\chi_E|/\mu(x)=1$ at $\mu$-a.e. $x \in \R^n\setminus H$, and $v_x=\nu^E(x)$ $\mu$-almost everywhere. Inserting in \eqref{bbb2} we deduce
\[
\int_{\R^n\setminus H} f \left(\nu^E(x) \right)\de \mu(x) = 
\int_{\R^n\setminus H} f \left( \int_{\mathbb{S}^{n-1}} v \de \nu_x(v) \right)\de \mu(x) = \int_{\R^n\setminus H} f( \nu^E(x) ) \de |D\chi_E|(x) .
\]
Since $f(\nu^E(x))>0$ and $\mu\ge |D\chi_E|$, we deduce that $\mu=|D\chi_E|$, and then $|D\chi_{E_i}|$ weakly* converges to $|D\chi_E|$.

We can now fix an increasing sequence of Lipschitz bounded open sets $\Omega_j \subset\subset \R^n\setminus H$ such that $\cup_j\Omega_j = \R^n\setminus H$ and $P(E_i, \partial\Omega_j) = P(E, \partial\Omega_j) =0$ for every $i,j$. Hence $\lim_i P(E_i, \Omega_j) = P(E, \Omega_j)$ for any $j$. Moreover
\[
P(E_i, \R^n\setminus(H\cup \Omega_j))
\le \frac{1}{1-|\lambda|}\left( P_\lambda(E_i) - \int_{\Omega_j} f(\nu^{E_i}) \de |D\chi_{E_i}| \right),
\]
for any $i,j$.
Applying Reshetnyak continuity theorem \cite[Theorem 2.39]{AmbrosioFuscoPallara} on $\Omega_j$ we get
\[
\limsup_i P(E_i, \R^n\setminus(H\cup \Omega_j) )
\le \frac{1}{1-|\lambda|} \int_{\R^n\setminus(H\cup \Omega_j)} f(\nu^{E}) \de |D\chi_{E}| \le \frac{1+|\lambda|}{1-|\lambda|} P(E, \R^n\setminus(H\cup \Omega_j)),
\]
for any $j$. Therefore
\[
\limsup_i P(E_i, \R^n\setminus H) \le P(E, \Omega_j) + \frac{1+|\lambda|}{1-|\lambda|} P(E,\R^n\setminus(H\cup \Omega_j)),
\]
for any $j$. Letting $j\to\infty$, the proof follows.
\end{proof}

We will also exploit the concept of $(K,r_0)$-quasiminimal set.

\begin{definition}\label{def:QuasiminimalSets}
    Let $E\subset \R^n\setminus H$ be a set of finite perimeter with finite measure, and let $K\ge1, r_0>0$. We say that $E$ is a $(K,r_0)$-quasiminimal set (relatively in $\R^n\setminus H$) if
    \[
    P(E,\R^n\setminus H) \le K P(F, \R^n\setminus H),
    \]
    for any $F\subset \R^n\setminus H$ such that $E\Delta F \subset\subset B_r(x)$, for some ball $B_r(x)\subset \R^n$ with $r \le r_0$ and $x \in \{x_n\ge0\}$.
\end{definition}

Quasiminimal sets have well-known topological regularity properties following from uniform density estimates at boundary points. We recall these facts in the following statement.
The proof follows, for example, by repeatedly applying \cite[Theorem 4.2]{ShanmugalingamQuasiminimizers} with $X=\{x_n\ge0\}$ in domains $\Omega=X \cap B_{r_0}(x)$ for $x \in X$, in the notation of \cite[Theorem 4.2]{ShanmugalingamQuasiminimizers}. Observe that in \cite{ShanmugalingamQuasiminimizers}, the perimeter functional coincides with the relative perimeter in $\R^n\setminus H$, hence the definition of quasiminimal set in \cite[Definition 3.1]{ShanmugalingamQuasiminimizers} coincides with our \cref{def:QuasiminimalSets}. Alternatively, the proof follows by adapting the proof of \cite[Theorem 21.11]{MaggiBook} working with $(K,r_0)$-quasiminimal sets instead of $(\Lambda,r_0)$-minimizers.

\begin{theorem}\label{thm:RegolaritaQuasiminimal}
    Let $E\subset \R^n\setminus H$ be a $(K,r_0)$-quasiminimal set, for some $K\ge1, r_0>0$. Then there exist $m=m(n,K,r_0)\in(0,1)$ and $r_0'=r_0'(n,K,r_0)\in(0,r_0]$ such that
    \begin{equation*}
        m\le\frac{|E \cap B_r(x)|}{|B_r(x) \setminus H|} \le 1-m \qquad \forall x \in \overline{\partial E \setminus H}, \,\, \forall\, r \in (0,r_0'].
    \end{equation*}
    In particular the set $E^{(1)}$ of points of density $1$ for $E$ is an open representative for $E$.
\end{theorem}

We will identify a $(K,r_0)$-quasiminimal set with its open representative $E^{(1)}$. In order to prove \cref{thm:QuantitativaBagnata} we need two preparatory lemmas.

\begin{lemma}\label{lem:StimaBetaQuasiminimal}
    For any $K\ge 1, r_0>0$ there exist $\delta_6, C_5, C_6>0$ depending on $n,\lambda, K,r_0$ such that the following holds. If $E\subset \R^n\setminus H$ is a bounded $(K,r_0)$-quasiminimal set with $|E|=|B^\lambda|$ and $D_\lambda(E) \le \delta_6$, then
    \begin{equation}\label{eq:StimaDistHaus}
        d_\hh\left(\overline{\partial E \setminus H}, \overline{\partial B^\lambda(|B^\lambda|, x) \setminus H} \right) \le C_5 \alpha_\lambda(E)^{\frac1n},
    \end{equation}
    where $B^\lambda(|B^\lambda|, x)$ is a bubble realizing the asymmetry of $E$. Moreover
    \begin{equation}\label{eq:StimaBetaQuasiminimalSets}
        \beta_\lambda(E) \le C_6 D_\lambda(E)^{\frac{1}{2n}}.
    \end{equation}
\end{lemma}

\begin{proof}
Up to translation, we can assume that $x=0$. Also, letting $m, r_0'$ be given by \cref{thm:RegolaritaQuasiminimal}, up to decreasing $r_0$ we can assume that $r_0=r_0'$.
Let $p \in \overline{\partial E \setminus H}$ be such that
\[
d_0:= {\rm dist}\left(p,  \overline{\partial B^\lambda(|B^\lambda|) \setminus H}\right) = \max \left\{ {\rm dist}\left(y,  \overline{\partial B^\lambda(|B^\lambda|) \setminus H}\right) \st y \in \overline{\partial E \setminus H} \right\}.
\]
Hence $B_{d_0}(p) \cap \overline{\partial B^\lambda(|B^\lambda|) \setminus H}=\emptyset$. Then either $B_{d_0}(p) \setminus H \subset B^\lambda(|B^\lambda|)$ or $B_{d_0}(p) \setminus H \subset \R^n \setminus (H \cup B^\lambda(|B^\lambda|))$. In the first case \cref{thm:RegolaritaQuasiminimal} implies
\[
m|B_r(x)\setminus H| \le | B_r(x) \setminus (H \cup E)| \le |B^\lambda(|B^\lambda|) \setminus E| = \frac12 \alpha_\lambda(E) \qquad \forall\, r \in (0, \min\{d_0, r_0\}),
\]
while in the second case \cref{thm:RegolaritaQuasiminimal} implies
\[
m|B_r(x)\setminus H| \le | B_r(x) \cap E| \le |E \setminus B^\lambda(|B^\lambda|)| = \frac12 \alpha_\lambda(E)
\qquad \forall\, r \in (0, \min\{d_0, r_0\}).
\]
Since $|B_R(x) \setminus H| \ge C r^n$, then $\min\{d_0, r_0\}^n \le C\alpha_\lambda(E)$, for $C=C(n,\lambda,K,r_0)$. So by \cref{lemma:fmp2.3}, choosing $\delta_6$ small enough we have that $\alpha_\lambda(E)$ is so small that $\min\{d_0, r_0\}=d_0$ and then
\[
d_0^n \le C \alpha_\lambda(E).
\]
Since density estimates as those in \cref{thm:RegolaritaQuasiminimal} hold for $B^\lambda(|B^\lambda|)$, repeating the above argument exchanging the roles of $E$ and $B^\lambda(|B^\lambda|)$, \eqref{eq:StimaDistHaus} follows.

From \eqref{eq:StimaDistHaus}, we deduce that
\[
\overline{\partial E \setminus H} \subset \left\{ y \in \{x_n\ge0\} \st {\rm dist}\left(y,  \overline{\partial B^\lambda(|B^\lambda|) \setminus H}\right) \le C_5 \alpha_\lambda(E)^{\frac1n}\right\}.
\]
Hence
\[
\begin{split}
    \hh^{n-1}&\left( \partial^*E \Delta \partial B^\lambda(|B^\lambda|) \cap \partial H \right) \\&\le  \hh^{n-1}\left( \left\{
    (x',0)\in\R^n \st (1-\lambda^2)^{\frac12} - C_5 \alpha_\lambda(E)^{\frac1n} \le |x'|
    \le  (1-\lambda^2)^{\frac12} + C_5 \alpha_\lambda(E)^{\frac1n} 
    \right\} \right)
    \\&
    \le C \alpha_\lambda(E)^{\frac1n} \le C D_\lambda(E)^{\frac{1}{2n}},
\end{split}
\]
for some $C=C(n,\lambda,K,r_0)$, where we used \cref{thm:FinalQuantitativeInequality} in the last inequality. Hence \eqref{eq:StimaBetaQuasiminimalSets} follows.
\end{proof}

\begin{lemma}\label{lem:QuantitativaBetaDeficitBasso}
    There exists $\delta_7, C_7>0$ depending on $n,\lambda$ such that for any measurable set $E\subset\R^n\setminus H$ with $|E|=|B^\lambda|$ and $D_\lambda(E) \le \delta_7$ there holds
    \begin{equation}
        \beta_\lambda(E) \le C_7 D_\lambda(E)^{\frac{1}{2n}}.
    \end{equation}
\end{lemma}

\begin{proof}
Fix $\Lambda>n$. Let $Q\subset \R^n$ be a large cube whose interior contains the closure of $B^\lambda(|B^\lambda|)$, and let $F\subset Q \setminus H$ be such that $|B^\lambda|/2 \le |F| \le 2 |B^\lambda|$. Let $G\subset\R^n\setminus H$ be such that $G \Delta F \subset\subset B_{r_0}(x)$, for $x \in \{x_n\ge0\}$ and $r_0\in(0,1)$ to be chosen small. Let $Z := G \cap Q$. Observe that
\begin{equation}\label{eq:zzqqa}
    \begin{split}
        \left||Z| - |F| \right|
        & \le |Z \Delta F| 
        \le |G \Delta F|^{\frac1n} |G \Delta F|^{\frac{n-1}{n}} \le \omega_n^{\frac1n} r \, \left(|G|^{\frac{n-1}{n}} + |F|^{\frac{n-1}{n}} \right)\\
        &\le  \overline{C}(n) r_0  \left(P(G,\R^n\setminus H) + P(F, \R^n\setminus H )\right),
    \end{split}
\end{equation}
where in the last inequality we used the relative isoperimetric inequality in a half-space (see \cite{ChoeGhomiRitoreInequality} for the sharp inequality). Let $y,z \in \partial H$ be such that
\[
\beta_\lambda(F) = \frac{\hh^{n-1}\left( \partial^* F \Delta \partial B^\lambda(|F|,y) \cap \partial H\right)}{\hh^{n-1}\left( \partial B^\lambda(|F|,y) \cap \partial H\right)},
\qquad
\beta_\lambda(Z) = \frac{\hh^{n-1}\left( \partial^* Z \Delta \partial B^\lambda(|Z|,z) \cap \partial H\right)}{\hh^{n-1}\left( \partial B^\lambda(|Z|,z) \cap \partial H\right)},
\]
Observe that $y,z$ exist since $F, Z \subset Q$. Suppose, for instance, that $\beta_\lambda(Z)\ge \beta_\lambda(F)$. Then
\begin{equation}\label{eq:zzqqa1}
    \begin{split}
        &\beta_\lambda(Z)-\beta_\lambda(F)
        \le
        \frac{\hh^{n-1}\left( \partial^* Z\Delta \partial B^\lambda(|Z|,y) \cap \partial H\right)}{\hh^{n-1}\left( \partial B^\lambda(|Z|,y) \cap \partial H\right)}
        - 
        \frac{\hh^{n-1}\left( \partial^* F \Delta \partial B^\lambda(|F|,y) \cap \partial H\right)}{\hh^{n-1}\left( \partial B^\lambda(|F|,y) \cap \partial H\right)}
        \\
        & \le 
        \frac{
        \hh^{n-1}\left( \partial^* Z\Delta \partial^* F \cap \partial H\right)
        +
        \hh^{n-1}\left( \partial^* F\Delta \partial B^\lambda(|F|,y) \cap \partial H\right)
        +
        \hh^{n-1}\left( \partial B^\lambda(|F|,y) \Delta \partial B^\lambda(|Z|,y) \cap \partial H\right)       
        }{\hh^{n-1}\left( \partial B^\lambda(|Z|,y) \cap \partial H\right)}
        + \\
        &\qquad
        - 
        \frac{\hh^{n-1}\left( \partial^* F \Delta \partial B^\lambda(|F|,y) \cap \partial H\right)}{\hh^{n-1}\left( \partial B^\lambda(|F|,y) \cap \partial H\right)} \\
        &= 
        \frac{
        \hh^{n-1}\left( \partial^* Z\Delta \partial^* F \cap \partial H\right)      
        }{\hh^{n-1}\left( \partial B^\lambda(|Z|,y) \cap \partial H\right)}
        + 
        \frac{
        \left|
        \hh^{n-1}\left( \partial B^\lambda(|F|,y) \cap \partial H \right)
        -
        \hh^{n-1}\left( \partial B^\lambda(|Z|,y) \cap \partial H\right) 
        \right|       
        }{\hh^{n-1}\left( \partial B^\lambda(|Z|,y) \cap \partial H\right)} + \\
        &\qquad
        + 
        \hh^{n-1}\left( \partial^* F \Delta \partial B^\lambda(|F|,y) \cap \partial H\right) \left( 
        \frac{1}{\hh^{n-1}\left( \partial B^\lambda(|Z|,y) \cap \partial H\right)}
        -
        \frac{1}{\hh^{n-1}\left( \partial B^\lambda(|F|,y) \cap \partial H\right)}
        \right)
    \end{split}
\end{equation}
By a trace inequality \cite[Theorem 3.87]{AmbrosioFuscoPallara} we estimate
\[
\begin{split}
    \hh^{n-1}\left( \partial^* Z\Delta \partial^* F \cap \partial H\right) &\le
    C(n)\left(   
    |Z \Delta F| + P(Z, \R^n\setminus H) + P(F, \R^n\setminus H) 
    \right)\\
    &\le C(n)\left(   
    |G \Delta F| + P(G, \R^n\setminus H) + P(F, \R^n\setminus H) 
    \right)\\
    &\le C(n)\left(   
     P(G, \R^n\setminus H) + P(F, \R^n\setminus H) 
    \right),
\end{split}
\]
where the last inequality follows as in \eqref{eq:zzqqa}, and $C$ denotes a constant depending on suitable parameters that changes from line to line. For $r_0$ small, depending only on $n,\lambda$, we can ensure that
\[
\hh^{n-1}\left( \partial B^\lambda(|Z|,y) \cap \partial H\right) \ge C(n,\lambda)>0.
\]
Finally
\[
\begin{split}
    \left|
        \hh^{n-1}\left( \partial B^\lambda(|F|,y) \cap \partial H \right)
        -
        \hh^{n-1}\left( \partial B^\lambda(|Z|,y) \cap \partial H\right) 
        \right| 
        & \le
        L \left| |Z| - |F| \right| \\
        &\overset{\eqref{eq:zzqqa}}{\le} C(n,\lambda) \left(   
     P(G, \R^n\setminus H) + P(F, \R^n\setminus H) 
    \right),
\end{split}
\]
for a suitable Lipschitz constant $L=L(n,\lambda)$. Therefore \eqref{eq:zzqqa1} becomes
\[
\beta_\lambda(Z)-\beta_\lambda(F)
        \le C(n,\lambda) \left(   
     P(G, \R^n\setminus H) + P(F, \R^n\setminus H) 
    \right).
\]
In case $\beta_\lambda(Z)<\beta_\lambda(F)$, the very same argument leads to an analogous estimate. Hence
\begin{equation}\label{eq:zzqqa2}
    \left|\beta_\lambda(Z)-\beta_\lambda(F) \right|
        \le \widetilde{C}(n,\lambda) \left(   
     P(G, \R^n\setminus H) + P(F, \R^n\setminus H) 
    \right).
\end{equation}

Up to taking a smaller $r_0$, we fix $r_0=r_0(n,\lambda, \Lambda)\in(0,1)$ and $\epsilon_0=\epsilon_0(n,\lambda, \Lambda)\in(0,1)$ such that $1-|\lambda| - \epsilon_0 \widetilde C(n,\lambda) - \Lambda \overline{C}(n) r_0 >0$, and we define
\begin{equation*}
    K:= \frac{1+|\lambda| + \epsilon_0 \widetilde C(n,\lambda) + \Lambda \overline{C}(n) r_0 }{1-|\lambda| - \epsilon_0 \widetilde C(n,\lambda) - \Lambda \overline{C}(n) r_0 } >1.
\end{equation*}

Let $\delta_6, C_6$ be given by \cref{lem:StimaBetaQuasiminimal} corresponding to the parameters $K, r_0/2$. We want to prove that if $\delta_7$ is sufficiently small, then
\[
\beta_\lambda(E) \le 2 C_6 D_\lambda(E)^{\frac{1}{2n}}.
\]
We argue by contradiction assuming that there exist sets $E_j \subset \R^n\setminus H$ with $|E_j|=|B^\lambda|$ and $D_\lambda(E_j)\le 1/j$ such that
\begin{equation}
    \beta_\lambda(E_j) > 2 C_6 D_\lambda(E_j)^{\frac{1}{2n}},
\end{equation}
for any $j$. Up to translation, $E_j\to B^\lambda(|B^\lambda|,0)$ and $P_\lambda(E_j)\to P_\lambda(B^\lambda)$. Since the trace operator is continuous with respect to strict convergence of $BV$ functions, see \cite[Theorem 3.88]{AmbrosioFuscoPallara}, by \cref{lem:ConvergenzaForteSuccessioni} we deduce that $\beta_\lambda(E_j)\to0$.\\
Let $F_j$ be a minimizer of the problem
\begin{equation}\label{eq:ProbMinimoQuasiminimal}
    \min \left\{ 
    P_\lambda(E) + \epsilon_0 |\beta_\lambda(E) - \beta_\lambda(E_j) | + \Lambda \left| |E| - |E_j| \right| \st E \subset Q
    \right\}.
\end{equation}
By \cref{lemma:precompattezza}, up to subsequence $F_j$ converges to a limit set $F$ in $L^1$. If by contradiction $\beta_j(F_j)\not\to 0$, by \cref{dispense:lemma5.3} for large $j$ we would have that
\[
P_\lambda(B^\lambda(|B^\lambda|)) + \epsilon_0 \beta_\lambda(E_j) < P_\lambda(F_j) + \epsilon_0 |\beta_\lambda(F_j) - \beta_\lambda(E_j) | + \Lambda \left| |F_j| - |E_j| \right| ,
\]
contradicting the minimality of $F_j$. Hence $\beta_\lambda (F_j)\to0$. It follows that, up to translation, $F_j$ converges to $B^\lambda(|B^\lambda|)$ in $L^1$. Comparing with $E_j$, we also see that $P_\lambda(F_j) \to P_\lambda(B^\lambda)$.

We want to show that $F_j$ is $(K,r_0)$-quasiminimal for $j$ large. Indeed, $|B^\lambda|/2 \le |F_j| \le 2 |B^\lambda|$ for $j$ large. Hence we can apply \eqref{eq:zzqqa} and \eqref{eq:zzqqa2} with $F=F_j$. Letting $G\subset\R^n\setminus H$ such that $G \Delta F \subset\subset B_{r_0}(x)$, for $x \in \{x_n\ge0\}$, denoting $Z := G \cap Q$, by minimality of $F_j$ for \eqref{eq:ProbMinimoQuasiminimal} we find
\begin{equation*}
    \begin{split}
         (1-|\lambda|)&P(F_j, \R^n\setminus H) 
         \le 
          (1+|\lambda|)P(Z, \R^n\setminus H) 
          +\epsilon_0\left|\beta_\lambda(Z)-\beta_\lambda(F_j) \right| + \Lambda \left| |Z| - |F_j| \right| \\
          &\le 
          (1+|\lambda|)P(G, \R^n\setminus H) 
          +\left( \epsilon_0 \widetilde{C}(n,\lambda) + \Lambda \overline{C}(n) r_0 \right) \left(   
     P(G, \R^n\setminus H) + P(F_j, \R^n\setminus H) 
    \right),
    \end{split}
\end{equation*}
proving that $F_j$ is $(K,r_0)$-quasiminimal.

By minimality of $F_j$, we have
\begin{equation}\label{aaa}
\begin{split}
P_\lambda(F_j) &+ \Lambda \left||F_j| - |B^\lambda|\right| + \epsilon_0\left|\beta_\lambda(F_j) - \beta_\lambda(E_j)\right|  \le P_\lambda(E_j) \\ & \le P_\lambda(B^\lambda) + \frac{P_\lambda(B^\lambda)}{(2 C_6)^{2n}} \beta_\lambda^{2n}(E_j) 
\le P_\lambda(F_j) + \Lambda\left||F_j| - |B^\lambda|\right| + 
\frac{P_\lambda(B^\lambda)}{(2 C_6)^{2n}} \beta_\lambda^{2n}(E_j) 
\end{split}
\end{equation}
Therefore
\begin{equation*} |\beta_\lambda(F_j) - \beta_\lambda(E_j)| \le \frac{P_\lambda(B^\lambda)}{\epsilon_0(2 C_6)^{2n}} \beta_\lambda^{2n}(E_j) ,
\end{equation*}
and then
\[
\frac{\beta_\lambda(F_j)}{\beta_\lambda(E_j)} \to 1.
\]
Next we select $\{\hat\lambda_j\} \subset (0, \infty)$ such that, setting $\tilde F_j := \hat\lambda_j F_j$, then $|\tilde F_j| = |B^\lambda|$.
Clearly $\hat\lambda_j \to 1$ since $|F_j|\to |B^\lambda|$.
Since $P_\lambda(F_j) \to P_\lambda(B^\lambda))$ and $\Lambda > n$, for $j$ sufficiently large  we have $P_\lambda(F_j) < \Lambda |F_j|$ and \begin{equation*} \begin{split} \left|P_\lambda(\tilde F_j) - P_\lambda(F_j)\right| & = P_\lambda(F_j) \left|\hat\lambda_j^{n - 1} - 1\right|  \le P_\lambda(F_j) \left|\hat\lambda_j^n - 1\right|  \le \Lambda \left|\hat\lambda_j^n - 1\right|\,|F_j|  = \Lambda \left||\tilde F_j| - |F_j|\right|. \end{split} \end{equation*}
Hence, by definition of $\hat\lambda_j$ and by \eqref{aaa} we get
\begin{equation}\label{aaa1} \begin{split} P_\lambda(\tilde F_j) & \le P_\lambda(F_j) + \Lambda\left||\tilde F_j| - |F_j|\right| = P_\lambda(F_j) + \Lambda \left||F_j| - |B^\lambda|\right| 
\overset{\eqref{aaa}}{\le} P_\lambda(B^\lambda) + \frac{P_\lambda(B^\lambda)}{(2 C_6)^{2n}} \beta_\lambda^{2n}(E_j).
\end{split}
\end{equation}
Since $\beta_\lambda(F_j)/\beta_\lambda(E_j) \to 1$ as $j \to \infty$ and $\beta_\lambda$ is scale invariant, we have $\beta_\lambda(E_j)^{2n} < 2 \beta_\lambda (\tilde F_j)^{2n}$ for $j$ sufficiently large.
Hence from~\eqref{aaa1} we obtain
\[
\beta_\lambda(\tilde F_j)^{2n} \ge 2^{2n-1} C_6^{2n} D_\lambda(\tilde F_j),
\]
that is $\beta_\lambda(\tilde F_j)\ge 2^{1-\frac{1}{2n}} C_6 D_\lambda(\tilde F_j)^{\frac{1}{2n}}$. On the other hand, for $j$ large, $\tilde F_j$ is $(K, \hat \lambda_j r_0)$-quasiminimal. As $\hat \lambda_j \to 1$, then $\tilde F_j$ is $(K, r_0/2)$-quasiminimal for $j$ large. Moreover $D_\lambda(\tilde F_j)\to0$. By the choice of $C_6$ above, \cref{lem:StimaBetaQuasiminimal} implies that
\[
\beta_\lambda(\tilde F_j) \le C_6  D_\lambda(\tilde F_j)^{\frac{1}{2n}},
\]
giving a contradiction.    
\end{proof}

\begin{proof}[Proof of \cref{thm:QuantitativaBagnata}]
By \cref{lem:QuantitativaBetaDeficitBasso} it follows that for any $A>0$ there exists $C_A>0$ such that for any set $E\subset\R^n\setminus H$ with $|E|=|B^\lambda|$ and $\hh^{n-1}(\partial^*E \cap \partial H)\le A$ there holds
\begin{equation}\label{qqq}
    \beta_\lambda(E) \le C_A D_\lambda
    (E)^{\frac{1}{2n}}.
\end{equation}
Indeed, if $D_\lambda(E)\le \delta_7$, for $\delta_7$ as in \cref{lem:QuantitativaBetaDeficitBasso}, then \eqref{qqq} follows with $C_A=C_7$. Otherwise we just have
\[
\beta_\lambda(E) \le C(n,\lambda) \left( \hh^{n-1}(\partial^*E \cap \partial H)+ \hh^{n-1}(\partial^*B^\lambda(|B^\lambda|,0) \cap \partial H)\right)  \le C(n,\lambda,A) \frac{\delta_7^{\frac{1}{2n}}}{\delta_7^{\frac{1}{2n}}} \le C(n,\lambda,A) D_\lambda^{\frac{1}{2n}}.
\]

Next we observe that, letting $C_\lambda$ such that $P_\lambda(B^\lambda) \le C_\lambda \hh^{n-1}(\partial B^\lambda(|B^\lambda|) \cap \partial H)$, then for any set $E\subset\R^n\setminus H$ with $|E|=|B^\lambda|$ and $\hh^{n-1}(\partial^*E \cap \partial H) \ge \frac{2C_\lambda}{1-\lambda}\hh^{n-1}(\partial B^\lambda(|B^\lambda|) \cap \partial H)$ there holds
\begin{equation}\label{qqq2}
    \beta_\lambda(E) \le C_8 D_\lambda(E),
\end{equation}
for a constant $C_8=C_8(n,\lambda)>0$.\\
Indeed
\[
P_\lambda(E) -P_\lambda(B^\lambda)
\ge (1-\lambda) \hh^{n-1}(\partial^*E \cap \partial H)
- C_\lambda \hh^{n-1}(\partial B^\lambda(|B^\lambda|) \cap \partial H)
\ge \frac{1-\lambda}{2}\hh^{n-1}(\partial^*E \cap \partial H),
\]
and
\[
\beta_\lambda(E)  \le C(n,\lambda) \left( \hh^{n-1}(\partial^*E \cap \partial H)+ \hh^{n-1}(\partial^*B^\lambda(|B^\lambda|,0) \cap \partial H)\right) \le C(n,\lambda,C_\lambda) \hh^{n-1}(\partial^*E \cap \partial H).
\]

Setting now $A:= \frac{2C_\lambda}{1-\lambda}\hh^{n-1}(\partial B^\lambda(|B^\lambda|) \cap \partial H)$ in \eqref{qqq}, taking into account \eqref{qqq2} we conclude that for any set $E\subset\R^n\setminus H$ with $|E|=|B^\lambda|$ there holds
\[
\begin{split}
    \beta_\lambda(E) \le \max\{C_A, C_8\} \, \max\left\{ D_\lambda(E), D_\lambda(E)^{\frac{1}{2n}} \right\}.
\end{split}
\]
\end{proof}

\appendix

\section{Auxiliary results}
\label{sec:Appendix}

\subsection{Regularity of \texorpdfstring{$(\Lambda, r_0)$}{}-minimizers}

We recall definitions and basic properties of local $(\Lambda, r_0)$-minimizers of the perimeter. A detailed account on the theory of $(\Lambda, r_0)$-minimizers can be found in \cite{MaggiBook}.

\begin{definition}\label{def:LambdaMin}
Let $\Omega \subset \R^n$ be an open set and let $E \subset \R^n$ be a set of finite perimeter. We say that $E$ is a local $(\Lambda, r_0)$-minimizer of the perimeter in $\Omega$, with $\Lambda, r_0>0$, if
\begin{equation*}
P(E, B_r(x)) \le P(F, B_r(x)) + \Lambda |E \Delta F|, \end{equation*}
whenever $E \Delta F \subset\subset B_r(x) \subset\subset \Omega$ and $r \le r_0$.    
\end{definition}

It is well-known that local $(\Lambda, r_0)$-minimizers have bounded mean curvature in a generalized sense. We could not find an explicit reference in the literature, hence we provide a proof in the following result.

\begin{lemma}\label{prop:meancurvature}
Let $\Omega \subset \R^n$ be an open set and let $E \subset \R^n$ be a local $(\Lambda,r_0)$-minimizer of the perimeter in $\Omega$. 
Then there exists $H \in L^\infty(P(E, \cdot), \R^n)$ such that $\|H\|_{L^\infty} \le \Lambda$ and \begin{equation*} \int_{\partial^*E} {\rm div}_T X = - \int_{\partial^*E}\Braket{X,H} \qquad \forall X \in C_c^1(\Omega, \R^n), \end{equation*}
where ${\rm div}_T X$ is the tangential divergence of $X$ along the ($\hh^{n-1}$-a.e. defined) tangent space of $\partial^* E$.
We shall refer to $H$ as to the {\em (generalized) mean curvature} of $E$.
\end{lemma}

\begin{proof} Let $X \in C_c^1(B_r(x))$, with $B_r(x) \subset\subset \Omega$ and $r \le r_0$. Let $\{g_t\}_{t < |\eta|}$ be the flow of the vector field $X$, and define $F_t = g_t(E)$. Then $E \Delta F_t \subset\subset B_r(x)$. Let $f_t = g_t^{- 1} = g_{- t}$. 
If $u \in C^1(\Omega)$, for $t>0$ we have \begin{equation}\label{eq:zaq} \begin{split} \int_{B_r(x)} |u(f_t(y)) - u(y)|\de y & = \int_{B_r(x)}\left|\int_0^t\partial_s[u(f_s(y))]\de s\right| \de y
 \le \int_{B_r(x)}\int_0^t|\nabla u(f_s(y))|\; |X(f_s(y))|\de s \de y \\ & = \int_0^t \int_{B_r(x)} |\nabla u(f_s(y))|\;|X(f_s(y))| \frac{Jf_s(y)}{Jf_s(y)} \de y \de s \\ & \le (1 + o(1)) \int_0^t\int_{B_r(x)}|\nabla u(f_s(y))|\;|X(f_s(y))|\;Jf_s(y)\de y\de s \\ & = (1 + o(1)) \int_0^t\int_{B_r(x)}|\nabla u(z)|\;|X(z)|\de z \de s \\ & = (1 + o(1)) t\int_{B_r(x)} |\nabla u(z)|\;|X(z)|\de z,
\end{split} \end{equation}
where $o(1)\to0$ as $t\to0^+$.\\
Setting $u = u_\epsilon = \chi_E \star \rho_\epsilon$, then $|\nabla u_\epsilon| \mathcal{L}^n \to P(E, \cdot)$ as $\epsilon\to0$ by \cite[Proposition 12.20]{MaggiBook}. Also \begin{equation*} \int_{B_r(x)}\left|u_\epsilon(f_t(y)) - \chi_{f_{- t}(E)}(y)\right|\frac{Jf_t}{Jf_t}\de y \le 2 \int_{B_r(x)}|u_\epsilon(z) - \chi_E(z)|\de z.
\end{equation*}
Hence setting $u=u_\epsilon$ in \eqref{eq:zaq} and letting $\epsilon \to 0$ implies
\begin{equation*}
|E \Delta F_t| = 
\int_{B_r(x)}|\chi_{F_t} - \chi_E| \le (1 + o(1)) t \int_{B_r(x)}|X| \de P(E, \cdot). \end{equation*} 
By $(\Lambda, r_0)$-minimality we deduce
\begin{equation*}
P(E, B_r(x)) - P(F_t, B_r(x)) \le (1 + o(1)) \Lambda t \int_{B_r(x)} |X| \de P(E, \cdot).
\end{equation*}
Dividing by $t>0$ and letting $t\to 0^+$ we get
\begin{equation*}
-\int_{\partial^*E}{\rm div}_T X\de P(E, \cdot) \le \Lambda\int_{\partial^*E}|X| \de P(E, \cdot). \end{equation*}
Up to changing $X$ with $- X$ we obtain \begin{equation*} \left|\int_{\partial^*E}{\rm div}_T X \de P(E, \cdot)\right| \le \Lambda \int_{\partial^*E}|X|\de P(E, \cdot), \end{equation*}
that implies the existence of the generalized mean curvature $H \in L^\infty(P(E, \cdot)\mres B_r(x))$ for $E$ in $B_r(x)$ with $\|H\|_{L^\infty} \le \Lambda$. Since $B_r(x)$ was arbitrary in $\Omega$, by a partition of unity argument the claim follows. 
\end{proof}

Let us further recall the following fundamental regularity properties of local $(\Lambda, r_0)$-minimizers.

\begin{theorem}[{\cite{Tamanini}, \cite[Theorem 26.3, Theorem 26.6]{MaggiBook}}] \label{thm:RegolaritaLambdaMinimizers}
Let $\Omega \subset \R^n$ be an open set. Let $E \subset \Omega$ be a local $(\Lambda,r_0)$-minimizer in $\Omega$. Then the set $E^{(1)}$ of points of density $1$ for $E$ is an open representative for $E$. Moreover, representing $E$ with $E^{(1)}$, we have that $\partial^* E \cap \Omega$ is a $C^{1,\frac12}$ manifold and $\hh^d(\partial E \cap \Omega\setminus \partial^*E)=0$ for any $d > n-8$.\\
Let $E_i \subset \Omega$ be a sequence of local $(\Lambda,r_0)$-minimizers in $\Omega$ that converges to $E$ in $L^1(\Omega)$. If $\partial E \cap B_r(x)$ is of class $C^2$, for some $B_r(x)\subset \Omega$, then $\partial E_i \cap B_{r/2}(x)$ is of class $C^{1,\frac12}$ for large $i$ and converges to $\partial E \cap B_{r/2}(x)$ in $C^{1,\alpha}$ for any $\alpha \in (0,1/2)$.
\end{theorem}

Thanks to \cref{thm:RegolaritaLambdaMinimizers}, we will always identify a local $(\Lambda, r_0)$-minimizer $E$ with the open set $E^{(1)}$.

\subsection{Axially symmetric hypersurfaces}

We recall a formula for the mean curvature of axially symmetric hypersurfaces of class $W^{2,p}$.

\begin{lemma}\label{curvature:revolution}
Let $a<b$. Let $\alpha$, $\beta : (a, b) \to (0, \infty)$ be $W^{2, p}$ functions, with $p \in (1,  \infty]$, parametrizing the curve $\gamma : (a, b) \to {\rm span}\{e_1, e_n\} \subset \R^n$ given by $\gamma(t) = (\alpha(t), 0,\ldots,0, \beta(t))$, and assume that $|\gamma'(t)|=1$ and that $\inf_{(a,b)} \alpha>0$. Let $S$ be the axially symmetric hypersurface around the $n$-th axis parametrized by
\begin{equation*} \begin{split} & \phi_S : \mathbb{S}^{n - 2} \times (a, b) \to \R^n \\ & \phi_{S} (\theta, t) = (\alpha(t) \theta, \beta(t)). \end{split} \end{equation*}
Then the vector
\begin{equation*} H = \left( \Braket{k_\gamma, \nu} - (n - 2) \frac{\beta'}{\alpha}\right) \left(-\beta' \theta, \alpha' \right),
\end{equation*} 
for every $\theta$ and a.e. $t$, where $k_\gamma$ is the curvature of $\gamma$ and $\nu(t)= (-\beta',0,\ldots,0,\alpha')$, is the (generalized) mean curvature of $S$. More precisely
\begin{equation}\label{zz:zz}
    \int_S {\rm div}_T X = - \int_S \braket{X, H},
\end{equation}
for any $X \in C^1_c(\R^n,\R^n)$ such that ${\rm spt} X \cap \partial S = \emptyset$, where ${\rm div}_T X$ is the tangential divergence of $X$ along $S$.
\end{lemma}

\begin{proof}
If $\alpha, \beta$ are smooth, the claimed formula follows by a direct computation. The statement then follows by approximating $\alpha, \beta$ in $W^{2,p}$. The approximating hypersurfaces $S_i$ converge to $S$ in $C^1$ and the measures $H_i\hh^{n-1}\mres S_i$ converge to $H\hh^{n-1}\mres S$ in duality with compactly supported continuous fields. Hence \eqref{zz:zz} passes to the limit.
\end{proof}

\section*{Declarations}

\noindent\textbf{Acknowledgments.} The authors are grateful to Nicola Fusco for many suggestions and for stimulating discussions.

\medskip
\noindent\textbf{Funding.}  The authors are members of INdAM - GNAMPA. The second author is partially supported by the PRIN Project 2022E9CF89 - PNRR Italia Domani, funded by EU Program NextGenerationEU.

\medskip
\noindent\textbf{Data availability statement.} The present paper has no associated data.

\medskip
\noindent\textbf{Conflict of interest.} The authors declare no conflict of interest.

\medskip
\noindent\textbf{Ethics approval.} This research did not involve any studies with human participants or animals. Therefore, ethical approval was not required for this study.

\medskip
\noindent\textbf{Consent.} The authors agreed with the content of the paper and they gave consent to the present submission.

\printbibliography[title={References}]
\end{document}